\documentclass{amsart}
\usepackage{amssymb, amsmath}
\usepackage{mathrsfs}
\usepackage{amscd} 
\usepackage{bm}
\usepackage{hyperref}
\usepackage[normalem]{ulem}
\usepackage[T1]{fontenc}
\usepackage{lmodern}      
\usepackage[english]{babel}
\usepackage{microtype}               
\usepackage{graphicx}
\usepackage[curve]{xypic}
\usepackage[usenames,dvipsnames,x11names]{xcolor}
\usepackage{tikz}
\usepackage[all]{xy}
\usepackage{cancel}
\usepackage{tikz-cd}
\usetikzlibrary{matrix}
\usetikzlibrary{arrows,decorations,shapes,shadows}
\usepackage{verbatim}
\usepackage{stmaryrd}

\newbox\mybox
\def\overtag#1#2#3{\setbox\mybox\hbox{$#1$}\hbox to
  0pt{\vbox to 0pt{\vglue-#3\vglue-\ht\mybox\hbox to \wd\mybox
      {\hss$\ss#2$\hss}\vss}\hss}\box\mybox}
\def\undertag#1#2#3{\setbox\mybox\hbox{$#1$}\hbox to 0pt{\vbox to
    0pt{\vglue#3\vglue\ht\mybox\hbox to \wd\mybox
      {\hss$\ss#2$\hss}\vss}\hss}\box\mybox}
\def\lefttag#1#2#3{\hbox to 0pt{\vbox to 0pt{\vglue -6pt\hbox to
      0pt{\hss$\ss#2$\hskip#3}\vss}}#1}
\def\righttag#1#2#3{\hbox to 0pt{\vbox to 0pt{\vglue -6pt\hbox to
      0pt{\hskip#3$\ss#2$\hss}\vss}}#1}
\let\ss\scriptstyle
\def\splicediag#1#2{\xymatrix@R=#1pt@C=#2pt@M=0pt@W=0pt@H=0pt}
\def\Dot{\lower.2pc\hbox to 2pt{\hss$\bullet$\hss}}
\def\Circ{\lower.2pc\hbox to 2pt{\hss$\circ$\hss}}
\def\Vdots{\raise5pt\hbox{$\vdots$}}
\newcommand\lineto{\ar@{-}}
\newcommand\dashto{\ar@{--}}
\newcommand\dotto{\ar@{.}}

\usepackage{xcolor}

\newcommand{\fract}[2]{\hbox{\leavevmode
  \kern.1em \raise .25ex \hbox{\the\scriptfont0 $#1$}\kern-.1em }\big/
  {\hbox{\kern-.15em \lower .5ex \hbox{\the\scriptfont0 $#2$}} }}

\newcommand\Aut{{\operatorname{Aut}}}

\renewcommand{\setminus}{\smallsetminus}
\newcommand\Q{{\mathbb Q}}
\newcommand\R{{\mathbb R}}
\newcommand\C{{\mathbb C}}
\newcommand\Z{{\mathbb Z}}
\newcommand{\LL}{\mathbb L}
\newcommand\N{{\mathbb N}}
\newcommand\pa{{\mathfrak a}}

\newcommand\pb{{\mathfrak b}}

\newcommand\pc{{\mathfrak c}}

\newcommand\pq{{\mathfrak q}}
\newcommand\pp{{\mathfrak p}}
\newcommand{\Id}{\operatorname{Id}}

\renewcommand{\t}{\tau}

\newcommand{\unit}{\operatorname{unit}}
\newcommand{\Ker}{\operatorname{Ker}}
\renewcommand\O{{\mathcal O}}
\newcommand\lb{\llbracket}
\newcommand\rb{\rrbracket}
\newcommand{\x}{\bold{x}}
\newcommand{\oo}{\boldsymbol{\omega}}
\newcommand{\z}{\bold{z}}
\newcommand{\gam}{\bm{\gamma}}
\newcommand{\w}{\bold{w}}
\newcommand{\y}{\bold{y}}
\newcommand{\uu}{\bold{u}}
\renewcommand{\phi}{\varphi}
\renewcommand{\leq}{\leqslant}
\renewcommand{\geq}{\geqslant}

\newcommand{\p}{\mathfrak{p}}
\newcommand{\lgw}{\longrightarrow}
\newcommand{\lgm}{\longmapsto}
\newcommand{\s}{\sigma}
\newcommand{\sg}{\sigma}
\newcommand{\tg}{\sigma}
\newcommand{\lag}{\lambda}
\newcommand{\ovl}{\overline}
\newcommand{\wdh}{\widehat}
\newcommand{\PP}{\mathbb P}
\newcommand{\ini}{\operatorname{in}}
\newcommand{\wdt}{\widetilde}
\newcommand{\het}{\operatorname{ht}}
\newcommand{\rank}{\operatorname{rank}}
\newcommand{\la}{\lambda}
\renewcommand{\O}{\mathcal{O}}
\renewcommand{\k}{\Bbbk}
\renewcommand{\a}{\alpha}
\renewcommand{\b}{\beta}
\newcommand{\e}{\varepsilon}
\newcommand{\K}{\mathbb{K}}
\newcommand{\g}{\gamma}
\newcommand{\m}{\mathfrak m}
\newcommand{\ord}{\operatorname{ord}}
\renewcommand{\d}{\delta}
\newcommand{\Frac}{\operatorname{Frac}}
\newcommand{\D}{\Delta}

\newcommand{\Gr}{\operatorname{r}}
\newcommand{\Fr}{\operatorname{r}^{\mathcal{F}}}
\newcommand{\Ar}{\operatorname{r}^{\mathcal{A}}}

\newtheorem{theorem}{Theorem}[section]
\newtheorem{proposition}[theorem]{Proposition}
\newtheorem*{theorem*}{Theorem}
\newtheorem*{thmGR}{Gabrielov's rank Theorem}
\newtheorem*{thmZM}{Zariski's main Theorem}
\newtheorem*{thmWP}{Weierstrass preparation Theorem}
\newtheorem*{thmAY}{Abhyankar-Jung Theorem}

\newtheorem{corollary}[theorem]{Corollary}
\newtheorem{lemma}[theorem]{Lemma}

\theoremstyle{definition}

\newtheorem*{amalgamation*}{Amalgamation}
\newtheorem{example}[theorem]{Example}
\newtheorem{remark}[theorem]{Remark}
\newtheorem{problem*}[theorem]{Problem}
\newtheorem*{remark*}{Remark}
\newtheorem{definition}[theorem]{Definition}

\newcommand{\lcm}{\operatorname{lcm}}

\begin{document}
\title[Gabrielov's rank Theorem]{A proof of A. Gabrielov's rank Theorem}
\author[A.~Belotto da Silva]{Andr\'e Belotto da Silva}
\author[O.~Curmi]{Octave Curmi}
\author[G.~Rond]{Guillaume Rond}

\address[A.~Belotto da Silva, O.~Curmi, G.~Rond]{Universit\'e Aix-Marseille, Institut de Math\'ematiques de Marseille (UMR CNRS 7373), Centre de Math\'ematiques et Informatique, 39 rue F. Joliot Curie, 13013 Marseille, France}
\email[A.~Belotto da Silva]{andre-ricardo.belotto-da-silva@univ-amu.fr}
\email[O.~Curmi]{octave.curmi@univ-amu.fr}
\email[G.~Rond]{guillaume.rond@univ-amu.fr}
\thanks{The third author would like to thank J. M. Aroca, E. Bierstone, F. Cano, F. Castro, M. Hickel and  M. Spivakovsky for the fruitful discussions he had about this problem.}

\subjclass[2010]{Primary 13J05, 32B05; Secondary 12J10, 13A18, 13B35, 14B05, 14B20, 30C10, 32A22, 32S45}

\keywords{}

\begin{abstract}
This article contains a complete proof of Gabrielov's rank Theorem, a fundamental result in the study of analytic map germs. Inspired by the works of Gabrielov and Tougeron, we develop formal-geometric techniques which clarify the difficult parts of the original proof. These techniques are of independent interest, and we illustrate this by adding a new (very short) proof of the Abhyankar-Jung Theorem. We include, furthermore, new extensions of the rank Theorem (concerning the Zariski main Theorem and elimination theory) to commutative algebra.
\end{abstract}

\maketitle

\section{Introduction}

This article contains a complete and self-contained proof of Gabrielov's rank Theorem, a fundamental result in the study of analytic map germs. Let us briefly present its context and the theorem.

Let $\phi :(\K^n,0)\lgw (\K^m,0)$ be an analytic map germ of generic rank $r$ over the field $\K$ of real or complex numbers, that is, the image of $\phi$ is generically a submanifold of $\K^m$ of dimension $r$. When $\phi$ is algebraic, by a theorem of Chevalley \cite{Chevalley} (in the complex case) and Tarski \cite{Ta} (in the real case), the image of $\phi$ is a constructible set, that is, a set defined by polynomial equalities and inequalities. In particular the Zariski closure of the image has dimension $r$. If $\phi$ is complex analytic and proper, Remmert proved that the image of $\phi$ is always analytic \cite{Re}. In the case of an analytic map germ, however, the image is very far from being analytic. For instance, Osgood gave in \cite{Os} an example for which the dimension of the smallest germ of analytic set containing the image is greater than $r$ and, subsequently, Abhyankar generalized  this example in a systematic way, see \cite{Ab}. In this context Grothendieck, in \cite{Gro}, asked if the dimension of the smallest germ of analytic set containing the image (the \emph{analytic rank}) is equal to the dimension of the smallest germ of formal set containing the image (the \emph{formal rank}). Gabrielov answered negatively to this question \cite{Ga1}, and provided a sufficient condition for the answer to be positive \cite{Ga2}. Roughly speaking, the result is the following (see Theorems \ref{thm:MainSmooth} and \ref{thm:main} for a precise formulation):

\begin{thmGR}\label{thm:mainRough}
For a $\K$-analytic map $\phi:(\K^n,0)\lgw (\K^m,0)$, if the generic rank of $\phi$ equals its formal rank, then it also equals its analytic rank.
\end{thmGR}

Gabrielov's rank Theorem is a fundamental result because it provides a simple criteria for \emph{regular} maps, that is, maps whose three ranks coincide at every point of their sources. On the one hand, regular analytic maps constitute an important subclass of analytic maps, which share basic properties with (Nash) algebraic maps. For example, the images of regular proper real-analytic mappings form an interesting subclass of closed subanalytic sets, whose study goes back to works of Bierstone, Milman and Schwartz \cite{BMannals,BS}; see also \cite{BM87a,BM87b,Pawthesis,Paw,BMsubanalytic,ABM}. On the other hand, non-regular analytic maps are at the source of several pathological examples in complex and real-analytic geometry, e.g. \cite{Os,Ab,Ga1,Paw2,BP,BB}.

Nevertheless, the original proof of Gabrielov is considered very difficult, c.f. \cite[page 1]{IzDuke}. For example, in the $70$'s and $80$'s, several authors studied analytic map germs via more elementary techniques, avoiding Gabrielov's rank Theorem \cite{MT,EH,Mal2,BZ,CM,Iz2,IzDuke},  sometimes re-proving weaker versions of it. In the application to calculus of variations \cite{Tamm} Gabrielov's rank Theorem is cited but the author prefer adding further arguments in order to use a Frobenius type result of Malgrange \cite{Mal2} instead, and the situation is similar in an application to foliation theory \cite{CM}, c.f. \cite{CCD}. 
Moreover, some specialists believe that the proof contains ideas which would lead to the development of important new techniques concerning formal power series. In \cite{To2}, Tougeron proposed a new proof of Gabrielov's rank Theorem which, unfortunately, is still considered very difficult (and contains some unclear passages to us, which we point out in the body of the paper).

\bigskip

The first and main goal of this paper is to present a complete proof of Gabrielov's rank Theorem. We have been strongly influenced by the original papers of Gabrielov \cite{Ga2} and Tougeron \cite{To2}, but we do not fully understand either one of their proofs. We provide, therefore, several new arguments. The initial part of the proof, given in $\S$\ref{sec:GabrielovTheorem} and \ref{sec:Reduction}, follows closely the same strategy as the one of \cite{Ga2}, and we provided extra arguments whenever we felt it was necessary (see for example $\S\S$\ref{ssec:LocalBertini}). The second (and harder) part of the proof requires the development of several ideas and techniques inspired from \cite{Ga2,To2}, and our strategy deviates from theirs. It includes a new proof structure (c.f. the induction procedure given in Proposition \ref{cl:Main}), and the development of formal-geometric techniques such as projective rings, a Newton-Puiseux-Eisenstein Theorem, formal approximation of factors, in between others (see $\S$\ref{sec:OverviewLow} and \ref{sec:SemiGlobal}). Some of these techniques are of independent interest, and we illustrate this by including a new (very short) proof of the Abhyankar-Jung Theorem in $\S$\ref{section:AbhyankarJung}.

Then, we provide new extensions of Gabrielov's rank Theorem. The first extension, Theorem \ref{thm:equivalentGabrielov}(II) and Corollary \ref{strongly_inj2}, can be seen as an analogue of Zariski Main Theorem \cite{Za1, Za2} for morphisms of analytic algebras. The second, Theorem \ref{thm:equivalentGabrielov}(III), is a result concerning elimination theory of analytic equations. These extensions are expressed in a purely algebraic way, in contrast to Gabrielov's rank Theorem. Moreover they are  equivalent to Gabrielov's rank Theorem, in the sense that Gabrielov's rank Theorem can be easily deduced from any of these extensions. The third extension, Theorem \ref{thm:AplMacDonald}, characterizes polynomials with convergent power series coefficients in terms of the support of their solutions. We also include a discussion on strongly injective morphisms, see Theorem \ref{strong}, and on the relationship of Gabrielov's rank Theorem with the Weierstrass preparation Theorem, showing that Gabrielov's rank Theorem is a generalization of the Weierstrass preparation Theorem for convergent power series. Further details are given in $\S\S$\ref{ssec:Application}.

Given the history of the rank Theorem, we have made an extra effort to make the paper as self-contained as possible. We rely only on well-known results of commutative algebra, complex geometry and analysis which can either be found in books (e.g. resolution of singularities, Artin approximation) or admit simple proofs (e.g. Abhyankar-Moh reduction theorem \cite{A-M}). All other necessary results have been revisited.

\subsection{Gabrielov's rank Theorem}\label{ssec:MainTheorem}
Let $\K$ denote the field of real or complex numbers. We denote by $\x=(x_1,  \ldots, x_n)$ and $\uu=(u_1,  \ldots, u_m)$ two vectors of indeterminates. In general single indeterminates will be denoted by normal letters as $x$, $y$, $u$... and vectors of indeterminates will be denoted by bold letters as $\x$, $\y$, $\uu$... The ring of convergent (resp. formal) power series in $n$ indeterminates over $\K$ will be denoted by $\K\{\x\}$ (resp. $\K\lb \x\rb$). A \emph{morphism of convergent power series rings} is a ring morphism of the form
$$
\begin{array}{ccccc} \phi :& \K\{\x\}  & \lgw & \K\{\uu\}&\\
& f(\x)& \lgm & \varphi(f) &:= f(\phi_1(\uu),  \ldots, \phi_n(\uu))\end{array}
$$
where the $\phi_i(\uu)\in\K\{\uu\}$ for $i=1$, \ldots, $n$, do not depend on $f$. Note that $\phi$ induces an analytic map germ between smooth analytic space germs, that is:
$$
\begin{array}{ccccc} \phi^a :& (\K^m,0)  & \lgw & (\K^n,0)&\\
& \uu& \lgm & \phi(\uu) &:= (\phi_1(\uu),  \ldots, \phi_n(\uu))
\end{array}
$$
where $\phi^a$ is the geometrical counter-part of $\phi$. We are ready to provide a precise notion of ranks:

\begin{definition}[Ranks: the smooth case]\label{def:ranksSmooth}
Let $\phi : \K\{\x\}   \lgw  \K\{\uu\}$ be a morphism of convergent power series rings, and denote by $\widehat{\phi}: \K\lb \x\rb \lgw \K\lb \uu\rb$ its extension to the completion. We define
\[
\begin{aligned}
\text{the Generic rank:} & &\Gr(\phi) &:= \mbox{rank}_{\mbox{Frac}(\K\{\uu\})}(\mbox{Jac}(\phi^a)),\\
\text{the Formal rank:}& &\Fr(\phi) &:= \dim_{\K}\left(\frac{\K\lb \x\rb}{\Ker(\wdh{\phi})}\right)=n -\het(\Ker(\wdh \phi)),\\
\text{and the Analytic rank:}& &\Ar(\phi)&:=\dim_{\K}\left(\frac{\K\{ \x\}}{\Ker(\phi)}\right)
=n -\het(\Ker(\phi))
\end{aligned}
\]
of $\phi$, where $\mbox{Jac}(\phi^a)$ denotes the Jacobian matrix associated to $\varphi^a$, $\rank_{\Frac(\K\{\uu\})}(M)$ denotes the rank of the matrix $M$ over the field of fractions of $\K\{\uu\}$, and $\dim_{\K}(A)$ denotes the Krull dimension of the ring $A$.
\end{definition}

We can interpret geometrically the three ranks of $\phi$ via its geometrical counterpart $\phi^a :(\K^m,0)\lgw (\K^n,0)$ as follows:  for every sufficiently small open set $U \subset \K^{m}$ containing the origin, the image $\phi^{a}(U)$ is a subset of $\K^n$ which contains the origin. 
 Then $\Gr(\phi)$ is exactly this dimension of $(\phi^a(U),0)$ for $U$ small enough, $\Fr(\phi)$ is the dimension of the formal Zariski closure of $(\phi^a(U),0)$ in $(\K^n,0)$ for $U$ small enough, and $\Ar(\phi)$ is the dimension of the analytic Zariski closure of $(\phi^a(U),0)$ in $(\K^n,0)$ for $U$ small enough. This intuitively justifies the following well known inequality:

\begin{equation}\label{incr_ranks}
\Gr(\phi)\leq \Fr(\phi)\leq \Ar(\phi) \quad \text{(see e.g. \cite{IzDuke})}.
\end{equation}
Gabrielov's rank Theorem provides a simple criteria to show that all of the ranks are equal:

\begin{theorem}[Gabrielov's rank Theorem: the smooth case]\label{thm:MainSmooth}
For a morphism of convergent power series rings $\phi : \K\{\x\}   \lgw  \K\{\uu\}$: 
\[
\Gr(\phi)=\Fr(\phi) \implies \Fr(\phi)=\Ar(\phi).
\]
\end{theorem}

We are now interested in investigating singular spaces. Just as in \cite{IzDuke}, the essential case to consider is the complex-analytic one, and we specialize our study to $\K=\C$ (see Remark \ref{rk:Onreal}(2) below for a discussion on the real-analytic case). An \emph{analytic $\C$-algebra} $A$ is a local ring of the form $A=\C\{\x\}/I$ where $I$ is an ideal of $\C\{\x\}$. A \emph{morphism of $\C$-analytic algebras} is a morphism $\phi: A \lgw B$ where $A =\C\{\x\}/I$, $B=\C\{\uu\}/J$, and $\phi$ is induced by a morphism of convergent power series rings $\C\{\x\}\lgw \C\{\uu\}$. We denote by $\wdh \phi:\wdh A\lgw \wdh B$ the map induced by $\phi$ between the completions of $A$ and $B$. It is called the \emph{completion morphism of $\phi$}. 

Note that the definitions of the formal and analytic ranks of $\phi$ easily extend to a morphism of $\C$-analytic algebras $\phi: A \lgw B$. The generic rank, nevertheless, does not extend in a trivial way because we can not define the Jacobian in the singular context. In order to define the generic rank, note that a morphism of reduced $\C$-analytic algebras $\phi: A \lgw B$ also induces a morphism between (not necessarily smooth) analytic space germs $\phi^a :(Y,\pb) \lgw (X,\pa)$ so that $A=\mathcal{O}_{X,\pa}$ and $B=\mathcal{O}_{Y,\pb}$, where $\mathcal{O}_{X,\pa}$ and $\mathcal{O}_{Y,\pb}$ denote the local rings of analytic function germs at $\pa$ and $\pb$ respectively (more precisely, if $A=\C\{\x\}/I$ where $I$ is a radical ideal, we denote by $\pa$ the origin of $\C^n$ and $(X,\pa)$ is the germ of the analytic set defined by the vanishing of generators of $I$ in a neighborhood of $\pa$). We suppose that $B$ is an integral domain, which is equivalent to $(Y,\pb)$ being irreducible and reduced. We define the \emph{generic rank} of $\phi^a$ as (see \cite{IzDuke}):
$$\Gr(\phi^a):=\inf\left\{\sup\left\{\rank(\phi^a_{|M})\mid M \text{ $\C$-submanifold of }U\right\}\mid U\text{ neigh. of } \pb \text{ in } Y \right\}.$$
It coincides with the generic rank (given by definition \ref{def:ranksSmooth}) of $\phi^a$ restricted to $Y\setminus \mbox{Sing}(Y)$. This is well defined since $Y\setminus \mbox{Sing}(Y)$ is dense in $Y$.

\begin{definition}[Ranks: the general case]\label{def:ranks}
Let $\phi : A   \lgw  B$ be a morphism of {reduced} $\C$-analytic algebras where $B$ is an integral domain, and denote by $\widehat{\phi}: \widehat{A} \lgw \widehat{B}$ its extension to the completion. We define
\[
\begin{aligned}
\text{the Generic rank:} & &\Gr(\phi) &:= \Gr(\phi^a),\\
\text{the Formal rank:}& &\Fr(\phi) &:= \dim\left(\frac{\wdh{A}}{\Ker(\wdh{\phi})}\right)
=\dim(\wdh A)-\het(\Ker(\wdh \phi)),\\
\text{and the Analytic rank:}& &\Ar(\phi)&:=\dim\left(\frac{A}{\Ker(\phi)}\right)
=\dim(A)-\het(\Ker(\phi))
\end{aligned}
\]
of $\phi$. We recall that $\dim(A)=\dim(\wdh A)$ when $A$ is a Noetherian local ring {(cf \cite[Theorem 13.9]{Mat} for example)}.
\end{definition}

We note that the inequalities \eqref{incr_ranks} are again valid in this context. We are ready to formulate the general version of Gabrielov's rank Theorem:

\begin{theorem}[Gabrielov's rank Theorem]\label{thm:main}
Let $\phi: A \lgw B$ be a $\C$-analytic morphism, where $B$ is an integral domain. 
\[
\Gr(\phi)=\Fr(\phi) \implies \Fr(\phi)=\Ar(\phi).
\]
\end{theorem}

\begin{remark}[On the hypothesis of Theorem \ref{thm:main}]\label{rk:Onreal} \hfill
\begin{enumerate}
\item (The complex-analytic case) The Theorem holds true if $B$ is reduced (that is, free of nilpotent elements) instead of an integral domain. This can be deduced from Theorem \ref{thm:main} (see, e.g. \cite[Proposition 5.6]{Ga2} {or \cite[1.3 and 1.4]{Iz2}}).
\item {The proof of Theorem \ref{thm:main} reduces to Theorem \ref{thm:MainSmooth} by resolution of singularities in the range and by the normalization theorem in the source (see Lemma \ref{lem_1st_red}).}
\item (The real-analytic case) Given a real-analytic morphism $\phi: A \lgw B$ {where $A$ and $B$ are reduced}, there exists a well-defined complexification $\phi^{\C}: A^{\C} \lgw B^{\C}$ {where $A^\C= A \otimes \mathbb{C}$ and $B^\C = B \otimes \mathbb{C}$} {are reduced by \cite[Lemma 5.1]{Iz2}}. Following  \cite[$\S$~5]{Iz2} or \cite[$\S$~1]{IzDuke}, we can define the ranks of $\phi$ as the ranks of $\phi^{\C}$. It is now straightforward to prove that Gabrielov's rank Theorem holds whenever {$B$ is reduced}. 

This statement, nevertheless, hides a subtitle point in working in the real case which can be illustrated via the integral domain $B = \mathbb{R}\{x_1,x_2\}/(x_1^2+x_2^2)$. {Indeed,} if we denote by $(Y,\pb)$ the (real) geometrical counterpart of $B$, note that $\mbox{Sing}(Y) = Y$. It follows that the generic rank, as defined above, is intrinsically complex and does not coincide with the generic dimension of the image $Z=\phi^a(Y)$. {Here, we may consider the extra hypothesis} that $B$ is a real-closed integral domain (that is, $B = \R\{\x\}/I$ where $I$ is a real-closed prime ideal). This condition guarantees that the generic rank coincides with the generic dimension of the image $Z=\phi^a(Y)$; in particular $\mbox{Sing}(Y) \neq Y$.

\end{enumerate}
\end{remark}

In the rest of the paper, we focus on the essential case $\K=\C$.

\subsection{Applications and variations}\label{ssec:Application}

\subsubsection{Strongly injective morphisms}
The problem raised by Grothendieck has been generalized to the following problem:
given a morphism of $\C$-analytic algebras $\phi:A\lgw B$, when does $
\wdh{\phi}(\wdh{A})\cap B=\phi(A)$ hold true? If the equality is verified, we say that $\phi$ is \emph{strongly injective}. This terminology was introduced by Abhyankar and van der Put \cite{AP} who were the first ones to investigate this question. In particular they proved that $\phi$ is always strongly injective  when $A$ is a ring of convergent power series in two variables over any valued field. 

Without this assumption on the dimension, the equality $
\wdh{\phi}(\wdh{A})\cap B=\phi(A)$ does not hold in general (see Example \ref{ex:Osgood} and \eqref{gab_ex1}). In this work, we provide a simple proof of the following characterization:

 \begin{theorem}[{\cite[Theorem 5.5]{Ga2},\cite[Theorem 1]{IzDuke} c.f. \cite{EH}}]\label{strong}
 Let $\phi:A\lgw B$ be a morphism of analytic $\C$-algebras where $B$ is an integral domain. Then
 $$
 \Gr(\phi)=\Fr(\phi)= \Ar(\phi)\Longleftrightarrow \phi \text{ is strongly injective}.
 $$
 \end{theorem}

The direct implication of this Theorem has first been proven by Gabrielov \cite[Theorem 5.5]{Ga2}. Eakin and Harris \cite{EH} also gave a proof of this implication (avoiding Gabrielov's Theorem \ref{thm:main}), in the case where $A$ and $B$ are rings of convergent power series. They also proved the reverse implication in the same situation. Finally Izumi \cite{IzDuke} gave a proof of the equivalence (avoiding Gabrielov Theorem \ref{thm:main}) in the general case. In $\S\S$\ref{ssec:ApliStrong} we present a proof of this result, relying on Theorem \ref{thm:main} and \cite{EH}.

\subsubsection{Variations of Gabrielov's rank Theorem}

We prove that Gabrielov's rank Theorem admits three alternative formulations, which are of independent interest:

\begin{theorem}[Variations of Gabrielov's rank Theorem]\label{thm:equivalentGabrielov}
The following statements hold true:
\begin{enumerate}
\item[(I)] Let $\phi: A \lgw B$ be a $\C$-analytic morphism, where $B$ is an integral domain. 
\[
\Gr(\phi)=\Fr(\phi) \implies \Fr(\phi)=\Ar(\phi).
\]
\item[(II)] Let $\phi:A\lgw B$ be a strongly injective morphism of analytic $\C$-algebras where $B$ is an integral domain. If $f\in B$ is integral  over $\wdh{A}$ then $f$ is integral over $A$.
\item[(III)] Let $f\in \C\{\x,t\}$, where $t$ is a single indeterminate. Assume that there is a non-zero $g\in\C\lb \x,t\rb$, such that $fg\in\C\lb \x\rb[t]$. Then there is a non-zero $h\in\C\{\x,t\}$ such that $fh\in\C\{\x\}[t]$.
\item[(IV)] Let $f\in \C\{\x,z\}$ where $z$ is a single indeterminate and $n\geq 2$. Set
\[
A:=\C \lb \x\rb \quad \text{and} \quad \displaystyle B:=\frac{\C \lb \x,z \rb}{(x_1-x_2z)}.
\]
If the image of $f$ in $B$ is integral over $A$, then $f$ is integral over $\C \{\x\}$.
\end{enumerate}
\end{theorem}

The proof of the above result is given in $\S\S$\ref{ssec:AplVariation}. Note that we actually prove that $(I) \implies (II) \implies (III) \implies (IV) \implies (I)$. The Theorem then immediately follows because $(I)$ is Gabrielov's rank Theorem \ref{thm:main}.

One striking feature of Theorem \ref{thm:equivalentGabrielov} is that statements $(II,III,IV)$ are intrinsically algebraic. This contrasts with the statement of Gabrielov's rank Theorem, which depends on the generic rank (a geometrical condition). It seems important to clarify this relationship. We believe that the following open problem would be helpful in this investigation:

\begin{problem*}\label{problem}
Consider a family of local rings $(\mathcal{A}_n)_{n \in \mathbb{N}}  $, where $\mathcal{A}_n$ is a subring of $\mathbb{K} \lb x_1,\ldots,x_n\rb$. It is natural to ask:
\begin{enumerate}
\item Under which hypothesis over $(\mathcal{A}_n)_{n \in \mathbb{N}}  $ does Gabrielov's rank Theorem hold {for morphisms $A\lgw B$ where $A$ and $B$ are reduced quotients of rings $\mathcal A_n$}?
\item Under which hypothesis over $(\mathcal{A}_n)_{n \in \mathbb{N}}  $ are all the four statements in Theorem \ref{thm:equivalentGabrielov} equivalent {where $A$ and $B$ are reduced quotients of rings $\mathcal A_n$}?
\end{enumerate}
\end{problem*}

Note that the problem is also well-posed when $\mathbb{K}$ is a field of positive characteristic (see \cite{Ro1} for the generalization of the geometric rank to fields of positive characteristic). Furthermore, if we consider a morphism $\phi:A\lgw B$ where $A$ and $B$ are quotients of convergent power series rings by ideals generated by algebraic power series, and if the components of $\phi$ are algebraic power series, then we always have $\Gr(\phi)=\Fr(\phi)=\Ar(\phi)$ (see \cite[Theorem 6.7]{Ro1} for the general case, and \cite{To1,Be,Mi} for partial cases).

We finish this paragraph by pointing out that the statement of Theorem \ref{thm:equivalentGabrielov}$(II)$ above can be refined in the following way:

\begin{corollary}\label{strongly_inj2}
Let $\phi:A\lgw B$ be a morphism of analytic $\C$-algebras where $B$ is an integral domain. Let us assume that $\phi$ is strongly injective. If $f\in B$ is algebraic over $\wdh{A}$ then $f$ is algebraic over $A$.
\end{corollary}

The proof of this result is given in $\S\S$\ref{ssec:AplStronglyInj2}.

\subsubsection{Connection with Zariski Main Theorem} We now turn our attention to Zariski Main Theorem, a classical result in algebraic geometry {(we recall that a ring is essentially finitely generated over a field $\k$ if it is the localization of a finitely generated extension of $\k$)}:

\begin{thmZM}[{\cite{Za1,Za2}}]
Let $A$ be a reduced local ring that is essentially finitely generated over a field $\k$.  Let $\ovl A$ denote the integral closure of $A$ in $\Frac(A)$ (that is, the integral closure with respect to $A \lgw  \Frac(A)$). Then the integral closure of $\wdh A$ in $\Frac(\wdh A)$ coincides with the completion of $\ovl A$.
\end{thmZM}

Note that Theorem \ref{thm:equivalentGabrielov}(II) and Corollary \ref{strongly_inj2} can be seen as generalizations of the above result, where we replace the morphism $A \lgw  \Frac(A)$ with a strongly injective morphism $\varphi: A \lgw  B$.


\subsubsection{Connection with elimination theory and completion}
In commutative algebra and in algebraic geometry, elimination theory is the study of elimination of variables between polynomials. This is the main step in the resolution of polynomial equations. For example, in the case of linear equations, elimination theory reduces to Gaussian elimination. In general, the main tools in elimination theory are the resultant and the Gr\"obner basis. Note that, unfortunately, there is no analogue of the resultant for power series, and the analogue of Gr\"obner basis, the standard basis, is not as powerful for the objectives of elimination theory.

The general situation is the following: Let $\x$ and $\y$ be two vectors of indeterminates and $I$ an ideal of $\C\{\x,\y\}$. The problem is to determine $I\cap\C\{\x\}$. Note that, unlike in the polynomial case, we may have $I\cap\C\{\x\}=(0)$ even if $\het(I)$ is larger than the number of indeterminates $y_i$ \cite{Os}. By Remark \ref{rk:GabrielovExCompletionEliminationBad} below, we may even have $I\cap\C\{\x\}=(0)$ while $I\C\lb \x,\y\rb\cap\C\lb \x\rb\neq (0)$. Therefore, an interesting question is to determine under which hypothesis $I\C\lb \x,\y\rb\cap\C\lb \x\rb$ is generated by $I\cap\C\{\x\}$.

This question has been investigated for the first time in \cite{CPR} where it is related to several other properties. 

In this context, given $f \in \C\{\x,y\}$ where $y$ is a single variable, it is important to understand under which conditions we may assume that $f \in \C\{\x\}[y]$, up to multiplication by a convergent unit. Such a result would allow us to adapt arguments from elimination theory to the more general context of convergent power series. From this perspective, Theorem \ref{thm:equivalentGabrielov}(III) provides a formal characterization of the above condition.


\subsubsection{Connection with the Weierstrass preparation Theorem}
The Weierstrass preparation Theorem is a very powerful tool in algebraic and analytic geometry. In this subsection, we show how Gabrielov's rank Theorem can also be seen as an extension of the Weierstrass preparation Theorem for rings of convergent power series. Recall that the usual form of the Weierstrass Theorem is the following one:

\begin{theorem*}[Weierstrass preparation Theorem: usual formulation]
Let $f$ be a formal (resp. convergent) power series in the indeterminates $x_1$,   \ldots, $x_n$ over $\C$. Assume that $f$ is $x_n$-regular of order $d$, that is, $f(0,  \ldots, 0,x_n)=x_n^d\times \unit(x_n)$. Then there exist unique formal (resp. convergent) power series $a_1$,   \ldots, $a_d$ in the indeterminate $\x':=(x_1, \ldots, x_{n-1})$ such that
$$f(\x)=\left(x_n^d+a_1(\x')x_n^{d-1}+\cdots+ a_d(\x')\right)\times \unit(\x).$$
\end{theorem*}

Another classical form of the Weierstrass preparation Theorem is the following one (see \cite{Mal} for instance):

\begin{thmWP}
Let $A\lgw B$ be a morphism of analytic (resp. complete) $\C$-algebras. Let $\m$ be the maximal ideal of $A$. Then $B$ is finite over $A$ if and only if $B/\m B$ is finite over $A/\m=\C$.
\end{thmWP}

As a direct Corollary, we obtain  the following case of Theorem \ref{thm:main}:

\begin{corollary}[Gabrielov's rank Theorem for finite morphisms]\label{fin_inj}
Let $\phi : A\lgw B$ be an injective and finite morphism of analytic $\C$-algebras where $B$ is an integral domain. Then $\wdh \phi:\wdh A\lgw \wdh B$ is injective and finite.
\end{corollary}

\begin{proof}
Let $\m$ (resp. $\wdh \m$) be the maximal ideal of $A$ (resp. $\wdh A$). We have $\m\wdh A=\wdh \m$. Thus, if $\phi:A\lgw B$ is finite, then $A/\m\lgw B/\m B$ is finite by the Weierstrass preparation Theorem. But $A/\m=\wdh A/\wdh\m$ and $B/\m B=\wdh B/\wdh \m B$. Hence $\wdh \phi:\wdh A\lgw \wdh B$ is again finite, by the Weierstrass preparation Theorem applied to $\wdh\phi$.

Now, since $\phi$ is finite, we have $\dim(A)=\dim(B)$ by \cite[Theorems 9.3.ii, 9.4.ii]{Mat}. Hence $\dim(\wdh A)=\dim(A)=\dim(B)=\dim(\wdh B)$. But, since $\wdh A\lgw \wdh B$ is finite, the induced morphism $\wdh A/\Ker(\wdh\phi)\lgw \wdh B$ is also finite, thus $\dim (\wdh A/\Ker(\wdh\phi))=\dim(\wdh B)$. Therefore $\dim(\wdh A/\Ker(\wdh \phi))=\dim(\wdh A)$ and $\Ker(\wdh\phi)$ is a height 0 prime ideal. But, since $\phi$ is injective and $B$ is a domain, $A$ is a domain, and $\wdh A$ is also a domain (c.f. \cite[Proposition 4.1]{Ro} for example). This proves that $\Ker(\wdh\phi)=(0)$ and $\wdh\phi$ is injective.
\end{proof}

\begin{remark}[On the connection with Problem \ref{problem}]\label{rk:weierstrass}
We claim that the Weierstrass preparation Theorem is a necessary condition for Gabrielov's rank Theorem to hold in a family of real or complex rings $(\mathcal{A}_n)_{n \in \mathbb{N}}$, as asked in Problem \ref{problem}. Indeed, let $\phi:A\lgw B$ be an injective morphism between rings that are quotients of rings $\mathcal A_n$, and assume that $A/\m\lgw B/\m B$ is finite. By the Weierstrass preparation Theorem for complete local algebras, we have that $\wdh \phi:\wdh A\lgw \wdh B$ is finite. In particular any element $f\in B$ is integral over $\wdh A$. Therefore, if Theorem \ref{thm:equivalentGabrielov} (II) is satisfied for the family $(\mathcal A_n)_n$, we necessarily have that $f$ is integral over $A$. Therefore if $B=\mathcal A_n/I$ for some $n$ and some ideal $I$ of $\mathcal A_n$, and $\mathcal A_n$ is a subring of $\K\lb x_1,\ldots, x_n\rb$ as in Problem \ref{problem}, we have that the $x_i$ are integral over $A$, therefore $B$ is integral over $A$.
\end{remark}

\subsubsection{Convergent power series with support in strongly convex cones}

In general roots of {monic} polynomials with coefficients in $\C\lb \x\rb$ can be represented as Laurent Puiseux series with support in  a rational strongly convex cone by a Theorem of MacDonald \cite{MD} {(and its generalization due to Gonz\'alez Perez \cite{Gonz})}. We will reformulate Gabrielov's rank Theorem in this setting. Before we need to give some definitions.

\begin{definition}
Let $\s$ be a strongly convex rational cone containing ${\R_{\geq 0}}^n$. This means that $\s$ has the form
$$
\s=\{u\in\R^n\mid \exists \la_1,  \ldots, \la_k\in\R_{\geq 0},\ u=\la_1 v_1+\cdots+\la_k v_k\}
$$
where $v_1$,   \ldots, $v_k$ are given vectors with integer coordinates and $\s$ does not contain any non-trivial linear subspace of $\R^n$.
\end{definition}

For such a cone we denote by $\C\lb \s\rb$ the set of formal power series with support in $\s\cap \Z^n$, that is:
$$
\C\lb\s\rb:=\left\{f=\sum_{\a\in\s\cap\Z^n}f_\a \x^\a \mid f_\a\in\C\right\}.
$$
More generally, if $d\in\N^*$, we denote by $\C\lb \s\cap\frac{1}{d}\Z^n\rb$ the set of formal power series with support in $\s\cap\frac{1}{d}\Z^n$. Let us mention the following result:

\begin{theorem}[{MacDonald Theorem \cite{MD} \cite{Gonz}}]\label{macdonald}
Let $P(z)\in\C\lb \x\rb[z]$ be a monic polynomial. Then there exists a strongly convex rational cone $\s$ containing $\R_{\geq 0}^n$ and a positive integer $d$ such that the roots of $P(z)$ are in
$\C\lb\s\cap\frac{1}{d}\Z^n\rb$. 
\end{theorem}

Since $\s$ is a convex rational cone, there exists indeterminates $u_1$,   \ldots, $u_s$ and a binomial ideal $I$ of $\C[\uu]$ such that $\C\lb \s\rb\simeq \C\lb \uu\rb/I$. Therefore we define the analogue of the ring of convergent power series $\C\{\s\}$ as the subring of $\C\lb \s\rb$ which is isomorphic to $\C\{\uu\}/I$. 

{\begin{remark}
The original result of MacDonald is given for a (non necessarily monic) polynomial $P$ with coefficients in $\C[\x]$. In this case, the polynomial not being monic, the supports of the roots of $P$ are in the translation of a strongly convex rational cone $\s$. \\
Theorem \ref{macdonald} has been proved in \cite{Gonz} in the convergent case. The general case is a direct consequence of \cite[Theorem 6.2]{P-R}: for a given $P(z)\in\C\lb \x\rb[z]$  monic, we consider a vertex $\a\in\N^n$ of the Newton polyhedra $NP(\Delta_P)$ of the discriminant $\Delta_P$ of $P$. There exists a strongly convex rational cone $\s$ such that $NP(\D_P)\subset \a+\s$. This means that $\D_P=\x^\a U(\x)$ where the support of $U(\x)$ is in $\s$ and $U(0)\neq 0$. Then, by \cite[Theorem 6.2]{P-R}, the roots of $P$ are Puiseux series with support in $\s$, that is, elements of $\C\lb\s\cap\frac{1}{d}\Z^n\rb$ for some $d\in\N^\ast$.
\end{remark}}

Theorem \ref{thm:equivalentGabrielov}(II) has the following corollary about the Galois group of a polynomial with formal power series coefficients:

\begin{theorem}\label{thm:AplMacDonald}
Let $P(z)\in\C\lb \x\rb[z]$ be a monic irreducible polynomial such that the roots of $P(z)$ are in  $\C\lb \s\cap\frac{1}{d}\Z^n\rb$, where $\s$ is a strongly convex rational cone containing $\R_{\geq 0}^n$ and $d$ is a positive integer. If one of the roots of $P$ is in $\C\{ \s\cap\frac{1}{d}\Z^n\}$ then  the coefficients of $P(z)$ are in $\C\{\x\}$.
\end{theorem}

The proof of this result is given in $\S\S$ \ref{ssec:AplMacDonald}.

\subsection{Examples}
In this section, we recall the classical examples of Osgood \cite{Os} and Gabrielov \cite{Ga1}.

\begin{example}[{Osgood's example \cite{Os}}]\label{ex:Osgood}
Osgood showed the existence of a morphism $\phi:\C\{x_1,x_2,x_3\}\lgw \C\{u,v\}$ such that 
\begin{equation}\label{gab_ex2}
\Gr(\phi)=2,\ \Fr(\phi)=3,\ \Ar(\phi)=3.
\end{equation}
Indeed, consider the following morphism:
$$
\phi(x_1)=u,\ \phi(x_2)=uv, \ \phi(x_3)=uve^v.
$$
We denote by $\wdh\phi:\C\lb x_1,x_2,x_3\rb\lgw \C\lb u,v\rb$ the morphism induced by $\phi$. Given $f\in \Ker(\wdh\phi)$, let us write $f=\sum_{d\in\N} f_d(\x)$ where the $f_d(\x)$ are homogeneous polynomials of degree $d$ (when they are not zero), so that:
$$
0=\wdh\phi(f)=\sum_{d\in\N} f_d(u,uv,uve^v)=\sum_{d\in\N}u^df(1,v,ve^v)
$$
Therefore $f_d(1,v,e^v)=0$ for every $d$, hence $f_d=0$ for every $d$ since $v$ and $ve^v$ are algebraically independent over $\C$. It follows that $\Fr(\phi)=\Ar(\phi)=3$, while we can easily check that $\Gr(\phi)=2$. In particular the map $\phi^a:(\C^2,0)\lgw (\C^3,0)$ defined by $\phi^a(u,v)=(u,uv,uve^v)$ sends a neighborhood of the origin onto a subset $Z$ of $\C^3$ that is generically a complex manifold of dimension 2, but whose analytic or formal Zariski closure is $\C^3$.
\end{example}

\begin{example}[{Gabrielov's example \cite{Ga1}}]\label{ex:Gabrielov}
Gabrielov extended Osgood's example, and provided a morphism $\psi: \C\{x_1,x_2,x_3,x_4\} \lgw \C\{u,v\}$ such that 
\begin{equation}\label{gab_ex3}
\Gr(\psi)=2,\ \Fr(\psi)=3,\ \Ar(\psi)=4.
\end{equation}
which is built up from the observation that Osgood's example $\phi$ is not well-behaved in terms of elimination theory, that is:
\begin{equation}\label{gab_ex1}
\phi(\C\{\x\})\subsetneq\wdh{\phi}(\C\lb \x\rb)\cap \C\{\uu\},
\end{equation}
Indeed, we follow the heuristic that, even if $x_3-x_2e^{x_2/x_1}$ is not a power series, its image under $\phi$ should be $0$. Let us consider a polynomial truncation of its formal power series:
$$
f_n:=\left(x_3-x_2\sum_{i=0}^{n}\frac{1}{i!}\frac{x_2^i}{x_1^i}\right)x_1^n\in\C[x_1,\,x_2,\,x_3],\ \forall n\in\N.
$$
and note that
$$
\phi(f_n)=u^{n+1}v\sum_{i=n+1}^{+\infty}\frac{v^i}{i!},\ \forall n\in\N.
$$
It follows that $(n+1)!\phi(f_n)$ is a convergent power series whose coefficients have module less than 1. Moreover when the coefficient of $u^kv^l$ in the expansion of $\phi(f_n)$ is nonzero, we have $k=n+1$. This means that the supports of $\phi(f_n)$ and $\phi(f_m)$ are disjoint whenever $n\neq m$. Therefore the power series 
$$
h:=\sum_{n \in \mathbb{N}}(n+1)!\phi(f_n)
$$
is  convergent since each of its coefficients has module less than 1. But $\wdh{\phi}$ being injective, the unique element whose image is $h$ is necessarily: 
$$
\wdh{g}\colon=\sum_{n \in \mathbb{N}}(n+1)!f_n=\left(\sum_{n \in \mathbb{N}}(n+1)!\cdot x_1^n\right)x_3+\wdh{f}(x_1,\,x_2),
$$
Now, $\wdh{g}$ is a divergent power series  and $\wdh{\phi}(\wdh{g}(\x))=h(u,v)\in\C\{u,v\}$. This shows that \eqref{gab_ex1} holds.

Finally, consider the morphism $\psi:\C\{x_1,x_2,x_3,x_4\}\lgw \C\{u,v\}$ defined by 
$$
\psi(x_1)=u,\ \psi(x_2)=uv, \ \psi(x_3)=uve^v,\ \psi(x_4)=h(u,v).
$$
By the above considerations, we see that $x_4-\wdh g(\x)$ belongs to the kernel of $\wdh\psi$. In fact one can show that $\Ker(\wdh\psi)=(x_4-\wdh g(\x))$, while $\Ker(\psi)=(0)$.
\end{example}

\begin{remark}\label{rk:GabrielovExCompletionEliminationBad}
Note that Gabrielov's example illustrates a case where the completion operation does not commute with the elimination of indeterminates. Indeed, since $\Ker(\wdh\psi)\neq (0)$, there exist $\wdh k_1$,   \ldots, $\wdh k_4\in\C\lb \x,u,v\rb$ such that
 $$ (x_1-u)\wdh k_1+(x_2-uv)\wdh k_2+(x_3-uve^v)\wdh k_3+(x_4-h(u,v))\wdh k_4\in \C\lb \x\rb\setminus\{0\}.$$
 This means that $I\C\lb \x,u,v\rb\cap \C\lb \x\rb\neq (0)$ where $I$ denotes the ideal of $\C\{\x,u,v\}$ generated by 
 $$x_1-u,\ x_2-uv,\ x_3-uve^v,\ x_4-h(u,v).$$
 On the other hand, since $\Ker(\psi)=(0)$, we see in a similar way that $I\cap\C\{\x\}=(0)$, as claimed.
\end{remark}

\begin{remark}[Pathological real-analytic examples]
Variations of Osgood example have been used to provide the following list of pathological examples:
\begin{itemize}
\item In \cite{Paw2}, Paw\l ucki provides an example of a subanalytic set (given by a non regular morphism) which is neither formally nor analytically semi-coherent. In particular, this contradicted a result previously announced by Hironaka \cite{Hir3}.
\item In \cite{BP}, Bierstone and Parusi\'nski show the existence of a proper real-analytic (non regular) mapping which can not be transformed into a mapping with locally equidimensional fibers by global blowing ups (contrasting with the complex case where the result holds true, as proved by Hironaka \cite{Hir2}).
\item In \cite{BB}, the first author and Bierstone show the existence of a proper real-analytic (non-regular) mapping which can not be monomialized via global blowing ups in the source and target.
\end{itemize} 
\end{remark}



\section{Ranks and transformations}\label{sec:GabrielovTheorem}


\subsection{General properties}
We follow the notations introduced in $\S\S$ \ref{ssec:MainTheorem}. We start by stating basic properties of the ranks introduced in Definition \ref{def:ranks}.

\begin{proposition}[Basic properties]\label{rk:BasicPropertiesGabrielov1}
Let $\phi:A\lgw B$ be a morphism of  reduced $\C$-analytic algebras.
\begin{enumerate}
\item We have $\Gr(\phi)\leq \Fr(\phi)\leq \Ar(\phi).$
\item  If $\Ar(\phi)=\dim(A)$ (resp. $\Fr(\phi)=\dim(A)$), $\phi$ is injective (resp. $\wdh\phi$ is injective).
\item Assume that $B$ is an integral domain. Then
$$\Fr(\phi)=\Ar(\phi)\Longleftrightarrow  \Ker(\wdh \phi)=\Ker(\phi)\wdh A.$$
\end{enumerate}
\end{proposition}
\begin{proof}
A rigorous proof of (1) is given in Lemma (1.5) \cite{IzDuke}.

Now assume that $\Ar(\phi)=\dim(A)$ and $A$ is reduced. This means that $\Ker(\phi)$ is an ideal of height $0$. Since $A$ has no non trivial nilpotents, $\Ker(\phi)=(0)$. The same proof works in the same way when $\Fr(\phi)=\dim(A)$. Indeed, by Artin approximation Theorem, $\wdh A$ is reduced when $A$ is (see e.g. \cite[Proposition 4.1]{Ro}), and $\dim(\wdh A)=\dim(A)$ (see \cite[Theorem 13.9]{Mat} for example). This proves (2).

For (3), let us remark that $\Ker(\phi)\wdh A\subset \Ker(\wdh\phi)$. If $B$ is an integral domain, $\wdh B$ is also an integral domain by Artin approximation Theorem, therefore $\Ker(\phi)$ and $\Ker(\wdh\phi)$ are prime ideals. By Artin approximation Theorem, $\Ker(\phi)\wdh A$ is a prime ideal of $\wdh A$ of the same height as $\Ker(\phi)$. Therefore we have
$$ \Ker(\wdh \phi)=\Ker(\phi)\wdh A\Longleftrightarrow \het(\Ker(\wdh\phi))=\het(\Ker(\phi)).$$
This proves (3).
\end{proof}

It is straightforward that the three ranks are invariant under isomorphisms. They are also invariant under some more general transformations, as shown in the following proposition:

\begin{proposition}\label{rk:BasicPropertiesGabrielov2}
Let $\phi:A\lgw B$ be a morphism of  reduced $\C$-analytic algebras corresponding to a morphism of germs of analytic sets $\Phi:(Y,\pb)\lgw (X,\pa)$.
\begin{enumerate}
\item Assume that $B$ is an integral domain. Let $\s: B \lgw B_1$ be such that $\Gr(\s) = \dim(B)$, and  $B_1$ is an integral domain. Then all of the ranks of $\phi$ and $\s \circ \phi$ coincide, that is, $\Gr(\phi)=\Gr(\s\circ\phi)$, $\Fr(\phi)=\Fr(\s\circ\phi)$ and $\Ar(\phi)=\Ar(\s\circ\phi)$.
\item Let $\t: A_1 \lgw A$ be an injective finite morphism where $A_1$ is an integral domain, and assume that $B$ is an integral domain. Then all of the ranks of $\phi$ and $\phi \circ \t$ coincide.
\end{enumerate}
\end{proposition}

\begin{proof}
For (1), by Proposition \ref{rk:BasicPropertiesGabrielov1} (1) we have that $\s$ and $\wdh \s$ are injective, because $B$ is an integral domain. Therefore $\Ker(\wdh\s\circ\wdh\phi)=\Ker(\wdh\phi)$ and $\Ker(\s\circ\phi)=\Ker(\phi)$, and $\Fr(\s\circ\phi)=\Fr(\phi)$ and $\Ar(\s\circ\phi)=\Ar(\phi)$.

Let us denote by $(Z,\pc)$ the germ of analytic set associated to $ B_1$. Since $\Gr(\s)=\dim(B)$, the map $\s^a$ is an analytic diffeomorphism at a generic point in a neighborhood of $\pc$. It follows that $\Gr(\phi)=\Gr(\s\circ\phi)$.

Finally, for (2), assume that $\t$ is  an injective finite morphism where $A_1$ is an integral domain. We have $\Ker(\phi\circ\t)=\Ker(\phi)\cap A_1$. Since $B$ is an integral domain, $\Ker(\phi)$ and $\Ker(\phi\circ\t)$ are prime ideals. Thus, by the Going-Down theorem for integral extensions \cite[Theorem 9.4ii]{Mat}, we have that $\het(\Ker(\phi\circ\t))\leq\het(\Ker(\phi))$, thus $\Ar(\phi)\leq\Ar(\phi\circ\t)$. On the other hand, we have the equality $\Ar(\phi)=\Ar(\phi\circ\t)$ because $\het(\Ker(\phi\circ\t))=\het(\Ker(\phi))$ by \cite[Theorem 9.3ii]{Mat}.
Now, since $\t$ is finite and injective, $\wdh\t$ is also finite and injective by Corollary \ref{fin_inj}. Moreover, we have 
$$
\dim(\wdh A_1)-\het(\Ker\wdh{\t})=\dim(\wdh A)=\dim(A)=\dim( A_1)=\dim(\wdh A_1)
$$
since finite morphisms preserve the dimension and $\t$ is injective. But $\het(\Ker(\wdh{\t}))=0$ if and only if $\Ker(\wdh{\t})=(0)$ because $A_1$ is an integral domain. Thus, $\wdh{ \t}$ is injective and $\Fr(\phi\circ\t)=\Fr(\phi)$.

Eventually, if we denote by $(Z,\pc)$ the germ of analytic set defined by $A_1$, we have $\t^a :(X,\pa)\lgw (Z,\pc)$ is a finite map. Therefore $\Gr(\phi\circ\t)=\Gr(\phi)$.
\end{proof}

We now use the above Proposition to prove the following Lemma, which implies that Theorem \ref{thm:main} follows from Theorem \ref{thm:MainSmooth}:

\begin{lemma}\label{lem_1st_red}
Let $\psi : A\lgw B$ be a morphism of analytic $\C$-algebras, where $B$ is an integral domain. There exists an injective analytic morphism of $\C $-algebras $\phi: \C\{\x\} \lgw \C\{\uu\}$, where  $\x=(x_1,  \ldots,x_m)$ and $\uu=(u_1,  \ldots,u_n)$, such that $\Gr(\psi)=\Gr(\phi)$, $\Fr(\psi)=\Fr(\phi)$ and $\Ar(\psi)=\Ar(\phi)$.
\end{lemma}
\begin{proof}
Note that we can replace $\psi$ by the morphism $\frac{A}{\Ker(\psi)}\lgw B$ induced by $\psi$, since the quotient by the Kernel clearly preserve all of the three ranks. Thus we may assume that $\psi$ is injective. By resolution of singularities there exists an injective morphism of analytic $\C$-algebras $\s:B\lgw B'$ which is a composition of quadratic transformations and analytic isomorphisms such that $B'=\C\{\uu\}$ is regular. Next, by the Normalization Theorem for convergent power series (see \cite[Theorem 45.5]{Na} or  \cite[Corollary 3.3.19]{dJP}), there exists an injective finite morphism $\t:\C\{ x\}\lgw A$. We now set $\phi := \s \circ \psi \circ \t$ and we conclude by Proposition \ref{rk:BasicPropertiesGabrielov2}.
\end{proof}


\subsection{Monomial maps}

Thanks to Lemma \ref{lem_1st_red}, we can now focus on the regular case, that is, when $A = \C\{\x\}$ and $B=\C\{\uu\}$. Apart form the isomorphisms, the typical morphisms between smooth spaces that we use are those of the following form {(see Lemma \ref{lemma:preparationphi} or Section \ref{ssec:Reduction})}:
\begin{enumerate}
\item[i)] Power substitutions:
$$\begin{array}{ccc} \C\{u_1,  \ldots, u_m\} & \lgw & \C\{\wdt u_1,  \ldots, \wdt u_m\}\\
u_1 & \lgm & \wdt u_1^{a_1}\\
\cdots & \cdots &\cdots\\
u_m & \lgm & \wdt u_m^{a_m}\end{array}$$
where $a_i\in\N^*$.
\item[ii)] Quadratic transformations: 
$$\begin{array}{ccc} \C\{u_1,  \ldots, u_m\} & \lgw & \C\{\wdt u_1,  \ldots, \wdt u_m\}\\
u_1 & \lgm & \wdt u_1\wdt u_2\\
u_2 & \lgm & \wdt u_2\\
\cdots & \cdots &\cdots\\
u_m & \lgm & \wdt u_m\end{array}$$
\end{enumerate}

{\begin{remark}
The quadratic transformations do not correspond to blowing ups because they are not bimeromorphic maps. They correspond to the blowing up of a codimension 2 linear space in some affine chart followed by taking the "analytification" at a closed point. The blowing up corresponds to the injective morphism
$$\C\{u_1,\ldots, u_m\}\lgw \fract{\C\{u_1,\ldots, u_m\}[\wdt u_1]}{(u_1-\wdt u_1u_2)}$$  and the "analytification" corresponds to the morphism 
$$\fract{\C\{u_1,\ldots, u_m\}[\wdt u_1]}{(u_1-\wdt u_1u_2)}\lgw \fract{\C\{u_1,\ldots, u_m, \wdt u_1\}}{(u_1-\wdt u_1u_2)}\simeq \C\{\wdt u_1, u_2,\ldots, u_m\}.$$
\end{remark}}

Let $\phi: A \lgw B$ be a morphism of $\C$-algebras, where $A=\C\{\x\}$ and $B=\C\{\uu\}$. It follows from Proposition \ref{rk:BasicPropertiesGabrielov2} that: composition with a power substitution or a quadratic transformation in the target $\s: B \lgw B_1$ preserves all ranks; and composition with power substitutions in the source $\t: A_1 \lgw A$ preserves all ranks. Unfortunately, quadratic transformations in the source may not preserve the ranks:

\begin{remark}[On quadratic transformations in the source]\label{rk:QaudraticBad}
Let us consider the morphism $\phi:\C\{x,y,z\}\lgw \C\{u,v\}$ defined by $\phi(x)=u$, $\phi(y)=v$ and $\phi(z)=uve^v$, and a quadratic transformation $\t:\C\{x_1,y_1,z_1\}\lgw \C\{x,y,z\}$ defined by $\t(x_1)=x$, $\t(y_1)=xy$ and $\t(z_1)=z$. Then we have
$$
\phi\circ\t(x_1)=u,\ \phi\circ\t(y_1)=uv,\ \phi\circ\t(z_1)=uve^v.
$$
which is Osgood's map (see Example \ref{ex:Osgood}). Thus we have $\Fr(\phi\circ\t)=\Ar(\phi\circ\t)=3$ while $\Fr(\phi)=\Ar(\phi)=2$ (because $\Ker(\phi)$ and $\Ker(\wdh\phi)$ are generated by $z-xye^y$) and $\Gr(\phi)=\Gr(\phi\circ\t)=2$.
\end{remark}

Power substitutions and quadratic transformations are monomial morphisms. One basic but important property of these morphisms is the following one:

\begin{lemma}\label{lemma:Convergencemodif}
Consider an $n\times n$ square matrix $M=\left(\mu_{ij}\right)$ of natural numbers such that $\det(M)\neq 0$, and the monomial map $\t\colon \C\lb \x\rb \lgw  \C \lb \uu\rb$ defined by:
\begin{align*}
\t(x_i) = \uu^{\mu_i} = u_1^{\mu_{i1}} \cdots u_n^{\mu_{in}},\quad i=1,\ldots,n.
\end{align*} 
If $f\in \C\lb \x\rb $ is such that $\t(f)\in \C\{\uu\}$, then $f\in \C\{\x\}$.
\end{lemma}

\begin{proof}
Consider the formal expansions 
\[
f=\sum_{\a\in\N^n} f_\a \x^\a \text{ and } \t(f) = \sum_{\a\in\N^n} f_\a \uu^{M\cdot\a} = \sum_{\b\in\N^n} g_\b \uu^\b.
\] 
By hypothesis, there exists two constants $A,B\geqslant 1$ such that $|g_\b|\leqslant A B^{|\b|}$ for every $\b \in \mathbb{N}^n$. Let $\mu_{\infty} = \|M\|_{\infty} = \max \mu_{k,j}$. Since $\det(M)\neq 0$, we conclude that
\[
|f_\a|=|g_{_{M\a}}|\leqslant AB^{|M\a|}\leqslant A (B^{n^2 \mu_{\infty}})^{|\a|} 
\] 
for every $\a \in \mathbb{N}^n$, proving that $f$ is convergent.
\end{proof}

The following result shows how we can use power substitutions and quadratic transformations in order to transform a given morphism of convergent power series rings into a morphism with a simpler form, but without changing the ranks:

\begin{lemma}[Preparation of $\varphi$]\label{lemma:preparationphi}
Let $\phi:\C\{x_1,  \ldots,x_n\}\lgw \C\{u_1,  \ldots,u_n\}$ be a morphism of convergent power series rings. There is a commutative diagram
$$
\xymatrix{\C\{\x\}\ar[r]^{\phi} \ar[d]^\t& \C\{\uu\} \ar[d]^\s \\
\C\{\x\}\ar[r]^{\phi'} & \C\{\uu\}}
$$
where
\begin{enumerate}
\item[i)] $\s$ is a composition of quadratic transformations, power substitutions and isomorphisms;
\item[ii)] $\t$ is a composition of power substitutions and isomorphisms;
\item[iii)] If $\Gr(\phi)=1$, then $\phi'(x_1)=u_1$ and $\phi'(x_j)= 0$ for $j=2,  \ldots, n$.
\item[iv)] If $\Gr(\phi)=n=2$, then $\phi'(x_1)=u_1$ and $\phi'(x_2) = u_1^au_2^b$ with $a\geq 0$ and $b\in \mathbb{Z}_{>0}$. 
\item[v)] If $\Gr(\phi)>1$ and $\phi$ is injective, then 
\begin{equation}\label{eq:NormalFormPrepared}
\phi'(x_1)=u_1, \quad \phi'(x_j) = u_1^{a_j} g_j(\uu), \, j=2,  \ldots,n
\end{equation}
where $a_j \in \mathbb{Z}_{\geq 0}$, $g_j(0)=0$ and $g_j(0,u_2,\ldots, u_n)\neq  0$ for $j=2,  \ldots,n$. 
\end{enumerate}
In these conditions, we have $\Gr(\phi') = \Gr(\phi)$, $\Fr(\phi')=\Fr(\phi)$ and $\Ar(\phi') = \Ar(\phi)$.
\end{lemma}

In particular, condition $(iii)$ combined with Lemma \ref{lem_1st_red} immediately implies the following very particular case of Gabrielov's rank Theorem:

\begin{corollary}[Generic rank $1$]\label{rank1}
Let $\phi: A\lgw B$ be a morphism of reduced analytic $\C$-algebras. If $\Gr(\phi)=1$, then $\Fr(\phi)=\Ar(\phi)=1$.
\end{corollary}

\begin{proof}[Proof of Lemma \ref{lemma:preparationphi}] We will prove the lemma by transforming step by step the morphism $\phi$ into the morphism $\phi'$.

Up to a linear change of coordinates in $\C \{ \uu\}$ we may assume that the initial form of $\phi(x_1)$ evaluated at $(u_1,0,  \ldots,0)$ is equal to $Cu_1^e$ for some $e>0$ and $C\in\C^*$. Let $\s_1:\C \{ \uu\}\lgw \C \{ \uu\}$ be the quadratic transform defined by $\s_1(u_1)=u_1$ and $\s_1(u_j)=u_1u_j$ for $j=2,  \ldots,n$. Then $\s_1\circ\phi(x_1)=u_1^e U( \uu)$ where $U( \uu)\in\C \{ \uu\}$ is a unit. Let us replace $\phi$ by $\s_1\circ\phi$. Up to replacing $x_1$ by $\frac{1}{U(0)}x_1$ we may assume that $U(0)=1$. Now, let $V( \uu)\in\C \{ \uu\}$ be a convergent power series whose $e$-th power is equal to $U( \uu)$. Let $\t_1:\C \{\x\}\lgw \C \{\x\}$ be the finite morphism (power substitution) given by $\t_1(x_1)=x_1^e$, $\t_1(x_j)=x_j$ for $j=2,\dots,n$. Replacing $\phi$ by $\phi \circ \t_1$, we may assume that $\phi(x_1)=u_1V(\uu)$ where $V(\uu)$ is a unit. Moreover by composing $\phi$   with the inverse of the isomorphism of $\C \{\uu\}$ sending $u_1$ onto $u_1V( \uu)$, we may assume that $\phi(x_1)=u_1$.

Let $\phi(x_j) = \phi_j(\uu) \in\C\{\uu\}$ the image of $x_j$ under $\phi$ and consider the analytic isomorphism:
$$
x_1\lgm x_1,\ x_j\lgm x_j-\phi_j(x_1,0), \, j=2,  \ldots, n
$$
If $\Gr(\phi)=1$, we conclude that $\phi_j(\uu) \equiv 0$ as we wanted to prove. Let us assume, therefore, that $\Gr(\phi)>1$ and $\phi$ is injective. We easily conclude that $\phi_j(\uu) \notin\C \{u_1\}$ for all $j=2,  \ldots, n$. Furthermore, because of the change of variables, we know that $\phi_j(u_1,0,  \ldots,0) = 0$, which implies that $\phi_j(\uu) = u_1^{a_j} g_j(\uu)$ for some $a_j \geq 0$, $g_j(0)=0$ and $g_j(0,u_2,\ldots, u_n)\neq 0$, proving (v).

Finally, assume that $n=2$. After composing $\phi$ with $k$ quadratic transformations of the form $(u_1,u_2) \lgm (u_1,u_1u_2)$, for a sufficiently large $k$, we can suppose that $g_2(\uu) = u_2^b W(\uu)$, where $b>0$ and $W(0) \neq 0$. After composing $\phi$ with the isomorphism whose inverse is defined by $u_1\lgm u_1$ and $u_2\lgm u_2W(\uu)^{1/b}$, we have the desired result.

The last statement follows from Proposition \ref{rk:BasicPropertiesGabrielov2}.
\end{proof}


\section{Gabrielov's rank Theorem}\label{sec:Reduction}

\subsection{Low dimensional Gabrielov's rank Theorem}

Somehow surprisingly, the most difficult case in the proof of Gabrielov's rank Theorem is the following:

\begin{theorem}[Low-dimension Gabrielov I]\label{th:MainLowDimI}
Let $\phi: \C\{x_1,x_2,x_3\} \lgw \C\{u_1,u_2\}$ be an $\C$-analytic morphism of convergence power series. Then
\[
\Gr(\phi)=\Fr(\phi)=2 \implies \Ar(\phi)=2.
\]
\end{theorem}

Indeed, we deduce the Theorem \ref{thm:main} from Theorem \ref{th:MainLowDimI} in the next subsection, following the same strategy as the one originally used by Gabrielov. Later, sections \ref{sec:OverviewLow} and \ref{sec:SemiGlobal} will be entirely dedicated to proving Theorem \ref{th:MainLowDimI}, where we deviate from Gabrielov's original approach. This last part will involve a geometric setting and the use of transcendental tools.

 There exists a particular case of Theorem \ref{th:MainLowDimI} which admits a simple algebraic proof, namely when the generator of $\Ker(\widehat{\phi})$ is a quasi-ordinary polynomial{, that is when its discriminant is a monomial times a unit}. This particular case turns out to be crucial later on in the proof of Theorem \ref{th:MainLowDimI}. We finish this section by proving this result:

\begin{proposition}[Quasi-ordinary case]\label{prop:ConvergentSolutionsTheOrigin}
Let  $P \in \C\lb x_1,x_2\rb[y]$ be a reduced monic non-constant polynomial and let $\D_P$ denote tis discriminant. Assume the following:
\begin{enumerate}
\item[i)] $\Delta_P=x_1^{a_1}x_2^{a_2}\times \unit(\x)$ for some formal unit $\unit(\x)$,
\item[ii)] there exists a morphism $\phi:\C\{\x,y\}\lgw \C\{u_1,u_2\}$ with $\Gr(\phi)=2$ such that $P\in\Ker(\wdh\phi)$.
\end{enumerate}
Then $P$ admits a non trivial monic divisor in $\C\{x_1,x_2\}[y]$.
\end{proposition}


\begin{proof}
The proof combines the Abhyankar-Jung Theorem, recalled in $\S$ \ref{section:AbhyankarJung} at the end of this paper, with Lemma \ref{lemma:preparationphi}.

By assumption, we have $\Gr(\phi)=2$ and $\Fr(\phi)=2$ since $P\in\Ker(\wdh\phi)$. Because $\Fr(\phi)=2$, $\Ker(\wdh\phi)$ is a height one prime ideal, thus a principal ideal, by \cite[Theorem 20.1]{Mat}. Therefore any generator of $\Ker(\wdh\phi)$ divides $P$. If $\Ar(\phi)=2$, then $\Ker(\phi)$ is height one prime ideal of $\C\{\x,y\}$, and $\Ker(\phi)\C\lb \x,y\rb=\Ker(\wdh\phi)$. Thus any generator $f(\x,y)$ of $\Ker(\phi)$ divides $P$. Since $P$ is a monic polynomial, we must have $f(0,y)\neq0$. Therefore, by the Weierstrass Preparation Theorem, $f(\x,y)=\text{unit}\times \wdt P(\x,y)$ where $\wdt P(\x,y)\in\C\{\x\}[y]$ is a monic polynomial. Thus $\wdt P$ divides $P$, and the proposition is proven in this case.

Let us now prove that $\Ar(\phi)=2$. We denote by $\phi_1(\uu)$, $\phi_2(\uu)$ and $\phi_3(\uu)$ the respective images of $x_1$, $x_2$ and $y$ under $\phi$. By Abhyankar-Jung Theorem, we can expand $P(y)$ as
$$
P=\prod\limits_{i=1}^{d} \left(y-\widehat{\xi}_i({x_1}^{1/e},x_2^{1/e})\right)
$$
{where $d=\deg_y(P)$ and $e=d!$.} Furthermore, by Lemma \ref{lemma:preparationphi} (iv)  we may suppose that $\phi_1(\uu)=u_1$ and $\phi_2(\uu) = u_1^au_2^b$ with $b>0$. Therefore we may extend $\phi$ as a morphism $\phi'$ from $\C\{x_1^{1/e},x_2^{1/e},y\}$ to  $\C\{u_1^{1/e},u_2^{1/e}\}$ by defining
$$
\phi'(x_1^{1/e}):=u_1^{1/e},\ 
\phi'(x_2^{1/e}):=u_1^{a/e}u_2^{b/e}.
$$
By Proposition \ref{rk:BasicPropertiesGabrielov2} (2) we have 
$\Ar(\phi')=\Ar(\phi).$
Therefore, if one of the $\wdh\xi_i$ is convergent, $\Ker(\phi')\neq (0)$, thus $\Ar(\phi)=\Ar(\phi')=2$.

By the above reduction, we may assume that $e=1$ by replacing $\phi$ by the morphism:
$$
\begin{array}{ccc} 
x_1 & \lgm & u_1\\
x_2 & \lgm & u_1^{a}u_2^{b}\\
y & \lgm & \phi_3(u_1^e, u_2^e)\end{array}
$$
and $P$ by $P(x_1^e, x_2^e,y)$. Replacing in the original equation, we have, since $P\in\Ker(\wdh\phi)$,
$$
\prod\limits_{i=1}^{d} \left(\phi(y)-\widehat{\xi}_i(u_1, u_1^au_2^b)\right)=0.
$$
Hence, there is an index $i$ such that $\wdh\t(\wdh\xi_i)(u_1,u_1^au_2^b)\in\C\{\uu\}$. Thus, by Lemma \ref{lemma:Convergencemodif}, $\wdh\xi_i(\x)\in\C\{\x\}$ and $\Ar(\phi)=2$, as we wanted to prove.

\end{proof}

\subsection{Reduction of Theorem \ref{thm:main} to Theorem \ref{th:MainLowDimI}}\label{ssec:Reduction}

The proof of Theorem \ref{thm:main} is done by contradiction, following closely the ideas of Gabrielov \cite[Theorem 4.8]{Ga2}. We note that we do not use in this section, at any point, a quadratic transformation $\t: A_1 \lgw A$ in the source, c.f. Remark \ref{rk:QaudraticBad}. We assume:

\medskip
\noindent
\textbf{$(\ast)$} There exists a morphism $\phi:A\lgw B$ of analytic $\C $-algebras, where $B$ is an integral domain, such that $\Gr(\phi)=\Fr(\phi)$ but $\Fr(\phi)<\Ar(\phi)$. 

\bigskip
\noindent
\textbf{1st Reduction.} Suppose that $(\ast)$ holds true. Then, there exists an injective morphism $\phi: \C\{\x\} \lgw \C\{\uu\}$, where $\x=(x_1,  \ldots, x_{m})$ and $\uu=(u_1,  \ldots,u_n)$, such that $\Gr(\phi)=\Fr(\phi)=m-1 \geq 1$, $\Ar(\phi)=m$ and $\Ker(\wdh{\phi})$ is a principal (nonzero) ideal.

\medskip
Indeed, by Lemma \ref{lem_1st_red}, there exists an injective morphism $\phi: \C\{\x\} \lgw \C\{\uu\}$, where $\x=(x_1,  \ldots, x_m)$ and $\uu=(u_1,  \ldots,u_n)$, such that $\Gr(\phi)=\Fr(\phi) \geq 1$, but $\Fr(\phi)<\Ar(\phi)$. Since $\Fr(\phi)<\Ar(\phi)=m$, we know that $\Ker(\wdh{\phi})\neq (0)$. Now, suppose that $\Ker(\wdh{\phi})$ is not principal or, equivalently, that its height is at least $2$. By the Normalization Theorem for formal power series, after a linear change of coordinates, the canonical morphism 
\[
\pi:\C\lb x_1,  \ldots, x_{\Gr(\phi)}\rb\lgw \frac{\C\lb \x\rb}{\Ker(\wdh\phi)}
\] 
is finite and injective. Therefore the ideal $\mathfrak p:=\Ker(\wdh\phi)\cap\C\lb x_1,  \ldots, x_{\Gr(\phi)+1}\rb$ is a nonzero height one prime ideal. Since $\C\lb x_1,  \ldots, x_{\Gr(\phi)+1}\rb$ is a unique factorization domain,  $\mathfrak p$ is a principal ideal (see \cite[Theorem 20.1]{Mat} for example).

Now, denote by $\phi'$ the restriction of $\phi$ to $\C\{x_1,  \ldots, x_{\Gr(\phi)+1}\}$. By definition $\Ker(\wdh\phi')=\mathfrak p$, thus $\Fr(\phi')=\Gr(\phi)+1-1=\Gr(\phi)=\Fr(\phi)$. Since $\phi$ is injective, $\phi'$ is injective and $\Ar(\phi')=\Gr(\phi)+1$. Moreover, since $\pi$ is finite, by  Proposition \ref{rk:BasicPropertiesGabrielov2} we have:
$$
\Gr(\phi')=\Gr(\wdh{\phi'})=\Gr(\wdh{\phi})=\Gr(\phi).
$$
 Therefore we replace $\phi$ by $\phi'$ and we assume that $\Gr(\phi)=\Fr(\phi)=m-1$ and $\Ar(\phi)=m$, as we wanted to prove. 

\bigskip
\noindent
\textbf{2nd Reduction.} Suppose that $(\ast)$ holds true. Then, we claim that there exists an injective morphism $\phi: \C\{\x\} \lgw \C\{\uu\}$, where $\x=(x_1,  \ldots, x_{n+1})$ and $\uu=(u_1,  \ldots,u_n)$ (that is, $m=n+1$), such that $\Gr(\phi)=\Fr(\phi) =n $ and $\Ker(\wdh{\phi})$ is a principal (nonzero) ideal (in particular, $\Fr(\phi)=n<n+1=\Ar(\phi)$).

\medskip
We consider the morphism given in the 1st Reduction. Up to a linear change of coordinates in $\uu$, we can suppose that the rank of the Jacobian matrix of $\phi$ evaluated in $(u_1,  \ldots ,u_{r},0,  \ldots ,0)$ is still equal to $r:=\Gr(\phi)$. Let us denote by $\phi_0$ the composition of $\phi$ with the quotient map $\C \{ \uu\}\lgw\C \{  u_1,  \ldots ,u_r\}  $, which satisfies $\Gr(\phi_0)=\Gr(\phi)$. Now, by Proposition \ref{rk:BasicPropertiesGabrielov1}  we have:
 $$
1=m-\Gr(\phi_0)\geq m-\Fr(\phi_0)=\het(\Ker(\wdh{\phi}_0))\geq \het(\Ker(\wdh{\phi}))=1
.$$
 The last inequality comes from the fact that $\Ker(\wdh\phi)\subset \Ker(\wdh\phi_0)$. This shows that $\het(\Ker(\wdh{\phi}_0))=1$. Therefore $\Ker(\wdh{\phi}_0)=\Ker(\wdh{\phi})$ since both are prime ideals of height one. We note, furthermore, that $\phi_0$ is injective. Indeed, if $\Ker(\phi_0)\neq (0)$, take a nonzero $f\in\Ker(\phi_0)$, and note that $f\in\Ker(\wdh{\phi})$ by the previous equality, which implies that $f\in\Ker(\phi)$, a contradiction. We can assume that $n=r$ and $m=n+1$, as was required.
 
\bigskip
\noindent
\textbf{3rd Reduction.} Suppose that $(\ast)$ holds true. Then, we claim that there exists an injective morphism $\phi: \C\{\x\} \lgw \C\{\uu\}$, where $\x=(x_1,x_2, x_{3})$ and $\uu=(u_1,u_2)$ (that is, $m =3$ and $n=2$), such that $\Gr(\phi)=\Fr(\phi) =2 $ and $\Ker(\wdh{\phi})$ is a principal (nonzero) ideal.

\medskip
(Note that the 3rd Reduction contradicts Theorem \ref{th:MainLowDimI}, providing the desired contradiction, and proving Theorem \ref{thm:main}.)

\medskip
We consider the morphism given in the 2nd Reduction. Let $P$ be a generator of $\mbox{Ker}(\widehat{\phi})$. After a linear change of coordinates we may assume that $P$ is a Weierstrass polynomial with respect to $x_{n+1}=:y$ and that $\phi(x_i)$ is not constant for $i=1,  \ldots,n$. Our goal is to reduce the dimension $n$ by restriction to hyperplanes in the coordinates $x$. More precisely, we rely in two results about generic sections:

\begin{theorem}[{Abhyankar-Moh Reduction Theorem, \cite[pp. 31]{A-M}}]\label{thm:Abhyankar-Moh}
Let $F \in \C \lb \x \rb $ be a divergent power series, and denote by $\Lambda$ the subset of $\C$ of all constants $\lambda \in \C$ such that $F(\lambda x_2,x_2,  \ldots,x_n) \in \C \{x_2,  \ldots,x_n\}$ is convergent. Then $\Lambda$ has measure zero. 
\end{theorem}

\begin{theorem}[{Formal Bertini Theorem, \cite{C}}]\label{thm:LocalBertini}
Let $\k$ be an uncountable field, $n\geq 3$ and $\x=(x_1,  \ldots, x_n)$. Let $P(y)\in\k\lb \x\rb[y]$ be irreducible. Then there is an at most countable subset $A\subset \k$, such that $P(cx_2+dx_3,x_2,  \ldots, x_n,y)$ remains irreducible in $\k\lb x_2,  \ldots, x_n\rb[y]$ for every $c\in \k\setminus A$ and every $d\in\k$ transcendental over $\k_{P,c}$ where $\k_{P,c}$ denotes the field extension of the prime field of $\k$ generated by the coefficients of $P(\x,y)$ and $c$ (in particular $\k_{P,c}$ is a countable field).
\end{theorem}

{In order to be self-contained we give a complete proof of Theorem \ref{thm:LocalBertini}} {in $\S\S$ \ref{ssec:LocalBertini}.}

We will say that a hyperplane $x_1 -\sum_{i=2}^n \lambda_i x_i=0$ is \emph{generic} if $\lambda \in \C^{n-1}$ can be chosen outside a subset of measure zero. Note that both theorems demand a generic hyperplane section in the variable $\x$. Our task now is to modify the morphism (without changing its ranks) in order to guarantee that it can be restricted to a generic hyperplane section.

Indeed, by Lemma  \ref{lemma:preparationphi}, we may assume that $\Gr(\phi)>1$ and that the $\phi(x_j)$ with $j=1$,   \ldots, $n$ satisfy the normal form given in Lemma \ref{lemma:preparationphi} (v). The image of $\varphi$, nevertheless, does not yet necessarily include generic hyperplanes in $x$. In order to deal with this issue we use a trick of Gabrielov (whose idea we illustrate in the concrete example \ref{ex:GabrielovTrick} below).

Let us perform the trick. Up to composing $\phi$ with a quadratic transformation $u_1 \lgm u_1$ and $u_j\lgm u_1u_j$ for $j=2,  \ldots,n$, we may suppose that $a_j \in \mathbb{Z}_{>0}$ in the normal forms \eqref{eq:NormalFormPrepared} of Lemma \ref{lemma:preparationphi} (v). Furthermore, up to making power substitutions of the form $x_j \lgm x_j^{\alpha_j}$ with $\alpha_j = \prod_{k\neq j} a_k$ for every $j=2,  \ldots, n$, we can suppose that $a_j = a>0$ is independent of $j$. Finally, we consider the power substitution $x_1 \lgm x_1^{a+1}$. All these operations preserve the ranks of the morphism by Proposition \ref{rk:BasicPropertiesGabrielov2}. We therefore have obtained the following normal form:
\[
\phi(x_1) = u_1^{a+1}, \quad \phi(x_j) = u_1^{a} g_j(\uu), \quad j=2,  \ldots, n
\]
where $g_j(0)=0 $ and $g_j(0,u_2, \ldots, u_n)\neq  0$. Now, let us consider a linear function $h_{\lambda}(\x) = x_1 - \sum_{j=2}^n \la_j x_j$ with $ \lambda_j \in \C$ for $j=2,  \ldots,n$. Note that:
\[
\varphi\left(x_1-\sum_{i=2}^{n} \la_i x_i\right) = u_1^a \left( u_1 - \sum_{i=2}^{n} \lambda_j g_j(\uu)\right) =: \uu^a g_{\lambda}(\uu).
\]
We claim that for a generic choice of $\lambda \in \C^{n-1}$, the hypersurface $V_\la:=(g_{\lambda}(\uu)=0)$ is not contained in the set of critical points of the morphism $\varphi^a: (\C^{n},0) \lgw (\C^{n+1},0)$ (where we recall that $\varphi^{a \ast}=\varphi$), which we denote by $W$.

Indeed, on the one hand, outside of a proper analytic subset $\Gamma \subset \C^{n-1}$, we know that $g_{\lambda}(0)=0$, $g_{\lambda}(0,u_2,\ldots, u_n) \neq 0$ and $\partial_{u_1}g_{\lambda}(0) \neq 0$. Therefore, by the implicit function Theorem, the equation $g_{\lambda}(\uu)=0$ admits a nonzero solution $u_1=\xi_{\lambda}(u_2,  \ldots,u_n)$. 

On the other hand, assume by contradiction that $V_\la\subset W$ for $\la\in\Lambda \subset \C^{n-1}$, where $\Lambda$ is of positive measure. Since $W$ is a proper analytic subset of $\C^n$, $W$ contains only finitely many hypersurfaces. Therefore, when $\la$ runs over $\Lambda$, $V_\la$ runs over finitely many hypersurfaces. We define an equivalence relation on $\Lambda$ by $\la\simeq \wdt\la$ if $V_\la=V_{\wdt\la}$. Then $\Lambda$ is the disjoint union of the (finitely many) equivalence classes. Therefore, at least one of these equivalence classes is not included in an affine hyperplane. We denote it by $\Lambda_0$. Then
\begin{equation}\label{inv_matrix}
 \sum_{j=2}^n (\widetilde{\lambda}_j-\lambda_j) \cdot g_j(\uu) |_{V_{\lambda}} = \left(g_{\lambda}(\uu) - g_{\widetilde{\lambda}}(\uu)\right)|_{V_{\lambda}} \equiv 0, \quad \forall \, \lambda, \,\widetilde{\lambda} \in  \Lambda_0.
\end{equation}
Since $\Lambda_0$ is not included in an affine hyperplane, we may choose $\la$, $\wdt \la^{(1)}$,   \ldots, $\wdt \la^{(n)}$ such that the vectors $\la-\wdt \la^{(1)}$,   \ldots, $\la-\wdt \la^{(n)}$ are linearly independent. 
Therefore, \eqref{inv_matrix} applied to the $\la-\wdt\la^{(k)}$ implies that $g_j(\uu)|_{V_{\lambda}} \equiv 0$, so that $u_1|_{V_{\lambda}} \equiv 0$. In other words, $(g_{\lambda}(\uu) =0) \subset (u_1=0)$, so that $\xi_{\lambda}(u_2,  \ldots,u_n) = 0$, which is a contradiction.

Therefore, for a generic choice of $\lambda \in \C^{n-1} $, the induced morphism:
\[
\psi_{\lambda}: \frac{\C \lb x_1,   \ldots, x_n\rb [y] }{\left(x_1-\sum_{i=2}^{n} \la_i x_i\right)} \lgw \frac{\C \lb u_1,   \ldots, u_n\rb }{\left( u_1 - \sum_{i=2}^{n} \lambda_j g_j(\uu)\right)}
\]
is such that $\Gr(\psi_{\lambda}) = n-1$. Finally, let us assume that $n>3$. By Theorems  \ref{thm:Abhyankar-Moh} and \ref{thm:LocalBertini}, the polynomial $P$ remains irreducible and divergent in  $\C \lb  x_1,  \ldots, x_n\rb[y]/ (x_1-\sum_{i=2}^{n} \la_i x_i)$ for a generic choice of $\lambda =(\lambda_2,  \ldots,\lambda_{n}) \in \C^{n-1}$. We conclude, therefore, that $\Fr(\psi_{\lambda}) = n-1$ and $\Ar(\psi_{\lambda}) = n$. By repeating this process, we obtain the desired morphism with $n=2$. 

\begin{example}[On Gabrielov's trick]\label{ex:GabrielovTrick}\hfill
\begin{itemize}
\item[(1)] Consider the morphism $\varphi : \C\{x_1,x_2,x_3,y\} \lgw \C\{u_1,u_2,u_3\}$ given by 
\[
\varphi(x_1)=u_1, \quad \varphi(x_2)=u_1^2u_2, \quad \varphi(x_3) = u_1^2u_3 \quad \text{ and }\varphi(y)= f(\uu).
\] 
Consider a hyperplane $H_{\lambda}=(x_1 - \lambda_2 x_2 - \lambda_3 x_3)$ where $(\lambda_2,\lambda_3) \neq (0,0)$, and note that  $\varphi(x_1 - \lambda_2 x_2 - \lambda_3 x_3) = u_1 U(\uu)$ where $U(0) \neq 0$. It follows that the restriction of $\varphi$ to $H_{\lambda}$ induces a morphism:
\[
\psi_{\lambda}: \frac{\C\{x_1,x_2,x_3,y\}}{(x_1 - \lambda_2 x_2 - \lambda_3 x_3)} \lgw 
\frac{\C\{u_1,u_2,u_3\}}{(u_1)} = \C\{u_2,u_3\}
\]
which is constant equal to zero, so that $\Gr(\psi_{\lambda}) =0$.

\item[(2)] (Gabrielov's trick). In order to solve the above issue, we perform a power substitution in the source (which preserves all ranks by Proposition \ref{rk:BasicPropertiesGabrielov2}). More precisely, we consider $x_1 = x_1^3$, so that we now have $\varphi(x_1) = u_1^3$. It now follows that:
\[
\varphi(x_1 - \lambda_2 x_2 - \lambda_3 x_3) = u_1^2 \left(u_1 - \lambda_2 u_2 - \lambda_3 u_3 \right)
\] 
and the restriction of $\varphi$ to $H_{\lambda}$ induces a morphism:
\[
\psi_{\lambda}: \frac{\C\{x_1,x_2,x_3,y\}}{(x_1 - \lambda_2 x_2 - \lambda_3 x_3)} \lgw \frac{\C\{u_1,u_2,u_3\}}{(u_1 - \lambda_2 u_2 - \lambda_3 u_3)}
\]
and we can easily verify that $\Gr(\psi_{\lambda}) = 2 = \Gr(\varphi)-1$.
\end{itemize}
\end{example}

\subsection{Proof of Formal Bertini Theorem}\label{ssec:LocalBertini}
Before giving the proof of the theorem we make the following remarks:
\begin{remark}[On the Formal Bertini Theorem]\hfill
\begin{enumerate}
\item A stronger version of the above result was originally stated in \cite{C} (where $A$ is assumed to be finite), but we were not able to verify the proof. We follow the same strategy as Chow to prove the above result.
\item Tougeron proposes an alternative proof of the Formal Bertini Theorem in \cite[page 349]{To} via a ``Lefschetz type"  Theorem. The proof is geometric and does not adapt in a trivial way to the formal case.
\item (Counterexample of formal Bertini for $n=2$). Consider the irreducible polynomial $P(\x,y) = y^2 - (x_1^2+x_2^2)$. For every $\lambda \in \C$, we have that:
\[
P(x,\lambda x,y) = y^2 - (1+\lambda^2)x^2 = (y-x\sqrt{1+\lambda^2})(y+x\sqrt{1+\lambda^2})
\]
is a reducible polynomial. Therefore, there is no Formal Bertini Theorem for $n=2$.
\end{enumerate}
\end{remark}

For $c\in \k$ we set 
$$
R(c):=\frac{\k\lb \x\rb[z]}{(zx_3-(x_1+cx_2))}=\k\lb \x\rb\left[\frac{x_1+cx_2}{x_3}\right].
$$
 The ideal generated by the $x_i$ is a prime ideal of $R(c)$, denoted by $\p(c)$, and the completion of $R(c)_{\p(c)}$ is $\k(z)\lb x_2,  \ldots, x_n\rb$.
 
We begin by giving the proof of the following Proposition, given as a lemma in \cite{C}:
\begin{proposition}\label{bertini_prop}
Let $\k$ be an uncountable field and $n\geq 3$. 
Let $P(y)\in\k\lb \x\rb[y]$ be an irreducible  monic polynomial. Then there is a countable subset $A\subset \k$ such that, for all $c\in\k\setminus A$, $P(y)$ is irreducible in $\wdh R(c)_{\p(c)}[y]$.
\end{proposition}

\begin{proof}
More generally, for $c_1$,   \ldots, $c_s\in\k$, $s$ distinct elements with $s\geq 2$, we set 
$$R(c_1,  \ldots, c_s):=\k\lb \x\rb[z_1,  \ldots, z_s]/(z_1x_3-(x_1+c_1x_2),  \ldots, z_sx_3-(x_1+c_sx_2)).$$
The ideal generated by the $x_i$ is a prime ideal of $R(c_1,  \ldots, c_s)$, denoted by $\p$, and the completion of $R(c_1,  \ldots, c_s)_\p$ is isomorphic to $\k(z_1,  \ldots, z_s)\lb x_3,  \ldots, x_s\rb$, where the $z_i$ satisfy the following relations over $\k$ (for every $i$, $j$, $k$, $\ell$ with $i\neq j$ and $k\neq\ell$):
\begin{equation}\label{lin_rel}\frac{c_iz_j-c_jz_i}{c_i-c_j}=\frac{c_kz_\ell-c_\ell z_k}{c_k-c_\ell};\  \frac{z_i-z_j}{c_i-c_j}=\frac{z_\ell-z_k}{c_\ell-c_k}.\end{equation}
Therefore  the completion of $R(c_1,  \ldots, c_s)_\p$ is isomorphic to $\k(z_1, z_2)\lb x_3,  \ldots, x_s\rb$, and $z_1$ and $z_2$ are algebraically independent over $\k$ (we may replace $z_1$ and $z_2$ by any other $z_i$, $z_j$).
Let us fix $s=2$. We have, in $\wdh R(c_1,c_2)_{\p}$:
$$z_1x_3=x_1+c_1x_2,\ z_2x_3=x_1+c_2x_2,\ x_1=\frac{c_2z_1-c_1z_2}{c_2-c_1}x_3,\ \ x_2=\frac{z_2-z_1}{c_2-c_1}x_3.$$
Therefore, for an element $f\in \wdh R(c_1,c_2)_{\p}$, we have $f\in R$ if and only if
\begin{equation}\label{key_trick}
f=\sum_{\a\in\N^n}f_\a z_1^{\a_1}z_2^{\a_2}x_3^{\a_3}\cdots x_n^{\a_n}, \quad \text{ where } \quad \left[\a_1>\a_3 \text{ or } \a_2>\a_3\Longrightarrow f_\a=0\right]
\end{equation}

Let $P(y)\in \k\lb \x\rb[y]$ be irreducible and assume that there exist an uncountable set $I$ and distinct elements $c_i\in \k$, $i\in I$, such that $P(y)$ is reducible in $\wdh R(c_i)_{\p(c_i)}[y]$. We fix two such $c_i$ that we denote by $c_1$ and $c_2$.
By the previous discussion, we may assume that the $\wdh R(c_i)_{\p(c_i)}$ are all embedded in $\k(z_1,z_2)\lb x_3,  \ldots, x_n\rb$. Since there is only finitely many ways of splitting a monic polynomial into the product of two monic polynomials, we may assume that we have the same factorization $P(y)=P_1(y)P_2(y)$ in all the $\wdh R(c_i)_{\p(c_i)}[y]$.

Let $f$ be a coefficient of $P_1(y)$ or $P_2(y)$ in $\wdh R(c_1,c_2)_{\p(c_i)}$, that we write
$$f=\sum_{\b\in\N^{n-2}}f_\b x_3^{\b_3}\cdots x_n^{\b_n}$$
where the $f_\b\in \k(z_1,z_2)$. Since $f\in \wdh R(c_1)_{\p(c_1)}$ and
$$\wdh R(c_1)_{\p(c_1)}\simeq \k(z_1)\lb x_2,  \ldots, x_n\rb\simeq \k(z_1)\lb z_2x_3,x_3,  \ldots, x_n\rb$$
we have that $f_\b\in\k(z_1)[z_2]$ for all $\beta$, and $\deg_{z_2}(f_\b)\leq \b_3$ for all $\beta$ such that $f_\b\neq 0$. By symmetry, we have that $f_\b\in\k[z_1,z_2]$ and $\deg_{z_1}(f_\b)\leq \b_3$ for all $\beta$ such that $f_\b\neq 0$. Now let us choose one more $c_i$, that we denote by $c_3$. By \eqref{lin_rel}, we have
$$z_2=\frac{c_3-c_2}{c_3-c_1}z_1+\frac{c_1-c_2}{c_1-c_3}z_3.$$
By replacing $z_2$ by $\frac{c_3-c_2}{c_3-c_1}z_1+\frac{c_1-c_2}{c_1-c_3}z_3$ in the $f_\b$, we obtain the coefficients $g_\b$ of the expansion
 $$f=\sum_{\b\in\N^{n-2}}g_\b x_3^{\b_3}\cdots x_n^{\b_n}$$
  as an element of $\k(z_1,z_3)\lb x_3,  \ldots, x_n\rb$. In particular, we have that $g_\b\in \k[z_1,z_3]$ and $\deg_{z_i}(g_\b)\leq \b_3$ for $i=1$ or $3$, for all $\b$ such that $g_\b\neq 0$. For a given $\b$, we have 
  $$\deg_{z_1}(g_\b)\leq \deg_{z_1,z_2}(f_\b)$$
  where $\deg_{z_1,z_2}$ denotes the total degree in $z_1$, $z_2$. This inequality may be strict, as some cancellations may occur. But, for a given $\b$, there is a finite set $A_\b\subset \k$ (possibly empty) such that 
  $$\frac{c_3-c_2}{c_3-c_1}\notin A_\b\Longrightarrow \deg_{z_1}(g_\b)= \deg_{z_1,z_2}(f_\b).$$
We remark that the map $H:c_3\in\k\lgm \frac{c_3-c_2}{c_3-c_1}$ is injective. Therefore the set
$$A:=\bigcup_{\b\in\N^{n-2}}H^{-1}(A_\b)$$
is at most countable. Since $I$ is uncountable, we may choose $c_3\notin A$. Therefore for such a $c_3$, we have 
$$\forall \b,\ \ \ \deg_{z_1}(g_\b)= \deg_{z_1,z_2}(f_\b).$$
Now, by the previous discussion done for the $f_\b$, we also have $\deg_{z_1}(g_\b)\leq \b_3$. Therefore, for every $\b$, we have $\deg_{z_1,z_2}(f_\b)\leq \b_3$. Thus, by \eqref{key_trick}, we see that $f\in R$.\\
This argument applies to any coefficient of $P_1(y)$ or $P_2(y)$. Hence, we have that $P_1(y)$, $P_2(y)\in R[y]$. This contradicts the assumption that $P(y)$ is irreducible in $R[y]$.
\end{proof}

Now we can give the proof of Theorem \ref{thm:LocalBertini}:

\begin{proof}[Proof of Theorem \ref{thm:LocalBertini}]
We apply Proposition \ref{bertini_prop} to $P(y)$. Let $c\notin A$. By construction, the image of $P(y)$ in $S(c)[y]$ is
$$P(cx_2+zx_3,x_2,  \ldots, x_n,y)$$
where $S(c)[y]$ is identified with $\k(z)\lb x_2,  \ldots, x_n\rb[y]$. In fact, we have
$$P(cx_2+zx_3,x_2,  \ldots, x_n,y)\in \k_{P,c}(z)\lb x_2,  \ldots, x_n\rb[y].$$
Therefore, if we replace $z$ by any element of $\k$ that is transcendental over $\k_{P,c}$, we have that $P(cx_2+zx_3,x_2,  \ldots, x_n,y)$ is irreducible.
\end{proof}


\section{Proof of the low-dimensional Gabrielov Theorem}\label{sec:OverviewLow}

\subsection{Geometrical framework}
In order to prove Theorem \ref{th:MainLowDimI}, we will use geometric arguments involving transcendental tools. In particular, the article changes pace and we use essentially geometric language instead of algebraic. We start by fixing the notation.

Given a point $\pa \in \C^n$, when $n$ is clear from the context we write $\mathcal{O}_{\pa}$ for the ring of analytic germs $\O(\C^n)_\pa$, and $\widehat{\mathcal{O}}_{\pa}$ for its completion. Given an analytic germ $f \in \mathcal{O}_{\pa}$, we denote by $\widehat{f}_{\pa}$ the Taylor series associated to $f$ at the point $\pa$. In particular, the Taylor mapping is the morphism of local rings:
\[
\begin{matrix}
T_a : & \mathcal{O}_a  &\lgw&  \widehat{\mathcal{O}}_{\pa}\\
 & f & \lgm &  \widehat{f}_{\pa}
\end{matrix}
\]
A coordinate system centered at $\pa$ is a collection of functions $\x=(x_1,  \ldots,x_n)$ which generate the maximal ideal $\mathfrak{m}_{\pa}$ of $\mathcal{O}_{\pa}$, in which case we recall that $\mathcal{O}_{\pa}$ is isomorphic to $\C\{\x\}$. We define, similarly, coordinate systems $\widehat{\x}=(\widehat{x}_1,  \ldots,\widehat{x}_n)$ of $\widehat{\mathcal{O}}_{\pa}$, and we note that $\widehat{\mathcal{O}}_{\pa}$ is isomorphic to $\C\llbracket \widehat{\x} \rrbracket$. We say that a coordinate system $\widehat{\x}$ of $\widehat{\mathcal{O}}_{\pa}$ is \emph{convergent} if there exists a coordinate system $\x$ of $\mathcal{O}_{\pa}$ such that $\widehat{x}_i = T_{\pa}(x_i)$ for $i=1,  \ldots,n$. Whenever $\widehat{\x}$ is a convergent coordinate system, we abuse notation and we identify it with $\x$.

Let $\Phi: Y\lgw X$ be a complex analytic map between two smooth analytic spaces. Given $\pb\in Y$ and $\pa\in X$ where $\pa=\Phi(\pb)$, we will denote by $\Phi_\pb$ the associated analytic map germ $\Phi_\pb: (Y,\pb)\lgw (X,\pa)$. We denote by $\Phi^*_\pb : \O_\pa\lgw \O_\pb$ the morphism of local rings defined by
$$
\forall f\in\O_\pa,\ \ \Phi^*_\pb(f):=f\circ\Phi_\pb.
$$
Then we denote by $\wdh\Phi^*_\pb : \wdh\O_\pa\lgw \wdh\O_\pb$ the completion morphism of $\Phi^*_\pb$. Following the notation given in the introduction, we note that $(\Phi^{\ast}_\pb)^a$ is the localization of the morphism $\Phi$ to $\pb$, that is, $(\Phi^{\ast}_\pb)^a: (Y,\pb) \lgw (X,\pa)$.

When $\phi : A\lgw B$ is a morphism of local rings, and $P\in A[y]$ is a polynomial with coefficients in $A$, $P=p_0+p_1y+\cdots+p_dy^d$, we will use the following abuse of notation:
$$
\phi(P)=\phi(p_0)+\phi(p_1)y+\cdots+\phi(p_d)y^d\in B[y].
$$
Concretely, we will use this notation with $A = \mathcal{O}_{\pa}$ or $\widehat{\mathcal{O}}_{\pa}$, and $B = \mathcal{O}_{\pb}$ or $\widehat{\mathcal{O}}_{\pb}$.

\subsection{Geometrical formulation of low-dimension results}

We now re-phrase Theorem \ref{th:MainLowDimI} and Proposition \ref{prop:ConvergentSolutionsTheOrigin} in the geometrical context. Instead of a morphism $\varphi: \C\{x_1,x_2,y\} \lgw \C\{u_1,u_2\}$ such that $\Gr(\phi) =2$ and $\Fr(\phi) = 2$, we work with a morphism of germs $\Phi:(\C^2,\pa) \lgw (\C^3,\Phi(\pa))$ such that $\Phi^{\ast} = \varphi$, $\Gr(\Phi^{\ast}) =2$ and $\Fr(\Phi^{\ast}) = 2$. In order to simplify the notation, we always assume that $\Phi(\pa) \in \pa \times \C$ (which is always possible up to a translation in the target) and that there exists a formal polynomial $P \in \widehat{\mathcal{O}}_{\pa}[y]$ such that $\widehat{\Phi}^{\ast}(P) = 0$, that is, $\Ker(\widehat{\Phi}^{\ast}) = (P)$ (which we can always suppose by Weierstrass preparation). 

We introduce the following definition:

\begin{definition}[Convergent factor]
Let $P \in \widehat{\mathcal{O}}_{{\pa}}[y]$ be a non-constant monic polynomial. We say that $P$ admits a convergent factor (at $\pa$) if there exists an analytic non-constant monic polynomial $q \in \mathcal{O}_{\pa}[y]$ such that $T_{\pa}(q) = \widehat{q}_{\pa}$ divides $P$.
\end{definition}

We are ready to reformulate Theorem \ref{th:MainLowDimI}:

\begin{theorem}[Low-dimension Gabrielov II]\label{th:MainLowDim'}
Let $P \in \widehat{\mathcal{O}}_{{\pa}}[y]$ be a monic polynomial, where $\pa \in \C^2$. Suppose that there exists an analytic morphism $\Phi:(\C^{2},\pa) \lgw (\C^3,\Phi(\pa))$, generically of rank $2$, such that $\Phi(\pa) \in \{\pa\} \times \C$ and  
\[
\wdh\Phi^*_{\pa}(P)  =  0.
\]
Then $P$ admits a convergent factor. In particular, if $P$ is a formally irreducible polynomial, then $P$ is an analytic polynomial.
\end{theorem}

{The condition $\Phi(\pa)\in\{\pa\}\times \C$ is only stated for convenience. In practise we can always assume that this condition holds by a translation in the range of $\Phi$.} Note that Theorem \ref{th:MainLowDim'} immediately implies Theorem \ref{th:MainLowDimI}.

{It is easy to see that, in Theorem \ref{th:MainLowDim'}, we can always suppose that the polynomial $P \in \widehat{\mathcal{O}}_{{\pa}}[y]$ is reduced. In particular, the discriminant of $P$  (with respect to the projection $(\x,y) \lgm (\x)$) is a nonzero formal power series which we denote by $\Delta_P \in \C\lb \x\rb$.} {As said before,} in what follows, quasi-ordinary singularities (that is, when $\Delta_P$ is monomial) will play an important role. We start by distinguishing two cases of quasi-ordinary singularities:

\begin{definition}[Monomial discriminant]
Let $P\in\widehat{\mathcal{O}}_{{\pa}}[y]$ be a reduced monic non-constant polynomial. We say that the discriminant $\Delta_P$ is \emph{formally monomial} if there exists a coordinate system $\x=(x_1,x_2)$ of $\widehat{\mathcal{O}}_{{\pa}}$ such that:
\[
\Delta_P =\x^{\alpha} \cdot W(\x) = x_1^{\alpha_1} \cdot x_2^{\alpha_2} \cdot  W(\x)
\]
where $W(\x)\in\widehat{\mathcal{O}}_{{\pa}}$ is a unit. We say that the discriminant $\Delta_P$ is \emph{analytically monomial} if $\x=(x_1,x_2)$ is a coordinate system of $\mathcal{O}_{\pa}$.
\end{definition}

\begin{example}[Formally vs analytically monomial]\label{rk:convergentMonomial}
Suppose that:
\[
\Delta_P = x_1 \cdot \left(x_2 - \sum_{n=1}^{\infty} n! \cdot x_1^n \right).
\]
On the one hand, there is no (formal) unit $U(x_1,x_2)$ such that $U(x_1,x_2)(x_2 - \sum_{n=1}^{\infty} n! \cdot x_1^n)$ is convergent. Indeed, by the unicity of the preparation given by the Weierstrass preparation Theorem, this would imply that $U$ and $x_2 - \sum_{n=1}^{\infty} n! \cdot x_1^n$ are convergent power series, which is not the case. Therefore $\Delta_P$ is \emph{not} analytically monomial. On the other hand, after the formal change of coordinates:
\[
\widehat{x}_1 = x_1, \quad \widehat{x}_2 = x_2 - \sum_{n=1}^{\infty} n! \cdot x_1^n.
\]
we conclude that $\Delta_P(\widehat{\x}) = \widehat{x}_1 \cdot \widehat{x}_2$, which is formally monomial.
\end{example}

We are ready to reformulate Proposition \ref{prop:ConvergentSolutionsTheOrigin}:

\begin{proposition}[Final case]\label{prop:ConvergentSolutionsTheOrigin'}
Let $P \in \widehat{\mathcal{O}}_{{\pa}}[y]$ be a reduced monic polynomial, where $\pa \in \C^2$, whose discriminant $\Delta_P$ is analytically monomial. Suppose that there exists an analytic morphism $\Phi:(\C^{2},\pa) \lgw (\C^3,\Phi(\pa))$, generically of rank $2$, such that $\Phi(\pa) \in \{\pa\} \times \C$ and  
\[
\wdh\Phi^*_{\pa}(P)  =  0,
\]
then $P$ admits a convergent factor.
\end{proposition}

Note that Proposition \ref{prop:ConvergentSolutionsTheOrigin} immediately implies Proposition \ref{prop:ConvergentSolutionsTheOrigin'}.

We remark that Proposition \ref{prop:ConvergentSolutionsTheOrigin'} is a particular case of Theorem \ref{th:MainLowDim'}. The proof of Theorem \ref{th:MainLowDim'} consists in reducing to the quasi-ordinary case via \emph{blowing ups} in the target of $\Phi$. As we have remarked in \ref{rk:QaudraticBad}, a blowing up does not preserve the formal and analytic ranks of the analytic germ. We must make global arguments over the blown-up space, as we discuss in the next section.

\subsection{Blowing ups and the inductive scheme} 
We consider an analytic manifold $N$ (which is eventually assumed of dimension $2$) and a simple normal crossing divisor $F$ over $N$. An \emph{admissible blowing up} (for the couple $(N,F)$) is a blowing up:
\[
\tg : (\widetilde{N},\widetilde{F}) \lgw (N,F)
\]
whose center $\mathcal{C}$ is connected and has normal crossings with $F$, that is, at every point $\pa \in \mathcal{C}$, there exists a coordinate system $x$ of $\mathcal{O}_{\pa}$ and $t\in\{1,  \ldots,n\}$ such that $\mathcal{C} = (x_1=\cdots=x_t=0)$ and $F$ is locally given as a finite union of hypersurfaces $(x_i=0)$. In particular, note that if $\mathcal{C}$ is a point, then the blowing up is admissible.

A \emph{sequence of admissible blowing ups} is a finite sequence of morphisms:
$$
\xymatrix{ (\C^2,\pa)=(N_0,\pa)  & (N_1,F_1)  \ar[l]^{\qquad\  \tg_1} & \cdots \ar[l]^{\qquad \tg_2}& (N_r,F_r)  \ar[l]^{\tg_r}}
$$
and we fix the convention that $\tg: (N_r,F_r) \lgw (N_0,F_0)$ denotes the composition of the sequence.

The proof of Theorem \ref{th:MainLowDim'} demands an argument in terms of the history of the exceptional divisors, so it is convenient to introduce notation to keep track of the history explicitly. More precisely, consider a sequence of admissible blowing ups $(\tg_1,  \ldots,\tg_r)$. Note that, for every $j \in \{1,  \ldots, r\}$, the exceptional divisor $F_j$ is a simple normal crossing divisor which can be decomposed as follows:
\[
F_j = F_j^{(0)} \cup F_j^{(1)} \cup \cdots \cup F_j^{(j)}, \quad \forall \, j=1,  \ldots,r
\]
where $F_j^{(0)}$ is the strict transform of $F_0$, and for every $k \in \{1,  \ldots,j\}$, the divisor $F_j^{(k)}$ is an irreducible and connected component of $F_j$ which is uniquely defined via the following recursive convention:
\begin{itemize}
\item If $k=j$, then $F^{(j)}_{j}$ stands for the exceptional divisor introduced by $\tg_j$;
\item If $k< j$, the divisor $F_{j}^{(k)}$ is the strict transform of $F_{j-1}^{(k)}$ by $\tg_j$.
\end{itemize}

The proof of Theorem \ref{th:MainLowDim'} will follow from combining Proposition \ref{prop:ConvergentSolutionsTheOrigin'} with the following result, as we show in $\S\S$\ref{ssec:ReductionMain} below:

\begin{proposition}[Inductive scheme]\label{cl:Main}
Let $\pa \in \C^2$ and consider a non-constant reduced monic polynomial $P \in \widehat{\mathcal{O}}_{{\pa}}[y]$. Consider a sequence of admissible blowing ups
$$
\xymatrix{ (\C^2,\pa)=(N_0,\pa)  & (N_1,F_1)  \ar[l]^{\qquad\  \tg_1} & \cdots \ar[l]^{\qquad \tg_2}& (N_r,F_r)  \ar[l]^{\tg_r}}
$$
we set $\s:=\s_1\circ\cdots\circ\s_r$, and we assume  that:
\begin{enumerate}
\item[i)] $\forall \pb \in \tg^{-1}(\pa)$, we have that $\widehat{\tg}_{\pb}^*(\Delta_P)$ is formally monomial,
\item[ii)] $\exists k\in\{1,\ldots, r\}$, $\exists\pb \in F_{r}^{(k)}$ such that $P_{\pb} = \widehat{\tg}^*_{\pb}(P)$ has a convergent factor. 
\end{enumerate}
Then $P$ admits a convergent factor.
\end{proposition}

\begin{remark}
In Example \ref{rk:convergentMonomial}, we have illustrated that the discriminant of $P$ is not analytically monomial in general. In fact, even after a sequence of blowing ups, $\D_P$ is not analytically monomial in general. Indeed, let us choose the polynomial $P$ given in Example \ref{rk:convergentMonomial} {which is formally monomial}. If we consider the quadratic transformation $(x_1,x_2)= (z_1,z_1z_2)$ we get
$$
\wdh\s^*(\Delta_P)=z_1^2\left((z_2-1)-\sum_{n\geq 1}n!z_1^{n-1}\right).
$$
Therefore, $\wdh\s^*(\Delta_P)$ is {formally monomial at the point of coordinates $(z_1,z_2)=(0,1)$ but it} is not analytically monomial at this point  as shown in Example \ref{rk:convergentMonomial}. A straightforward induction shows that this is again the case after finitely many blowing ups. Therefore we need to be careful when we reduce the proof of Theorem \ref{th:MainLowDim'} to Propositions  \ref{prop:ConvergentSolutionsTheOrigin'} and \ref{cl:Main} (c.f. Case II in $\S\S$\ref{ssec:ReductionMain}).
\end{remark}

The proof of Proposition \ref{cl:Main} is given in $\S\S$\ref{ssec:Inductive} below. The crucial technical point to prove it, is the following extension result:

\begin{proposition}[Semi-Global extension]\label{prop:SemiGlobalLessPrecise}
Let $\pa \in \C^2$ and consider a non-constant reduced monic polynomial $P \in \widehat{\mathcal{O}}_{{\pa}}[y]$. Consider a sequence of admissible blowing ups
$$
\xymatrix{ (\C^2,\pa)=(N_0,\pa)  & (N_1,F_1)  \ar[l]^{\qquad\  \tg_1} & \cdots \ar[l]^{\qquad \tg_2}& (N_r,F_r)  \ar[l]^{\tg_r}}
$$
we set $\s:=\s_1\circ\cdots\circ\s_r$, and we assume  that:
\begin{enumerate}
\item[i)] $\forall \pb \in \tg^{-1}(\pa)$, $\widehat{\tg}_{\pb}^*(\Delta_P)$ is formally monomial,
\item[ii)] $\exists\pb \in F_{r}^{(1)}$ such that $P_{\pb} = \widehat{\tg}^*_{\pb}(P)$ has a convergent factor.
\end{enumerate}
Then, there exists an open neighbourhood $U_r^{(1)}$ of $F^{(1)}_r$, and a convergent non-constant polynomial $q \in \mathcal{O}_{U_r^{(1)}} [y]$ such that, at every point $\pc \in F_{r}^{(1)}$, the polynomial $\widehat{q}_{\pc}$ divides $P_{\pc}:= \widehat{\tg}_{\pc}^*(P)$ and either $P_{\pc}=\widehat{q}_{\pc}$ or the quotient $P_{\pc}/\widehat{q}_{\pc}$ does not admit a convergent factor.
\end{proposition}

In Corollary \ref{cor:SemiGlobal} below, we prove a more precise version of the above Proposition. Indeed, $\S$\ref{sec:SemiGlobal} is entirely dedicated to the proof of Proposition \ref{prop:SemiGlobalLessPrecise}, and it includes three Theorems which are of independent interest.

Finally, we need one more ingredient before proving Proposition \ref{cl:Main}. First, note that if $P  \in {\mathcal{O}_{\pa}}[y]$ is convergent, then it is clear that a convergent factor $q$ at a point $\pa$, is also a convergent factor of $P$ on a neighborhood of $\pa$. When $P\in \widehat{\mathcal{O}}_{{\pa}}[y]$ is divergent, then this property still holds over ``fibers", that is:

\begin{proposition}[Convergent factor along fibers]\label{prop:FormalContinuationConvergentFactor} Let $\pa \in \C^n$ and $\Phi: M \lgw \C^n$ be an analytic map, generically of maximal rank, where $M$ is smooth. Let $P \in \widehat{\mathcal{O}}_{{\pa}}[y]$ be a non-constant monic polynomial and suppose that there exists $\pb \in \Phi^{-1}(\pa)$ such that $P_{\pb} = \Phi_{\pb}^{\ast}(P)$ admits a convergent factor $q$ at $\pb$. Then, there exists a neighbourhood $U$ of $\pb$ such that, for every point $\pc \in U\cap \Phi^{-1}(\pa)$, the polynomial $q$ is a convergent factor of $P_{\pc}$.
\end{proposition} 

We prove Proposition \ref{prop:FormalContinuationConvergentFactor} in $\S\S$ \ref{sec:ProofFormalContinuation} below.


\subsection{Proof of Gabrielov's low dimension Theorem (Reduction of Theorem \ref{th:MainLowDim'} to Proposition \ref{cl:Main})}\label{ssec:ReductionMain}

The discriminant $\Delta_P$ is a formal curve in $(\C^2,\pa)$, so it admits a resolution of singularities via blowing ups of points which are always convergent centres. In other words, there exists a sequence of (analytic) point blowing ups:
$$
\xymatrix{ (\C^2,\pa)=(N_0,\pa)  & (N_1,F_1)  \ar[l]^{\qquad\  \tg_1} & \cdots \ar[l]^{\qquad \tg_2}& (N_s,F_s)  \ar[l]^{\tg_s}}
$$
such that for every $\pb \in \tg^{-1}(\pa)$ {(here $\s:=\s_1\circ\cdots\circ\s_s$)}, the pulled-back discriminant $\widehat{\tg}^*_{\pb}(\Delta_P)$ is formally monomial (but not necessarily analytically monomial). Note that any further sequence of point blowing ups preserve this property, and we may compose this sequence with further blowing ups of points if necessary. Now, let 
$$
\xymatrix{ (\C^2,\pa)\times\C=(M_0,\pa\times \C)  & (M_1,E_1)  \ar[l]^{\qquad \qquad \tg_1\times\Id} & \cdots \ar[l]^{\quad\  \tg_2\times\Id}& (M_s,E_s)  \ar[l]^{\tg_s\times\Id\quad}}
$$
be the associated sequence of admissible blowing ups over $\C^3$. We can now show that there exists a sequence of point blowing ups in the source of the morphism $\Phi$:
$$
\xymatrix{ (\C^2,\pa)=(L_0,\pa)  & (L_1,D_1)  \ar[l]^{\qquad\  \lag_1} & \cdots \ar[l]^{\qquad \lag_2}& (L_s,D_s)  \ar[l]^{\lag_s}}
$$
where $\lag_i$ denotes a finite sequence of blowing ups, including length zero (so, the identity); and mappings $\Phi_i: (L_i,D_i) \lgw (M_i,E_i)$ for $i=1,  \ldots, s$, such that the following diagram commutes:
$$
\xymatrix{ L_0 \ar[d]^{\Phi_0} && L_1 \ar[d]^{\Phi_1} \ar[ll]^{\lag_1} && \cdots \ar[ll]^{\lag_2} && L_{s-1}\ar[d]^{\Phi_{s-1}} \ar[ll]^{\lag_{s-1}} && L_s \ar[d]^{\Phi_s} \ar[ll]^{\lag_s}\\
M_0 && M_1 \ar[ll]^{\tg_1\times\Id} && \cdots \ar[ll]^{\tg_2\times\Id} && M_{s-1} \ar[ll]^{\tg_{s-1}\times\Id} && M_s \ar[ll]^{\tg_s\times\Id}
}$$
where $\Phi_0=\Phi$. Indeed, this result follows from usual resolution of indeterminacy of maps: let $\mathcal{I}_{0}$ be the reduced ideal sheaf whose support is the first center of blowing up $\mathcal{C}_0$ in $M_0$, and consider its pullback $\mathcal{J}_0 = \Phi_0^{\ast}(\mathcal{I}_0)$. Let $\lambda_1:L_1 \lgw L_0$ denote the sequence of point blowing ups that principalizes $\mathcal{J}_0$; we conclude by the universal property of blowing ups the existence of the morphism $\Phi_1:L_1 \lgw M_1$. It is enough to repeat this argument for the entire sequence.

Now, since $\Phi_0$ is generically of maximal rank and the $\lag_i$ are sequences of point blowing ups, we conclude that $\Phi_s$ is generically of maximal rank. We denote by $\lag$ the composition of the $\lag_i$. Let $\pc \in \lag^{-1}(\pa)$ and denote by $\Phi_s(\pc) = (\pb,b)$ its image. Note that:
$$\wdh{\Phi}_{s\,\pc}^*(P_\pb)=\wdh{\Phi}_{s\, \pc}^*\circ\wdh\sg_\pb^*(P)=\wdh\lag_\pc^*\circ\wdh{\Phi}^*_{0\,\pa}(P)$$
by hypothesis. 

We now consider the  two following cases:

\begin{figure}
\includegraphics[width=0.65\linewidth]{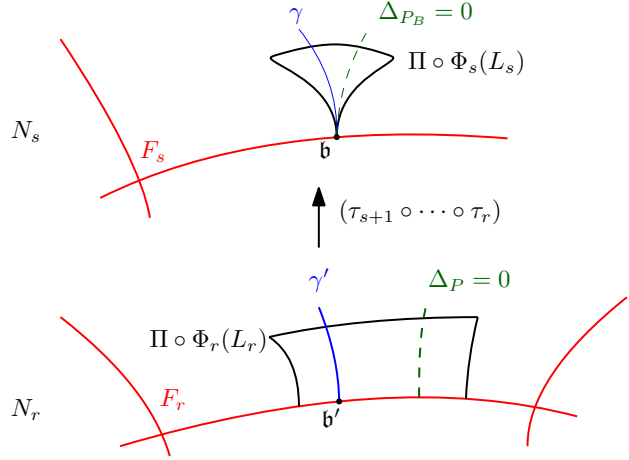}
\caption{Proof of Theorem \ref{th:MainLowDim'}: case II.}
\label{fig:LowGabrielovCaseII}
\end{figure}

\medskip
\noindent
\emph{Case I:} Suppose that $\Delta_{P_{\pb}}$ is analytically monomial. In this case we do not need to make any other subsequent blowing ups (and $r=s$ when we apply Proposition \ref{cl:Main}). Indeed, all hypothesis of Proposition \ref{prop:ConvergentSolutionsTheOrigin'} are satisfied, so we conclude that $P_{\pb}$ admits a convergent factor. This implies that all hypothesis of Proposition \ref{cl:Main} are satisfied, so we conclude that $P$ admits a convergent factor, as we wanted to prove. 

\medskip
\noindent
\emph{Case II:} Suppose that $\Delta_{P_{\pb}}$ is not analytically monomial, but only formally monomial. In this case, we make further blowing ups in order to reduce to Case $I$ (and $r>s$ when we apply Proposition \ref{cl:Main}).

Let $\Sigma \subset L_s$ be the analytic subset of $L_s$ where $\Phi_s$ is not of maximal rank (note that this set is the support of the ideal generated by the determinant of the two-by-two minors of the Jacobian of $\Phi_s$). Since $\Sigma$ is an analytic curve, there exists {an analytic} curve $\tilde{\gamma}:(\C,0) \lgw (L_s,\pc)$ which intersects $\Sigma\cup\Phi^{-1}(E_s)$ only at $\pc$. We set $\Pi:M_s\lgw N_s$ the canonical projection. Then, taking $\gamma = \Pi\circ\Phi_s \circ \tilde{\gamma}$, we obtain an analytic curve on $(N_s,\pb)$ which intersects the exceptional divisor $F_s$ only at $\pb$. {We claim} that $\gamma$ and $(\Delta_{P_{\pb}}=0)$ can not have flat contact. {Indeed, if they had flat contact, then $(\Delta_{P_{\pb}}=0)$ would have a  convergent factor $h$ that is coprime with $u$ where $F_s=(u=0)$. Thus $\D_{{P_{\pb}}}=h^au^b\times \unit$ and $(\D_{{P_{\pb}}}=0)$ would be convergent which contradicts the assumption.}
This implies that by a sequence of point blowing ups we can separate the strict transform of $\gamma$ and $(\Delta_{P_{\pb}}=0)$. We are now in Case $I$ when we center at the point $\pb'$ given by the strict transform of $\gamma$. Indeed, at $\pb'$, the germ defined by $\Delta_{P_{\pb'}}$ is equal to the germ defined by one irreducible component of the exceptional divisor.

\subsection{The induction Scheme (Reduction of Proposition \ref{cl:Main} to Proposition 
\ref{prop:SemiGlobalLessPrecise} and \ref{prop:FormalContinuationConvergentFactor})}\label{ssec:Inductive}
The proof of Proposition \ref{cl:Main} follows by induction on the lexicographical order of $(r,k)$. Note that the first step of the induction, that is when $r=k=0$, is tautological. We now fix $(r,k)$ and we assume that the Proposition is true whenever $(r',k')<(r,k)$. We divide the proof in two parts, depending if $k=1$ or $k>1$:

\smallskip
\noindent
\textbf{Case I: $k=1$.} By Proposition \ref{prop:SemiGlobalLessPrecise}, there exists an open neighbourhood $U^{(1)}_r$ of $F^{(1)}_r$, and a convergent polynomial $q \in \mathcal{O}_{U^{(1)}_r} [y]$ such that, at every point $\pc \in F_{r}^{(1)}$, the polynomial $\widehat{q}_{\pc}$ divides $P_{\pc}= \wdh\s_\pc^*(P) $ and, furthermore, either $P_{\pc}=\widehat{q}_{\pc}$ or the quotient $P_{\pc}/\widehat{q}_{\pc}$ does not admit a convergent factor. Note that, since $F_r$ is connected and $U_r^{(1)}$ is open, by Proposition \ref{prop:FormalContinuationConvergentFactor}, at each connected component of $F_r \setminus F_r^{(1)}$, there exists a point in this component, say $\pc_j$, where $\wdh\s_\pb^*(P) $ admits a convergent factor.

We now consider the geometrical picture after only the first blowing up $\tg_1$, which has exceptional divisor $F_1= F_1^{(1)}$. Let $\{\pa_1,  \ldots,\pa_l\}$ be all the points in $F_{1}$ which are centres of subsequent blowing ups. Note that $\tg_2 \circ \cdots \circ \tg_r: (N_r,F_r) \lgw (N_1,F_1)$ is an isomorphism at every point of $F_1 \setminus \{\pa_1,  \ldots,\pa_l\}$. There exists, therefore, an open neighbourhood $V_1$ of $F_1\setminus \{\pa_1,  \ldots,\pa_l\}$, and an analytic polynomial $\widetilde{p} \in \mathcal{O}_{V_1}[y]$ such that:
\[
\tg_r^{\ast} \circ \cdots \circ  \tg_1^{\ast}(\widetilde{p}) = q
\]
Now, let us denote by $P_j:=(\wdh \s_{1})^*_{\pa_j}(P)$ for $j=1,  \ldots,l$, which is a non-constant monic polynomial in $\widehat{\mathcal{O}}_{\pa_{j}}[y]$. We consider the sequence of blowing ups $\tg_{(2)}:=\tg_2 \circ \cdots \circ \tg_r$. Note that the pull-back of the discriminant $\Delta_{P_{j}}$ is everywhere formally monomial, since it coincides with the pull-back of $\Delta_P$ by the entire sequence $\tg$. Furthermore, since $P$ admits a convergent factor at a point in every connected component of $F_r \setminus F_r^{(1)}$, we conclude that $P_j$ admits a convergent factor (at some point, say $\pc_j$) after its composition with $\tg_{(2)}$. It follows that $P_j$ satisfies all conditions of Proposition \ref{cl:Main} with $(r',k')$ such that $r'\leq r-1$. By induction, $P_j$ admits a convergent factor of maximal degree $p_j$, defined in a neighborhood $W_j$ of $\pa_j$. Furthermore, the degree of $p_j$ must be the same as the degree of $\widetilde{p}$, since they coincide after pull-back by $\tg_{(2)}$ at a point of $F_r \setminus F_r^{(1)}$, by the inductive assumption.

Finally, since $p_j$ is convergent in a neighborhood $W_j$, there exists a point $\pb_{j} \in W_j \cap F_1^{(1)}$ where $p_j$ and $\widetilde{p}$ are well-defined. Since these polynomials are convergent, have the same degree, and $\widetilde{p}$ is the convergent factor of maximal degree of $P$, we conclude that $\widetilde{p} = p_j$ at $\pb_j$. It follows that $\widetilde{p}$ extends in a neighbourhood $U_1:=V_1 \cup_{j=1}^l W_j$ of $F_1$, and it formally divides $\wdh\s_1^*(P)$ everywhere in $F_1$. {We claim that $\widetilde{p} = \s_1^*(p)$ pour some $p\in \mathcal O_\pa[y]$. Indeed, we can define the coefficients of $p$ on an open punctured disc of $\C^2$ centered at the origin by  $\widetilde p=\s_1^*(p)$, since $\s_1$ is biholomorphic outside the origin. Then we conclude with the Riemann extension theorem (see e.g. \cite[Theorem 4.1.24]{dJP}). Since $\widetilde p$ formally divides $\wdh\s_1^*(P)$, we have that $p$ formally divides $P$, as we wanted to prove.}

\begin{figure}
\includegraphics[width=\linewidth]{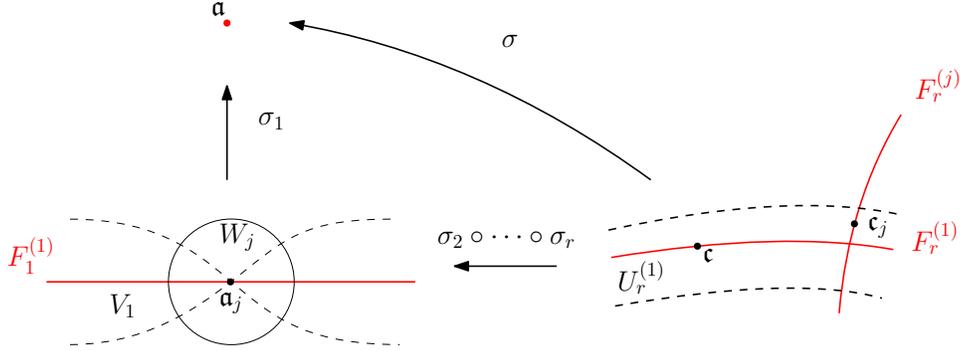}
\caption{Proof of Proposition \ref{cl:Main}: case I.}
\label{fig:CaseI}
\end{figure}

\smallskip
\noindent
\textbf{Case II: $k>1$.} Let $\pa_{k-1} $ denote the center of the blowing up $\tg_k$, and consider: 
\[
P_{k-1}:=\widehat{\left(\tg_{1}\circ\cdots\circ \tg_{k-1}\right)^*}_{\pa_{k-1}}(P),
\] 
and note that $P_{k-1}$ is a non-constant monic polynomial which belongs to $\widehat{\mathcal{O}}_{\pa_{k-1}}[y]$. We consider the sequence of blowing ups $\tg_{(k)}:=\tg_k \circ \cdots \circ \tg_r$. Note that the pull-back of the discriminant $\Delta_{P_{k-1}}$ is everywhere formally monomial, since it coincides with the pull-back of $\Delta_P$ by the entire sequence $\tg$. Furthermore, since $P$ admits a convergent factor over the exceptional divisor created by the blowing up $\sg_k$, we conclude that $P_k$ admits a convergent factor after its composition with $\tg_{(k)}$ at some point $\pc \in F_r^{(k)}$. It follows that $P_k$ satisfies all hypothesis of Proposition \ref{cl:Main} with $(r',k')$ such that $r'\leq r-k<r-1$. By induction, we conclude that $P_{k-1}$ admits a convergent factor $p_{k-1} \in \mathcal{O}_{\pa_{k-1}}[y]$, defined in some neighborhood $W_{k-1}$ of $\pa_{k-1}$. Therefore, by Proposition \ref{prop:FormalContinuationConvergentFactor}, there exists a point $\pb \in F_{k-1}^{(j)}$, for some $j\leq k-1$, where $\widehat{\sg}_{\pb}^*(P)$ admits a convergent factor. We conclude by induction.

\begin{figure}
\includegraphics[width=\linewidth]{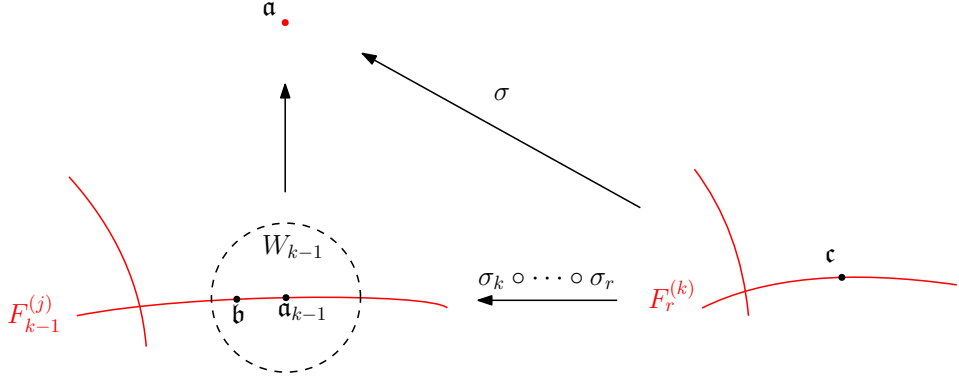}
\caption{Proof of Proposition \ref{cl:Main}: case II.}
\label{fig:CaseII}
\end{figure}

\subsection{Convergence of factors along fibers (Proof of Proposition \ref{prop:FormalContinuationConvergentFactor})} \label{sec:ProofFormalContinuation}

The proof is divided in two steps, depending on the nature of $\Phi^{-1}(\pa)$:

\medskip
\noindent
\emph{Step I:} Suppose that $\Phi^{-1}(\pa) = E$ is a SNC divisor. Let $(v,\w) = (v,w_2,  \ldots,w_n)$ be a coordinate system centered at $\pb$ such that $(v=0) \subset E$. We can write:
\[
x_i = v^{\alpha_i} \Psi_i(v,\w), \quad i=1,  \ldots, n
\]
where $\Psi = (\Psi_1,  \ldots, \Psi_n)$ is an analytic morphism defined in some open neighborhood $U$ of $\pb$, and $\Psi_i(0,\w) \neq 0$ for all $i =1,  \ldots,n$. Without loss of generality, we can suppose that there exists a disc $D \subset \C$ such that $U = D^n$ and that the coefficients of the polynomial $q$ are convergent over $U$. Next, we set $\alpha = (\alpha_1,  \ldots,\alpha_n)$, and let $A \in\widehat{\mathcal{O}}_{\pa}$ be a fixed function. We consider the expansion of $A$ in terms of $\alpha$-homogeneous polynomials:
\[
A = \sum_{i=0}^{\infty} A_i(\x) \quad \text{ where } A_i(\x^{\alpha}) \text{ is an homogenous polynomial of degree }i
\]
This implies that:
\[
\Phi_{\pb}^*(A) = \sum_{i=0}^{\infty} v^i \psi(A_i)= \sum_{i=0}^{\infty} v^i  \sum_{j=0}^{\infty} v^j a_{ij}(\w) = \sum_{k=0}^{\infty} v^k \sum_{i+j=k} a_{ij}(\w) = \sum_{k=0}^{\infty} v^k b_{k}(\w)
\]
where, since $A_i$ are polynomials and $\Psi$ is convergent in $U= D^n$, we conclude that $a_{ij}(\w)$ are analytic functions defined on $D^{n-1}$. Moreover, since for every $k$, the function $b_k(\w)$ is a finite sum of functions $a_{ij}(\w)$, we conclude that $b_k(\w)$ are also analytic functions defined over $D^{n-1}$. Denoting by $\mathcal{O}_{D^{n-1}}$ the ring of analytic functions defined over $D^{n-1}$, we conclude that $b_k(\w) \in \mathcal{O}_{D^{n-1}}$ for every $k \in \mathbb{N}$, so that $\Phi^*_{\pb}(A) \in \mathcal{O}_{D^{n-1}} \lb v \rb$. Since the choice of $A \in\widehat{\mathcal{O}}_{\pa}$ was arbitrary, we conclude that $P_{\pb} \in \mathcal{O}_{D^{n-1}} \lb v \rb [y]$. In particular, for every $\pc \in (v=0) \cap D^n$, we get that $P_{\pc} = P_{\pb}$ as elements of $\mathcal{O}_{D^{n-1}} \lb v \rb [y]$. Furthermore, the factor $q$ also belongs to $\O_{D^{n-1}} \lb v \rb [y]$, and it follows from the Euclidean division that $P_{\pb} = q \cdot Q + R$, where $R \in \mathcal{O}_{D^{n-1}} \lb v \rb [y]$ has an identically zero formal expansion at $\pb$. Since the ring $\mathcal{O}_{D^{n-1}}$ is of convergent series, this implies that $R \equiv 0$. It follows that $q$ divide $P_{\pc}$, at every point $\pc \in (v=0) \cap U$. We conclude the Proposition by remarking that $E$ is a SNC divisor and the choice of the hypersurface $(v=0) \subset E$ was arbitrary.

\medskip
\noindent
\emph{Step II:} Let $\Sigma:=\Phi^{-1}(\pa)$; since the morphism $\Phi$ is generically of maximal rank, we conclude that $\Sigma$ is a proper analytic subvariety of $M$. Consider a resolution of singularities of $\Sigma$, that is, an analytic morphism $\sg: (M',E') \lgw (M,\Sigma)$ of maximal rank such that $\sg^{-1}(\Sigma) = E'$ is a SNC divisor. From Step I, at every point $\pb' \in \sg^{-1}(\pb)$, there exists a neighbourhood $U_{\pb'}$ where $\s_{\pb'}(q)$ is a convergent factor of $P_{\pb'}$. Since $\Sigma$ is an analytic subvariety of $M$, there exists an open neighbourhood $U$ of $\pb$ where $U \cap \Sigma$ is connected. Since $\sg$ is proper, furthermore, up to shrinking $U$ we can suppose that 
\begin{equation}\label{eq:Basic1}
\sg^{-1}(U) \subset \bigcup_{\pb' \in \sg^{-1}(\pb)} U_{\pb'}.
\end{equation}
Now given a point $\pc \in \Sigma \cap U$, suppose by contradiction that $q$ is not a factor of $P_{\pc}$, that is, the formal division $P_{\pc} = q_{\pc} \cdot Q + R$ has a nonzero remainder $R \in \widehat{\mathcal{O}}_{\pc}[y]$. It follows that, at every point ${\pc'} \in \sg^{-1}(\pc)$ we have that $\sg_{{\pc'}}(R):=R_{\pc'} \neq 0$, which implies that $q_{{\pc'}}$ does not divide $P_{{\pc'}}$. But ${\pc'} \in E \cap \sg^{-1}(U)$, leading to a contradiction with \eqref{eq:Basic1}. It follows that $q$ formally divides $P_{\pc}$ at every point $\pc \in \Sigma \cap U$, as we wanted to prove.

\section{Semi-Global extension of convergent factors}\label{sec:SemiGlobal}


\subsection{Semi-global extension overview (Proof of Proposition \ref{prop:SemiGlobalLessPrecise})}

This subsection contains the full strategy to prove Proposition \ref{prop:SemiGlobalLessPrecise}. In order to motivate each object, we leave the proofs and development of the necessary supporting techniques, namely Theorems \ref{thm:NewtonPuiseux}, \ref{thm:SemiGlobalFormal}, \ref{thm:IzumiWalsh} and Proposition \ref{representative}, to subsections \ref{ssec:Newton-Puiseux-Eisenstein}, \ref{ssec:OnConvProjRings}, \ref{ssec:SemiGlobalFormalExtention} and \ref{ssec:LocalToSemiGlobal}.

We start by providing the adequate algebraic setting for the discussion. More precisely, following Tougeron \cite{To2}, we build up a subring of the algebraic closure of $\C(\x)$  {(See Definition \ref{def:projectiveRings} and Theorem \ref{thm:NewtonPuiseux})} which captures geometrical properties necessary to address Proposition \ref{prop:SemiGlobalLessPrecise}. Our presentation is at first general (that is, $ n \in \mathbb{N}$), and we specialize it to the case $n=2$ when it becomes convenient for the presentation.

\subsubsection{Preliminaries on valuation rings} 
We consider the ring of power series $\C\lb \x \rb$ where $\x=(x_1,  \ldots, x_n)$ and we denote by $\C(\!(\x)\!)$ its field of fractions. We denote by $\nu$ the $(\x)$-adic valuation on $\C\lb \x\rb $. The valuation $\nu$ extends to $\C(\!(\x)\!)$ by defining $\nu(f/g)=\nu(f)-\nu(g)$ for every $f$, $g\in\C\lb \x\rb$, $g\neq 0$.
We consider the valuation ring $V_\nu$ associated to it, that is
$$V_\nu:=\{f/g\mid f,g\in\C\lb \x\rb ,\ \nu(f)\geq \nu(g)\}=\{F\in\C(\!(\x)\!)\mid \nu(F)\geq 0\}.$$
 We denote by $\wdh V_\nu$ its completion. Classically elements of $\wdh V_\nu$ can be represented as formal series $A=\sum_{k\in\N} A_k$ where the $A_k$ are in $V_\nu$ and, if $A_k\neq 0$, $\nu(A_k)= k$. For $B\in V_\nu$ we have $B=f/g$ where $f$, $g\in\C\lb \x\rb$. We expand $f=\sum_{k\geq k_1} f_{k}$, $g=\sum_{k\geq k_2} g_{k}$ where the $f_{k}$ and $g_{k}$ are homogeneous polynomials of degree $k$ and $f_{k_1},g_{k_2}\neq 0$. Therefore we have
 $$
 B=
 \frac{f_{k_1}}{g_{k_2}}\left(1+\sum_{k>k_1}\frac{f_k}{f_{k_1}}\right)\left(1+\sum_{k>k_2}\frac{g_k}{g_{k_2}}\right)^{-1}
 $$
 Therefore  the elements $A \in \wdh V_\nu$ are of the form:
\[
A=\sum_{k\in\N}\frac{a_k(\x)}{b_k(\x)},
\]
where $a_k$ and $b_k$ are homogeneous polynomials such that $\deg(a_k)-\deg(b_k)=k$.

\begin{definition}[Weighted-homogenous polynomial]
Let $z_1$,   \ldots, $z_r$ be indeterminates, and let $\omega_1$,   \ldots, $\omega_r\in\Q_{\geq 0}$. We say that a polynomial $\Gamma(\x,\z)\in\C[\x,\z]$ is \emph{$(\omega_1,  \ldots, \omega_r)$-weighted homogeneous} if $\Gamma(\x,z_1^{\omega_1},  \ldots, z_r^{\omega_r})$ is homogeneous.
\end{definition}


\begin{definition}[Homogeneous elements]
A \emph{homogeneous element} $\g$ is an element of an algebraic closure of $\C(\x)$, satisfying a relation of the form $\Gamma(\x,\g)=0$ for some $\omega$-weighted homogeneous polynomial $\Gamma(\x,z)$, where $\omega \in \Q_{\geq 0}$. Furthermore, if $\Gamma(\x,z)$ is monic in $z$, we say that $\g$ is an \emph{integral homogeneous element}. In this case, $\omega$ is called the \emph{degree} of $\g$.  {It is straightforward to check that for $\x$ outside the discriminant locus of $\Gamma(\x,z)$ and $\la\in\R_+^*$, we have $\g(\la \x)=\la^\omega\g(\x)$. Thus the degree of $\g$ is well defined.}{This degree is well-defined. Indeed, the minimal polynomial of $\gamma$ is $\omega$-weighted as a factor of an $\omega$-weighted homogeneous polynomial, and conversely any polynomial which is weighted homogeneous and a multiple of a $\omega$-weighted homogeneous polynomial is $\omega$-weighted homogeneous.} 
\end{definition}

{\begin{remark}
 In order to give a general overview of the strategy of the proof we postpone the presentation of  the main properties of homogeneous elements  (cf. Lemma \ref{weighted}, Corollary \ref{cor_weighted} and Lemma \ref{PET}).
\end{remark}}

Given an integral homogeneous element $\gamma$ of degree $\omega$, there exists an extension of the valuation $\nu$, still denoted by $\nu$, to the field $\C(\x)[\g]$, defined by
$$
\nu\left(\sum_{k=0}^{d-1}a_{k}(\x)\g^{k}\right)=\min\{\nu(a_{k})+k\omega\}.
$$
where $d$ is the degree of the field extension $\C(\x) \lgw \C(\x)[\g]$. In particular, the above property justifies our use of the valuation $\nu$ (instead of naively using the notion of order, which would not extend to integral homogeneous elements).

More generally, given homogeneous elements $\gam =(\g_1$,   \ldots, $\g_r)$ of degrees $\oo = (\omega_1$,   \ldots, $\omega_r)$, there exists an extension of the valuation $\nu$, still denoted by $\nu$, to the field $\C(\x)[\gam]$, defined by
$$
\nu\left(\sum_{k_1,  \ldots, k_r}a_{\bold k}(\x)\gam^{\bold k}\right)=\min\{\nu(a_{\bold  k})+k_1\omega_1+\cdots+ k_r \omega_r\}
$$
where the indices $k_j$ run over $\{0,  \ldots, d_j-1\}$, where $d_j$ is the degree of the field extension 
$$
\C(\x)[\g_1,  \ldots, \g_{j-1}]\lgw \C(\x)[\g_1,  \ldots, \g_j].
$$
We denote by $V_{\nu,\gam}$ the valuation ring of $\nu$ defined on $\C(\!(\x)\!)[\gam]$. We note that $V_{\nu,\gam}$ is a local ring, and we denote by $\wdh V_{\nu,\gam}$ its completion. We remark that the image of $V_{\nu,\gam}$ or of $\wdh V_{\nu,\gam}$ under the valuation $\nu$ is the subgroup $\Gamma_{\nu,\gam}$ of $\Q$ generated by $1$, $\omega_1$,   \ldots, $\omega_r$. This group being a finitely generated $\Z$-module, it is a discrete group, therefore $V_{\nu,\gam}$ and $\wdh V_{\nu,\gam}$ are  discrete valuation rings.\\
Note that all elements of $\wdh V_{\nu,\gam}$ are written as finite sums:
$$
\sum_{k_1,  \ldots, k_r} a_{\bold{k}}(\x)\gam^{\bold k}, \quad \text{ where }a_{\bold k}(\x) \in  \Frac(\wdh V_\nu) \text{ and }\nu(a_{\bold{k}}(\x)\gam^{\bold{k}})\geq 0.
$$
\begin{definition}[Initial term]\label{def:Initial}
For a nonzero element $\xi\in\wdh V_{\nu,\gam}$, we can write 
\[
\xi=\sum_{k\in\Q_{\geq 0}\cap \Gamma_{\nu,\gam}} \xi_k, \quad \text{ where }\xi_k\in\C(\!(\x)\!)[\gam] \text{ are $\nu$-homogenous terms of degree }k
\]
The initial term of $\xi$, denoted by $\ini_{\nu}(\xi)$, is defined as $\xi_{k_0}$ where $$k_0=\min\{k\in\Q_{\geq 0}\cap \Gamma_{\nu,\gam}\mid \xi_k \neq 0\},$$
which is well-defined because $\Gamma_{\nu,\gam}$ is a finitely generated subgroup of $\Q$.
\end{definition}

%
\subsubsection{Projective rings and a Newton-Puiseux-Eisenstein Theorem}

We are now ready to define the rings of functions which we are interested in:

\begin{definition}[Projective rings]\label{def:projectiveRings}
Let $h$ be a homogeneous polynomial. We denote by $\PP_h (\!(\x)\!)$ the subring of $\Frac(\wdh{V_{\nu}})$ characterized by the following property: $A \in \PP_h (\!(\x)\!)$ if there exists $k_0 \in \Z$, $\a$, $\b\in\N$ and $a_k(\x)$ homogeneous polynomials for $k\geq k_0$ so that:
\[
A = \sum_{k\geq k_0} \frac{a_k(\x)}{h^{\a k+\b}}, \quad \text{ where } \nu(a_k)-(\a k+\b)\nu(h)=k,\, \forall\,k\geq k_0
\]
We denote by $\PP_h\lb \x\rb $ the subring of $\PP_h (\!(\x)\!)$ of elements $A$ whose initial term $k_0$ belongs to $\mathbb{Z}_{\geq 0}$. When $\g$ is an integral homogeneous element, we denote by $\PP_{h}\lb \x,\g \rb$ the subring of $\wdh  V_{\nu,\g}$, whose elements $\xi$ are of the form:
\[
\xi=\sum_{k=0}^{d-1} A_k(\x) \gamma^k, \quad \text{ where } A_k \in \PP_h (\!(\x)\!) \text{ and } \nu(A_{k}(\x)\g^{k})\geq 0, \, \, \, k=0,  \ldots,d-1.
\]
\end{definition}

\begin{remark}
Let us remark that the family $(\PP_h\lb \x\rb)_h$ is a directed set, since, for two homogeneous polynomials $h_1$ and $h_2$, we have
$$\PP_{h_i}\lb \x\rb\subset \PP_{h_1h_2}\lb \x\rb \text{ for } i=1,2.$$
\end{remark}

\begin{remark}[Geometrical properties of projective rings]\label{rk:GeometricalFormalProjectiveRing} \hfill
\begin{enumerate}
\item (Invariance by linear coordinate changes). Let $h$ be a homogenous polynomial, $\g$ be an homogenous integral element and $\Phi : (\C^n,0) \lgw (\C^n,0)$ be a linear coordinate change. Note that $\widetilde{h}:= h \circ \Phi$ is an homogenous polynomial and that $\PP_h\lb \x \rb$ (respectively $\PP_h (\!(\x)\!)$ and $\PP_{h}\lb \x,\g \rb$) is isomorphic to $\PP_{\widetilde{h}}\lb \x \rb$ (respectively $\PP_{\widetilde{h}} (\!(\x)\!)$ and $\PP_{\widetilde{h}}\lb \x,\g \rb$). 

\item (Blowing ups) We specialize to the case $n=2$. Let $\sg: (N,F)=(N_r,F_r) \lgw (\C^2,0)$ be a sequence of point blowing ups and let $\pb \in F_r^{(1)}$ be a point. We can find two different normal forms, depending on the nature of $\pb$:

\begin{enumerate}
\item Suppose that $\pb$ does not belong to the intersection of two exceptional divisors. Then, up to a linear change of coordinates in $x$, there exists a system of local coordinates $(v,w)$ centered at $\pb$ such that:
\[
x_1= v, \quad x_2 = vw 
\]
and, if $A \in\PP_h \lb \x\rb$, we obtain:
\[
A_{\pb}:= \wdh\s^*_{\pb}(A) := \sum_{k\geq 0} \frac{a_k(1,w)}{h(1,w)^{\a k +\b}} v^k,
\]
so that $\wdh\sigma^*_{\pb}: \PP_h \lb \x\rb \lgw \C (w)\lb v \rb $ is a well-defined morphism. In particular, if $\pb$ does not belong to the strict transform of $(h=0)$ (so $h(1,0)\neq 0$) then $A_{\pb} \in \widehat{\mathcal{O}}_{\pb}$ and $\wdh\sigma^*_{\pb}: \PP_h \lb \x\rb[y] \lgw \widehat{\mathcal{O}}_{\pb}$ is well-defined.

\item Suppose that $\pb$ belongs to the intersection of two components of the exceptional divisor. Then, up to a linear change of coordinates in $x$, there exists local coordinate system $(v,w)$ centered at $\pb$ and $c \in \mathbb{Z}_{> 0}$ such that:
\[
x_1= vw^c, \quad x_2 = v w^{c+1},
\]
and, if $A \in\PP_h \lb \x\rb$, we obtain:
\[
A_{\pb}:= \wdh\s^*_{\pb}(A) := \sum_{k\geq 0} \frac{a_k(1,w)}{h(1,w)^{\a k +\b}} (w^cv)^k.
\]
\end{enumerate}
\end{enumerate}
\end{remark}

Our interest in these rings is justified by the following version of Newton-Puiseux-Eisenstein Theorem, proved in subsection \ref{ssec:Newton-Puiseux-Eisenstein}:

\begin{theorem}[Newton-Puiseux-Eisenstein]\label{thm:NewtonPuiseux}
Let $P(\x,y) \in \C \lb \x\rb [y]$ be a monic polynomial. There exists an integral homogeneous element $\g$, and a homogeneous polynomial $h(\x)$, such that $P(\x,y)$ factors as a product of degree $1$ monic polynomials in $y$ with coefficients in $\PP_h\lb \x, \g\rb$.
\end{theorem}

The following result is an easy, but convenient, reformulation of Theorem \ref{thm:NewtonPuiseux}:
\begin{corollary}[Newton-Puiseux-Eisenstein factorization]\label{cr:NewtonPuiseux}
Let $P\in\C\lb \x\rb[y]$ be a monic polynomial. Then, there is a homogenous polynomial $h$ and  integral homogenous elements $\g_{i,j}$, such that $P$ can be written as 
\begin{equation}\label{key_factorization}
P(\x,y)=\prod\limits_{i=1}^s Q_i, \quad \text{ and } \quad  Q_i=\prod\limits_{j=1}^{r_i} (y-\xi_i(\x,\g_{i,j}))
\end{equation}
where 
\begin{enumerate}
\item[(i)] the $Q_i\in \PP_h\lb \x\rb [y]$ are irreducible in $\wdh V_\nu[y]$, 
\item[(ii)] for every $i$, there are $A_{i,k}(\x)\in\PP_h(\!(\x)\!)$, for $0\leq k\leq k_i$ such that
$$\xi_i(\x,\g_{i,j})=\sum_{k=0}^{k_i}A_{i,k}(\x)\g_{i,j}^k$$
\item[(iii)]
for every $i$, the $\g_{i,j}$ are distinct conjugates of an homogeneous element $\g_i$, that is, roots of its minimal polynomial $\Gamma_i$ over $\C(\x)$.
\end{enumerate}
\end{corollary}

\begin{proof} The first equality is a direct consequence of Theorem \ref{thm:NewtonPuiseux}. Fix $Q$ an irreducible factor of $P$ in $\PP_h\lb \x\rb [y]$ and let $\Frac(\PP_h\lb \x\rb) \hookrightarrow \K$ be a normal field extension containing all the roots of $Q$. Because $Q$ is irreducible, for every $j$, there is $\t_j\in \Aut\left(\K/\Frac(\PP_h\lb \x\rb)\right)$ such that $\t_j(\xi_1)=\xi_j$.  Now, seeing $\xi_1$ as an element of $\Frac(\PP_h\lb \x \rb)[\gamma]$, this gives $\t_j(\xi_1(\x,\g))=\xi_1(\x,\t_j(\g))$. But since $\C(\x)\subset \Frac(\PP_h\lb \x \rb)$, $\t_j(\g)$ is also a root of the minimal polynomial of $\g$ over $\C(\x)$.
\end{proof}

\subsubsection{Projective Convergent rings}

In order to prove Proposition \ref{prop:SemiGlobalLessPrecise}, we will show that if $P$ admits a convergent factor after a sequence of blowing ups, then a certain number of the polynomials $Q_i$ in equation \eqref{key_factorization} are ``convergent''. We start by making the latter notion precise, that is, we introduce a subring $\PP_h\{ \x\} $ of $\PP_h\lb \x\rb$ which formalizes the notion of convergence in $\PP_h\lb \x\rb$. More precisely:

\begin{definition}[Projective convergent rings]
Let $h$ be a homogeneous polynomial. We denote by $\PP_h\{\x\}$, the subring of $\PP_h\lb \x\rb$ characterized by elements $A\in \PP_h\lb \x\rb$ such that 
\begin{equation}\label{eq:ElementConvRing}
\sum_{k \geq 0} a_k(\x) \in \C \{ \x\}, \quad \text{ where } \quad A =\sum_{ k\geq 0} \frac{a_k(\x)}{h^{\a k+\b}}  .
\end{equation}
\end{definition}

{Note that it is not clear that $\PP_h\{ \x\}$ is well-defined, since the characterization of its elements seems to depend on the choice of the representation of $A\in \PP_h\lb \x\rb$ in power series, which is not unique. Proposition \ref{representative}, whose proof we postpone to subsection \ref{ssec:OnConvProjRings}, addresses this point and shows that $\PP_h\{ \x\}$ is well-defined.}

\begin{remark}\label{def:convergent}
We recall that a power series $f = \sum_{\alpha \in \mathbb{N}^n} f_{\alpha} \x^{\alpha}$,
is convergent, if and only if there exists $A,\,B>0$ such that the following inequalities hold:
\[
\forall \alpha \in \mathbb{N}^n,\ \ \ |f_{\alpha}| \leq A \cdot B^{|\alpha|}.
\]
Moreover, if we expand $f$ as 
$$f:=\sum\limits_{\k\in \N} f_k(\x)
$$
where $f_{k}(\x)$ are homogeneous polynomials of degree $k$, then $f \in \C\{\x\}$ if, and only if, there exists a compact neighbourhood $K$ of the origin and constants $A,\, B>0$ such that:
$$
\sup\limits_{\z\in K} |f_k(\z)|\leqslant AB^{|k|} , \forall k \in \mathbb{N}.
$$
\end{remark}

\begin{definition}
For $f$ analytic on some compact set $K\subset \C^n$, we set
$$\|f\|_K:=\sup_{\z\in K}|f(\z)|.$$
\end{definition}

\begin{proposition}[Independence of the representative]\label{representative}
Let $h_1$, $h_2$ be homogeneous polynomials and consider elements {$A_1=\sum_{k\geq 0}\frac{a_{1,k}(\x)}{h_1^{\a_1k+\b_1}} \in \PP_{h_1}\lb \x\rb$} and {$A_2=\sum_{k\geq 0}\frac{a_{2,k}(\x)}{h_2^{\a_2k+\b_2}} \in \PP_{h_2}\lb \x\rb$} such that $A_1 = A_2$ when they are considered as elements of $\widehat{V}_{\nu}$. Then {$\sum_{k\geq 0}a_{1,k}(\x)\in\C\{\x\}$ if and only if $\sum_{k\geq 0}a_{2,k}(\x)\in\C\{\x\}$.} In particular:
$$\PP_h\{\x\}\cap\C\lb \x\rb =\C\{\x\}.$$
\end{proposition}

The algebraic definition of $\PP_{h}\{\x\}$ captures a crucial geometrical property for this work, namely the ``generic'' convergence of elements of $A \in \PP_{h}\{\x\}$ after a point blowing up. More precisely:

\begin{lemma}[Geometrical characterization of $\mathbb{P}_h\{\x\}$]\label{max} Let $A \in\PP_h \lb \x\rb$ and consider a sequence of point blowing ups $\s : (N,F) \lgw (\mathbb{C}^n,0)$. Let $\pb$ be a point of $F_r^{(1)}$ which is not on the strict transform of $(h=0)$, nor on any other component of $F$. Then $A \in\PP_h \{ \x\}$, if and only if $A_{\pb} := \widehat{\s}^*_{\pb}(A)$ is a convergent power series.
\end{lemma}
\begin{proof}
We start by noting that the definition of $\mathbb{P}_h\{\x\}$ is invariant by linear changes of coordinates in $\x$, c.f. Remark \ref{rk:GeometricalFormalProjectiveRing}(1). {Therefore, since $\pb$ is only on $F_r^{(1)}$, as in Remark \ref{rk:GeometricalFormalProjectiveRing}(2.a), there exists a system of local coordinates $(v,w_2,\ldots, w_n)$ centered at $\pb$ such that:
\[
x_1= v, \quad x_i = vw_i \text{ for } i>1.
\]
and, if $A \in\PP_h \lb \x\rb$, we obtain:
\[
A_{\pb}:= \wdh\s^*_{\pb}(A) := \sum_{k\geq 0} \frac{a_k(1,\w)}{h(1,\w)^{\a k +\b}} v^k.
\]
 In particular, as $\pb$ does not belong to the strict transform of $(h=0)$, $A_{\pb} \in \widehat{\mathcal{O}}_{\pb}$. }\\
 Assume that $A \in \PP_h\{\x\}$ as in \eqref{eq:ElementConvRing}. The degree of $a_k(\x)$ is linear in $k$, say $ak +b$ with $a,\,b\in \mathbb{N}$. Since $\sum a_k(\x)\in \C\{\x\}$, there exists a compact neighbourhood $K$ of $\pb$ and a constant $C>0$ such that $\|a_k(1,\w)\|_K \leq C^k$. Since $\pb$ is not on the strict transform of $(h=0)$, $h(1,\w)\neq 0$, and, up to shrinking $K$, we can suppose that $\inf_{(v,\w) \in K} \|h(1,\w)\| = h_0 >0$, yielding:
\[
\left\|\frac{a_k(1,\w)}{h(1,\w)^{\a k +\b}} \right\|_K \leq \frac{1}{h_0^{\b}} \left(\frac{C}{h_0^{\a}}\right)^k
\]
and we easily conclude that $A_{\pb}$ is a convergent power series.

Next, suppose that $A_{\pb}$ is convergent and let us prove that the formal power series $G:= \sum_{k\geq k_0} a_k(\x) \in \C \lb \x \rb$ is convergent. Indeed, we note that:
\[
B(z,\w) := \sum_{k\geq 0}a_k(1,\w)z^k= h^{\b}(1,\w)\cdot A_{\pb}(h^{\a}(1,\w) \cdot z, \w)\in\C\{z,\w\}.
\]
Now, by definition, the degree of the polynomials $a_k(\x)$ is an affine function in $k$, say $ak+b$ where $a$, $b\in\N$. It therefore follows that $\widehat{\s}_{\pb}^*(G)= z^b\cdot B(z^a, \w)$ is a convergent power series. It now follows from Lemma \ref{lemma:Convergencemodif} that $G$ is convergent, finishing the proof.
\end{proof}

\subsubsection{Formal extensions}
We now turn our attention to the study of the behaviour after blowing up of $A \in \mathbb{P}_h\{\x\}$ at a point $\pb$ in the strict transform of $(h=0)$. We start by studying it formally:

\begin{definition}[Formal extension]
Let $A \in \PP_h \lb \x \rb$ and consider a sequence of point blowing ups $\sg: (N,F) \lgw (\mathbb{C}^n,0)$. Given a point $\pb \in F_r^{(1)}$, we say that $A $ \emph{extends formally} at $\pb$ if the composition $A_{\pb}:= \widehat{\s}^*_{\pb}(A)$ belongs to $\widehat{\mathcal{O}}_{{\pb}}$. Moreover, we say that $A$ \emph{extends analytically} at $\pb$ if $A_{\pb}$ belongs to $\mathcal{O}_{\pb}$.
\end{definition}

Given $A \in \PP_h \lb \x \rb$, we know by Lemma \ref{max} that $A$ extends to every point $\pb \in F_r^{(1)}$ which does not belong to the strict transform of $(h=0)$ or to the intersection of two divisors. But, under the hypothesis of Proposition \ref{prop:SemiGlobalLessPrecise}, the results hold true over every point $\pb \in F_r^{(1)}$, that is:


\begin{theorem}[Semi-global formal extension]\label{thm:SemiGlobalFormal}
Let $P(\x,y)\in\C \lb \x \rb [y]$ be a monic reduced polynomial, and let $P= \prod_{i=1}^s Q_i$ be the factorization of $P$ as a product of irreducible monic polynomials of $\PP_h\lb \x\rb[y]$ given by Theorem \ref{thm:NewtonPuiseux}. Suppose that $n=2$ and let $\sg: (N,F) \lgw (\C^2,0)$ be a sequence of point blowing ups so that, at every point $\pb \in F_r^{(1)}$, the pulled-back discriminant $\widehat{\s}_{\pb}^*(\Delta_P) $ is formally monomial. Then, for every point $\pb  \in F_{r}^{(1)}$, the polynomials $Q_i$ extend formally at $\pb$ to a polynomial which we denote by $\widehat{\s}_{\pb}^*(Q_i)$. Furthermore, the extension is compatible with the factorization of $P$, that is $\prod_{i=1}^r \wdh \s_{\pb}^*(Q_i) = \wdh\s_{\pb}^*(P)$.
\end{theorem}

{The proof of Theorem \ref{thm:SemiGlobalFormal} is postponed to Section \ref{ssec:SemiGlobalFormalExtention}.}

The formal extension property can be combined with analytic continuation type arguments in order to obtain:

\begin{lemma}[Analytic continuation of formal extensions]\label{ext} Let $A\in\PP_h\{\x\}$ and $\sg: (N,F) \lgw (\C^2,0)$ be a sequence of point blowing ups. Suppose that $A$ formally extends at $\pb \in F^{(1)}_r$. Then $A$ extends analytically at $\pb$.
\end{lemma}
\begin{proof}
The result easily follows from Lemma \ref{max} whenever $\pb$ is not in the strict transform of $(h=0)$ or in the intersection of more than one exceptional divisor. In the general case, using the normal forms of Remark \ref{rk:GeometricalFormalProjectiveRing}(2), we can suppose that there exists a coordinate system $(v,w)$ centered at $\pb$ and $c \in \mathbb{Z}_{\geq 0}$ such that:
\[
x_1 = v \cdot w^c, \quad x_2 = v\cdot w^{c+1}
\]
it follows from the hypothesis that there exists polynomials $b_k(w)$ for $k \geq k_0$ such that
\[
A_{\pb}  = \sum_{k \geq k_0} v^k w^{c k} \frac{a_k(1,w)}{h(1,w)^{\a k + \b}} = \sum_{k \geq k_0} v^k \cdot b_k(w) =: B(v,w)
\]
that is, $b_k \cdot h(1,w)^{\a k + \b} = w^{c  k} \cdot a_k(1,w)$. We claim that  $B$ is a convergent power series. Indeed, let us denote by $h(1,w) = w^d u(w)$, where $u(0) \neq 0$, and consider a closed ball $B(\pb,r)$ of radius $1\geq r>0$ where $\inf_{w\in B(\pb,r)} |u(w)| > C$ for some positive $C$.
 
For every polynomial $b(w)$, by the maximum principle, we have
$$\|b(w)\|_{B(\pb,r)}=|b(z)|=\frac{1}{r^d}|b(z)z^d|\leq \frac{1}{r^d}\|b(w)w^d\|_{B(\pb,r)}$$
for some $z\in \C$, $|z|=r$.

Therefore, for every polynomial $b(w)$, we get that:
\[
\begin{aligned}
\| b\|_{B(\pb,r)} \leq \| b \cdot u\|_{B(\pb,r)}\|  u^{-1}\|_{B(\pb,r)} & \leq \frac{1}{r^d} \| b \cdot h(1,w)\|_{B(\pb,r)}\|  u^{-1}\|_{B(\pb,r)}\\
& {\leq \frac{1}{Cr^d}  \| b \cdot h(1,w)\|_{B(\pb,r)}}
\end{aligned}
\]
We apply this inequality $\a k+\b$ times and we combine it with Remark \ref{def:convergent}: there is  a constant $D>0$ such that $\|a_k(1,w)\|_{B(\pb,r)} \leq D^k$ for every $k$. This shows that 
\[
\| b_k \|_{B(\pb,r)} \leq (C r^d)^{-(\alpha k + \beta)} D^{k}, \quad \forall k \geq 0.
\]
\end{proof}

\subsubsection{Local-to-Semi-global convergence of factors}
As a consequence of the previous discussion, if $P \in \mathbb{C} \lb \x \rb[y]$ is a monic polynomial and $Q \in \mathbb{P}_h\{\x\}[y]$ is a factor of $P$ then, under the conditions of Proposition \ref{prop:SemiGlobalLessPrecise}, $Q_{\pb}:= \widehat{\s}_{\pb}^*(Q)$ is a convergent factor of $P_{\pb}:= \widehat{\s}_{\pb}^*(P)$ at every point $\pb \in F^{(1)}_r$. It remains to show that there exists such a factor $Q$. This is the subject of the next result:

\begin{theorem}[Local-to-Semi-global convergence of factors]\label{thm:IzumiWalsh}
Let $P \in \mathbb{C} \lb \x \rb[y]$ be a monic reduced polynomial, and $h$ be a homogeneous polynomial as in Theorem \ref{thm:NewtonPuiseux}. Suppose that $n=2$ and let $\sg: (N,F) \lgw (\C^2,0)$ be a sequence of point blowing ups. Suppose that at every point $\pb \in F_r^{(1)}$, the pulled-back discriminant $\widehat{\s}_{\pb}^*(\Delta_P) $ is formally monomial. Suppose that there exists a point $\pb \in F_r^{(1)}$ such that $P_{\pb} := \widehat{\s}_{\pb}^*(P)$ admits a convergent factor. Then $P$ admits a non-constant factor $Q \in \PP_h\{\x\}[y]$ such that either $P/Q$ is constant, or $\widehat{\s}_{\pb}^*(P/Q)$ admits no convergent factor at every point $\pb \in F_r^{(1)}$. 
\end{theorem}

{The proof of Theorem \ref{thm:IzumiWalsh} is postponed to Section \ref{ssec:LocalToSemiGlobal}.}

\subsubsection{Proof of Proposition \ref{prop:SemiGlobalLessPrecise}}
We have now collected all necessary ingredients in order to prove the following result, which immediately yields Proposition \ref{prop:SemiGlobalLessPrecise}:

\begin{corollary}[Semi-global extension of convergent factors]\label{cor:SemiGlobal}
Let $P(\x,y)\in \C\lb \x \rb [y]$ be a monic reduced polynomial, and let $P= \prod_{i=1}^s Q_i$ be the factorization of $P$ as a product of irreducible monic polynomials of $\PP_h\lb \x\rb[y]$ given in Theorem \ref{thm:NewtonPuiseux}. Let $\sg: (N,F) \lgw (\C^2,0)$ be a sequence of point blowing ups, and suppose that:
\begin{itemize}
\item[(i)] At every point $\pb \in F_r^{(1)}$, the pulled-back discriminant $\wdh \s_\pb^*(\Delta_P)$ is formally monomial.
\item[(ii)] There exists $\pb_0 \in F_r^{(1)}$ where $P_{\pb_0}:= \widehat{\s}_{\pb_0}^*(P)$ admits a convergent factor.
\end{itemize}
Then, up to re-indexing the polynomials $Q_j$, there exists an index $t\geq 1$, a neighbourhood $U^{(1)}_r$ of $F^{(1)}_r$ and analytic polynomials $\widetilde{q}_j \in \mathcal{O}_{U^{(1)}_r} [y]$ for $j=1,  \ldots,t$ such that, for every point $\pb \in F_{r}^{(1)}$, we have that
\[
T_{\pb}(\widetilde{q}_j) =  \widehat{\s}_{\pb}^*(Q_j).
\]
Finally, denote by $\widetilde{q} = \prod_{i=1}^t \widetilde{q}_i$. Then, at every point $\pb \in F^{(1)}_r$, the quotient polynomial $P_{\pb}/T_{\pb}(\widetilde{q})$ is either constant, or does not admit a convergent factor.
\end{corollary}
\begin{proof}
Consider the factorization \eqref{key_factorization} given by Corollary \ref{cr:NewtonPuiseux}:
\[
P(\x,y) = \prod_{i=1}^s Q_i(\x,y), \quad \text{ where }Q_i(\x,y) \in \PP_h\lb \x\rb[y], \, i=1,  \ldots, s
\] 
From Theorem \ref{thm:SemiGlobalFormal}, we know that $Q_i$ extends formally to every point $\pb$ of $F^{(1)}_{r}$ for every $i=1,  \ldots,s$. It follows from Theorem \ref{thm:IzumiWalsh} that there exists $t\leq s$ such that $Q_i \in \PP_h\{\x\}[y]$ for every $1\leq i \leq t$. Now, by Lemma \ref{ext}, $Q_i$ extends analytically at every point $\pb \in F^{(1)}_r$ for every $i\leq t$. This implies that there exists a polynomial $\widetilde{q}_i$ defined in a neighbourhood $U_r^{(1)}$ of $F_r^{(1)}$ which formally coincides with $ \widehat{\s}_{\pb}^* (Q_i)$ at every point $\pb \in F_r^{(1)}$. Finally, it is immediate from the above construction that $t$ is constant along $F_r^{(1)}$, which implies that $P / \prod_{i=1}^t Q_i$ admits no further convergent factors, finishing the proof.
\end{proof}

\subsection{Newton-Puiseux-Eisenstein Theorem}\label{ssec:Newton-Puiseux-Eisenstein}
The proof of Theorem \ref{thm:NewtonPuiseux} is done via an induction argument in terms of the degree of $P$. Note that the case of $\deg(P)=1$ is trivial (with $h=1$ and $\g =1$), while $\mbox{degree}(P)=2$ still admits an elementary proof:

\begin{remark}[Elementary proof when $\deg(P)=2$]\label{rk:ThmNewtonDeg2}
If $\deg(P)=2$, then we can write $P(\x,y) = P_0(\x) + P_1(\x) y + y^2$ and obtain an explicit formula for their roots. More precisely, the discriminant $\D_P \in \C \lb \x\rb$, and can be written as:
\[
\D_P:= \sum_{k \geq k_0} \delta_{k}(\x)
\]
where $\delta_{k_0}= \ini(\D_P)$ and $\delta_k(\x)$ are homogeneous elements for every $k\geq k_0$. It follows that the roots of $P(\x,y)=0$ are given by:
\[
\xi_{\pm} = - \frac{P_1}{2} \pm \frac{\sqrt{\D_P}}{4} = - \frac{P_1}{2} \pm \frac{\sqrt{\delta_{k_0}(\x)}}{4} \cdot \sqrt{1 +\sum_{k > k_0} \frac{\delta_{k}(\x)}{\delta_{k_0}(\x)} }
\]
and we can easily verify that $\xi_{\pm} \in \PP_h\lb \x,\g \rb$, where $h := \ini(\D_P) = \delta_{k_0}(\x)$ and $\g := \sqrt{\ini(\D_P)} = \sqrt{\delta_{k_0}(\x)}$.

In general, it is not possible to choose $h=\ini(\D_P)$ as claimed (but not proven) in \cite{To2}.
\end{remark}

As we will see, a proof by induction on the degree of $P$ demands manipulations of several integral elements at the same time. We start by a couple of useful results about homogeneous elements:

\begin{lemma}[Degree compatibility under extensions]\label{weighted}
Let $P(\x,z_1,  \ldots, z_{k+1})\in\C[\x,z_1,  \ldots, z_{k+1}]$ be a $(\omega_1,  \ldots, \omega_{k+1})$-weighted homogeneous polynomial. Let $\g_i$ be  homogeneous elements of degree $\omega_i$, for $i\leq k$, such that 
\[
\deg_{z_{k+1}}(P(\x,\g_1,  \ldots, \g_k, z_{k+1}))\geq 1.
\]
Then any element $\g_{k+1}$ of an algebraic closure of $\C(\x)$ such that
$$
P(\x,\g_1,  \ldots, \g_k,\g_{k+1})=0
$$
is a homogeneous element of degree $\omega_{k+1}$.
\end{lemma}
\begin{proof}
The proof is made by induction on $k$. For $k=0$, the result follows directly from the definition. Let now $k\geqslant 1,$ and assume that the result is proved for $k-1$. Let $P_k(\x,z_k)$ be a nonzero irreducible $\omega_k$-weighted homogeneous polynomial such that $P_k(\x,\g_k)=0$. We set
\[
\wdt P(\x,z_1,  \ldots, z_{k-1},z_{k+1}):=\operatorname{Res}_{z_k}(P_k(\x,z_k),P(\x,z_1,  \ldots, z_{k+1}))
\]
where $\operatorname{Res}_{z_k}$ denotes the resultant of two polynomials seen as polynomials in the indeterminate $z_k$. Then $\wdt P(\x,z_1,  \ldots, z_{k-1},z_{k+1})$ is a $(\omega_1,  \ldots, \omega_{k-1},\omega_{k+1})$-weighted homogeneous polynomial, and we have  $\wdt P(\x,\g_1,  \ldots, \g_{k-1},\g_{k+1})=0$. Because $P_k(\x,z_k)$ is irreducible, $\gcd(P_k(\x,z_k),P(\x,\g_1,  \ldots, \g_{k-1}, z_k, z_{k+1}))$ is either $1$ or $P_k(\x,z_k)$, but
\[
\deg_{z_{k+1}}(P(\x,\g_1,  \ldots, \g_k, z_{k+1}))\geq 1,
\] thus $P_k(\x,z_k)$ does not divide $P(\x,\g_1,  \ldots, \g_{k-1}, z_k, z_{k+1})$, implying that  $P_k(\x,z_k)$ and $P(\x,\g_1,  \ldots,  \g_{k-1}, z_k, z_{k+1})$ are coprime and $\wdt P(\x,\g_1,  \ldots, \g_{k-1},z_{k+1})$ is nonzero. We conclude by induction.
\end{proof}

\begin{corollary}[Compatibility between homogenous elements]\label{cor_weighted}
If $\g_1$ and $\g_2$ are homogeneous elements, respectively of degrees $\omega_1$ and $\omega_2$, then
\begin{itemize}
\item[i)] for every $q\in\Q_{> 0}$, $\g_1^q$ is  homogeneous  of degree $q\omega_1$,
\item[ii)] $\g_1\g_2$ is  homogeneous  of degree $\omega_1+\omega_2$,
\item[iii)] if $\omega_1=\omega_2$, $\g_1+\g_2$ is  homogeneous  of degree $\omega_1$.
\end{itemize}
\end{corollary}

\begin{proof}
For $i=1$ and $2$, let us denote by $P_i(\x,z_i)$ a  nonzero $\omega_i$-weighted homogeneous polynomial with $P_i(\x,\g_i)=0$. In order to prove the corollary, we apply Lemma \ref{weighted} to $P(\x,z_1,z_2)=z_1^a-z_2^b$ if $q=a/b$ in case i), 
$P(\x,z_1,z_2,z_3)=z_3-z_1z_2$ in case ii), and  $P(\x,z_1,z_2,z_3)=z_3-z_1-z_2$ in case iii).
\end{proof}

We now turn to the proof of a key result in order to reduce the study from multiple homogeneous elements to a single one:

\begin{lemma}[Existence of a primitive integral homogeneous element]\label{PET}
Let $\gam =(\g_1$,   \ldots, $\g_r)$ be homogeneous elements. Then there exists an integral homogeneous element $\g_0$ such that
$$\wdh V_{\nu,\gam}\subset\wdh V_{\nu,\g_0}.$$
\end{lemma}

\begin{proof}
For simplicity we assume $r=2$. The general case is proven in a similar way.\\
Let us denote by $\omega_i$ the degree of $\g_i$, and write $\omega_i=\frac{a_i}{b_i}$ where $a_i$, $b_i\in \N$. We set $\g'_1:=\g_1^{1/a_1b_2}$ and $\g'_2:=\g_2^{1/a_2b_1}$. Therefore $\g_1'$ and $\g_2'$ are homogeneous elements of the same degree $\omega=\frac{1}{b_1b_2}$. By the Primitive Element Theorem, there exists $c\in\C$ such that
$$
\C(\x)(\g_1',\g_2')=\C(\x)(\g_1'+c\g_2').
$$
Therefore $\wdh V_{\nu,\g_1',\g_2'}=\wdh V_{\nu,\g'_1+c\g'_2}$. Since $\wdh V_{\nu,\g_1,\g_2}\subset \wdh V_{\nu,\g_1',\g_2'}$, this proves the existence of a homogeneous element $\g$ such that
$$\wdh V_{\nu,\g_1,\g_2}\subset\wdh V_{\nu,\g_0}.$$
Thus, we only need to prove that $\g_0$ may be chosen to be an integral homogeneous element. Let us assume that 
\begin{equation}\label{eq_lem}c_0(\x)\g_0^d+c_1(\x)\g_0^{d-1}+\cdots+c_d(\x)=0\end{equation}
where the $c_k(\x)$ are homogeneous polynomials of degree $\omega k+p$, where $\omega\in\Q_{\geq 0}$ and $p\in\N$. Let $\g'_0:=c_0(\x)\g_0$. Then, by multiplying \eqref{eq_lem} by $c_0(\x)^{d-1}$, we have
$${\g'_0}^d+c_1(\x){\g'_0}^{d-1}+\cdots+ c_k(\x)c_0(\x)^{k-1}{\g'_0}^{d-k}+\cdots+ c_0(\x)^{d-1}c_d(\x)=0.$$
This shows that $\g'_0$ is a integral homogeneous element of degree $p+\omega$. This proves the lemma since $\wdh V_{\nu,\g_0}=\wdh V_{\nu,\g'_0}$.
\end{proof}

Finally, we are ready to prove Theorem \ref{thm:NewtonPuiseux}. We divide the proof in two main steps (which combined immediately yield Theorem \ref{thm:NewtonPuiseux}), which are interesting on their own. The first step shows that we can concentrate our discussion to the rings $\widehat{V}_{\nu,\g}$ instead of $\PP_{h} \lb \x,\g \rb$. The proof of this first step follows from arguments in the spirit of Tougeron's implicit Function theorem \cite[Chapter 3, Theorem 3.2]{To}:

\begin{proposition}[Newton-Puiseux-Eisenstein: Step I]\label{prop:NP1}
Let $P(\x,y) \in \C \lb \x\rb [y]$ be a monic polynomial. Suppose that there exists an integral homogeneous element $\g$ such that $P(\x,y)$ factors as a product of degree $1$ polynomials in $y$ with coefficients in $\widehat{V}_{\nu,\gamma}$. Then there exists an homogeneous polynomial $h$ such that $P(\x,y)$ factors as a product of degree $1$ polynomials in $y$ with coefficients in $\PP_h\lb \x, \g\rb $.
\end{proposition}

\begin{proof}
Let $\xi$ be a root of the polynomial $P(\x,y)$; by hypothesis $\xi \in \widehat{V}_{\nu,\gamma}$, so it has the following form:
 $$
\xi=\sum_{i=1}^d A_i \g^i, \text{ where } A_i \in \Frac\left(\widehat{V}_{\nu}\right)\text{ and } \nu(A_i\g^i)\geqslant 0.
$$
Let us denote by $\g=:\g_1, \cdots, \g_d$ the conjugates of $\g$ over $\C(\x)$. Then $d$ is the degree of the minimal polynomial of $\g$ over $\C(\x)$. We denote by $M$ the Vandermonde matrix whose coefficients are the $\g_i^j$, for $1\leq i,j\leq d$. Note that:
\[
\xi_j:= \sum_{i=1}^d A_i \g^i_j, \quad j=1,  \ldots,d
\]
are also roots of $P(\x,y)$. We now have that: 
$$
M \cdot \begin{bmatrix} A_1\\ \vdots \\ A_d\end{bmatrix}=\begin{bmatrix}\xi_1\\ \vdots \\ \xi_{d}\end{bmatrix},
\text{ therefore } 
\begin{bmatrix} A_1\\ \vdots \\ A_d\end{bmatrix}=M^{-1}\begin{bmatrix}\xi_1\\ \vdots \\ \xi_{d}\end{bmatrix}.
$$
Because the entries of $M$ are algebraic over $\C(\x)$, the entries of $M^{-1}$ are also algebraic over $\C(\x)$. Thus the $A_i$ are algebraic over $\C(\!(\x)\!)$. Now, we claim that if $A\in \wdh V_\nu$ is algebraic over $\C(\!(\x)\!)$, then there exists $h$ (depending on $A$) such that $A\in \PP_{h}(\!( \x)\!)$. The Proposition easily follows from the Claim, by using the fact that:
\[
\PP_{h_1}\lb \x, \g \rb \cup \PP_{h_2}\lb \x, \g \rb  \subset \PP_{h}\lb \x, \g \rb \text{ where } h=\lcm(h_1,h_2),
\]
and noting that there are a finite number of roots $\xi$ of $P$, each one of them with a finite number of coefficients $A_i$. We now turn to the proof of the Claim. 

Let $A\in \Frac\left(\wdh V_\nu\right)$ be algebraic over $\C(\!(\x)\!)$. Note that if $x_1 A\in\PP_h(\!( \x)\!) $, then $A\in\PP_{x_1 h}(\!( \x)\!)$ and, if $A$ is algebraic, then $x_1A$ is algebraic too. So, up to multiplying $A$ by a large power of $x_1$, we may assume that $\nu(A)>0$. We write $A=\sum_{i>0} \frac{a_i(\x)}{b_i(\x)}$, where $a_i,b_i$ are homogeneous polynomials such that $\nu(a_i)-\nu(b_i)=i$, and we denote by $P(\x,z)$ the minimal polynomial of $A$. Now, we replace $x_i$ by $tx_i$ for every $i$. Therefore, $Q(z):=P(t\x,z)$ (where we leave the dependency in $\x$ and $t$ implicit) is separable and $Q(B)=0$, where $B:=\sum_{i>0} A_i(\x)t^i$. Note that $Q$ may not be irreducible over $\C(\!(t,\x)\!)[z]$ but it is separable because $P$ is, as an irreducible polynomial over a field of characteristic zero. We set $e:=\ord_t\left(\frac{\partial Q}{\partial z}(B)\right)>0$, where $\ord_t(\cdot)$ denotes the order of a series with respect to the variable $t$. Note that $\frac{\partial Q}{\partial z}(B)\neq 0$, since $Q$ is separable. We set $\ovl B:=\sum_{i\leq 2e+1} A_i(\x)t^i$ and $y=t^{e+1}v$. By Taylor expansions in $z$ centered at $\overline{B}$ and $B$, we have
$$
Q(t^{e+1}v+\ovl B)=Q(\ovl B)+\frac{\partial Q}{\partial z}(\ovl B) t^{e+1}v+\widetilde{Q_1}(t,\x,v)t^{2e+2}v^2,
$$ and $$
Q(t^{e+1}v+ B)=\frac{\partial Q}{\partial z}(B) t^{e+1}v+\widetilde{Q_2}(t,\x,v)t^{2e+2}v^2,
$$ where $\widetilde{Q_1},\widetilde{Q_2}\in \C(\!(\x,t)\!)[v]$.
Writing $B=\ovl B + t^{e+1}\sum\limits_{i>2e+1}\frac{a_i(\x)}{b_I(\x)}t^{i-e-1}$, this gives
$$
\ord_t\left(\frac{\partial Q}{\partial z}(B)\right)=\ord_t\left(\frac{\partial Q}{\partial z}(\ovl B)\right)=e,
$$ and $\ord_t(Q(\ovl B))\geqslant 2e+1$.

Now, note that every term of $\ovl B$ is of the form $\frac{a_i}{b_i}t^i$, where $a_i(\x),b_i(\x)$ are homogeneous polynomials and $\nu(a_i)-\nu(b_i)=i$. Furthermore, since $Q\in \C(\!(t\x)\!)[z]$, every term of $Q(\ovl B)$ is the form $\frac{f_i}{g_i}t^i$, where, again, $f_i(\x),g_i(\x)$ are homogeneous polynomials and $\nu(f_i)-\nu(g_i)=i$. Finally, since $\overline{B}$ is a finite sum with homogeneous denominators and $Q\in \C(\!(\x t)\!)[z]$, there exists an homogeneous polynomial $b(\x)$ such that $b(\x)Q(t^{e+1}v+\ovl B)\in \C\lb \x,t\rb [v]$. 
Therefore, dividing the equation $b(\x)Q(t^{e+1}v+\ovl B)=0$ by $t^{2e+1}$, we obtain
\begin{equation}\label{Tay}
R_0(t,\x)+R_1(t,\x)v+\widetilde{R}(t,\x,v)v^2=0,
\end{equation}
with $\ord_t(R_1)=0$, so we can write $R_1(t,\x)= g(\x) +t\widetilde{R}_1(t,\x)$ where $g(\x)$ is a nonzero homogeneous polynomial of degree $\geqslant 2e+1$. Moreover
$$
v=\sum_{i>2e+1}A_i(\x)t^{i-e-1}
$$
is a solution of \eqref{Tay} and, thus, $t$ divides $R_0(t,\x)$, that is $
R_0(t,\x)=t \widetilde{R}_0(t,\x)$. We set
$$
t=g(\x)^2s \text{ and } v=g(\x) w.
$$
Then, dividing \eqref{Tay} by $g(\x)^2$ we get
\begin{equation}\label{Tay2}
s\widetilde{R}_0(g(\x)^2s,\x)+\left[1+g(\x) s\widetilde{R}_1(g(\x)^2s,\x)\right]w+\widetilde{R}(g(\x)^2s,\x,g(\x) w)w^2=0.
\end{equation}
By the implicit function Theorem, \eqref{Tay2} has a unique solution $w(s,\x)\in\C\lb s,\x\rb $, $w(s,\x)=\sum_{i\in\N} w_i(\x) s^i$. Therefore
$$
\sum_{i\in\N}\frac{a_i(\x)}{b_i(\x)} t^i=\sum_{i\leq 2e+1} \frac{a_i(\x)}{b_i(\x)} t^i+\sum_{i\in\N}  \frac{w_i(\x)g(\x)}{g(\x)^{2i}}t^{i+e+1}.
$$
Hence, if $h$ denotes the product of $g^2$ with the $b_i$ for $i\leq 2e+1$, we have that $A(\x)\in\PP_h\lb \x\rb $. This finishes the proof.
\end{proof}

We now turn to the second step of the proof of Theorem \ref{thm:NewtonPuiseux}, which is actually a more general statement than Theorem \ref{thm:NewtonPuiseux}, and easier to prove by induction:

\begin{proposition}[Newton-Puiseux-Eisenstein: Step II]\label{prop:NP2}
Let $\g$ be an integral homogeneous element and $P(\x,y) \in \wdh V_{\nu,\g} [y]$ be a monic polynomial. Then there exists an integral homogeneous element $\g'$ such that $P(\x,y)$ factors as a product of degree $1$ polynomials in $y$ with coefficients in $\wdh V_{\nu,\g'}$.
\end{proposition}
\begin{proof}
We prove the Proposition by induction on the degree $\mbox{deg}_y(P)= d$. Fixed $\g$, note that the result is trivial when $d=1$, and suppose that the Proposition is proved for every homogeneous element $\g$ and every polynomial in $ \wdh V_{\nu,\g} [y]$ with degree $d'<d$. We fix an homogeneous element $\g$, and let $P(\x,y) \in \wdh V_{\nu,\g} [y]$ be a monic polynomial of degree $d$. Let us write 
$$
P(y)=y^d+a_1y^{d-1}+\cdots+a_d, \quad a_i \in \wdh V_{\nu,\g}, i=1,  \ldots,d.
$$
Up to replacing $y$ by $y-a_1/d$, we can assume that $a_1=0$. Let $k_0\in\{2,  \ldots,d\}$ be such that
\begin{equation}\label{inequality}\frac{\nu(a_{k_0})}{k_0}\leq \frac{\nu(a_k)}{k} \ \ \forall k\in\{2,  \ldots, d\}.\end{equation}
Let $\g_1$ be such that $\g_1^{k_0}-\ini_\nu(a_{k_0})=0$. By Corollary \ref{cor_weighted}, $\g_1$ is a homogeneous element. By \eqref{inequality}, we can write $a_k=b_k\g^k_1$ for every $k\in\{2,  \ldots,d\}$, where $b_k\in \wdh V_{\nu,\g,\g_1}$.  We remark that
$$P(\g_1y)=\g_1^d(y^d+b_2y^{d-2}+\cdots+b_d).$$
 We set $Q(y)=y^d+b_2y^{d-2}+\cdots+b_d$. We denote by $\m$, the maximal ideal of $\wdh V_{\nu,\g,\g_1}$, i.e. the set of elements $a$ such that $\nu(a)>0$, and we denote by $\ovl Q(y)$ the image of $Q(y)$ in $\wdh V_{\nu,\g,\g_1}/\m[y]\simeq\C[y]$. Then $\ovl Q(y)$ is not of the form $(y-c)^d$ since $b_1=0$ and $b_{k_0}\notin \m$. Therefore we may factor $\ovl Q(y)$ as a product $R_1(y)R_2(y)$, where the $R_i(y)$ are monic polynomials with complex coefficients, and $R_1(y)$ and $R_2(y)$ are coprime. By Hensel's Lemma, this factorization of $\ovl Q(y)$ lifts as a factorization of $Q(y)$: $Q(y)=Q_1(y)Q_2(y)$ where $Q_i(y)\mod \m= R_i(y)$ for $i=1$, 2. Now, by Lemma \ref{PET}, we know that there exists an homogeneous element $\g'$ such that $Q_1(y)$ and $Q_2(y)$ are polynomials with coefficients in $\widehat{V}_{\nu,\g'}$, and each one of them has degree $d_i<d$ for $i=1,2$. We conclude by induction and by Lemma \ref{PET} once again.
\end{proof}

\begin{remark}
The classical Newton-Puiseux Theorem asserts that a monic polynomial $P(y)$ with coefficients in $\k\lb \x\rb$, where $x$ is a single indeterminate and $\k$ is a characteristic zero field, has its roots in $\k'\lb \x^{1/d}\rb$ for some $d\in\N^*$ and $\k\lgw \k'$ a finite field extension.

Fix $i\in\{1,  \ldots, n\}$ and let $a(\x)$, $b(\x)$ be nonzero homogeneous polynomials. We have
$$
\frac{a(\x)}{b(\x)}=x_i^{\deg(a)-\deg(b)}\frac{a\left(\frac{x_1}{x_i},  \ldots, \frac{x_n}{x_i}\right)}{b\left(\frac{x_1}{x_i},  \ldots, \frac{x_n}{x_i}\right)}.
$$
This proves that $\wdh V_\nu$ is isomorphic to $\K\lb x_i\rb $ where  $\K\simeq \C\left(\frac{x_1}{x_i},  \ldots, \frac{x_n}{x_i}\right)$.

 Moreover, if $\g$ is an integral homogeneous element, we can consider the coefficients of its minimal polynomial as elements of $\K\lb x_1\rb $ by the previous remark. In this case, it is straightforward to check that $\g$ is identified with $cx_1^{1/d}u(x_1)$ where $c$ is algebraic over $\K$, $d\in\N^*$ and $U(x_1)$ is a unit of $\K\lb x_1\rb$.
Therefore, Proposition \ref{prop:NP2} is an extension of the classical Newton-Puiseux Theorem for univariate power series since $\PP_h\lb \x\rb\subset \wdh V_\nu$.

On the other hand, the classical Eisenstein Theorem \cite{Ei} is the following statement:
\begin{theorem*}[{Eisenstein Theorem \cite{Ei}}] Given  a univariate power series $f(\x)=\sum_{k\in\N}f_k\x^k\in\Q\lb \x\rb$  that is algebraic over $\Z[\x]$, there exists a nonzero natural number $b$ such that
$$
\forall k\in\N,\ \ b^{k+1}f_k\in \Z.
$$
\end{theorem*}
This shows that the $f_k$ can be written as $\frac{a_k}{b^{k+1}}$ where $a_k$ are integers. Therefore, Proposition \ref{prop:NP2} is a natural extension of Eisenstein Theorem to the situation where $\Z$ is replaced by $\C\lb \x\rb$.
\end{remark}

\subsection{On convergent projective rings}\label{ssec:OnConvProjRings} In this subsection we prove Proposition \ref{representative}. For this we need a complementary inequality to the following  inequality:
$$
\forall p, q\in\C[\x],\ \ |pq|\leq |p| |q|
$$
where $|p|$ is a well chosen norm. The investigation for such complementary inequalities is an old problem that goes back to K. Mahler \cite{Ma}. The complementary inequality that we need is given in Corollary \ref{cor_inequalities} given below. Such an inequality follows from results proven in this field. But for the sake of completeness we chose to include a proof of this inequality. Moreover our proof seems to be new and may be of independent interest. It is based on the Weierstrass division Theorem that is well know in local analytic geometry. 
Therefore we begin by defining the needed norms:

\begin{definition}\label{lp}Let $\rho>0$. We denote by $\C\{\x\}_{\rho}$ the subring of $\C\{\x\}$ of series $f(\x)=\sum_{\a\in\N^n}f_\a \x^\a$ such that
$$|f|_\rho:=\sum_{\a\in\N^n}|f_\a|\rho^{|\a|}<\infty.$$
The map $f\in\C\{\x\}_{\rho}\lgw |f|_{\rho}$ is an absolute value that makes $\C\{\x\}_{\rho}$ a Banach algebra.
\end{definition}

\begin{proposition}
Let $h(\x)\in\C\{\x\}$. Then there is $\rho>0$  and $C>0$ such that for every 
$ a(\x)\in\C\{\x\}_{\rho}$
$$ a(\x)\in h(\x)\C\lb \x\rb \Longrightarrow\left[
\dfrac{a(\x)}{h(\x)}\in\C\{\x\}_{\rho}\text{ and }
\left|\dfrac{a(\x)}{h(\x)}\right|_{\rho}\leq C|a(\x)|_{\rho}\right]
$$
\end{proposition}

\begin{proof}
The proof is essentially the proof given in \cite{To} of the Weierstrass division Theorem. First we fix $\rho>0$ small enough to ensure that $h(\x)\in\C\{\x\}_{\rho}$. If $h(\x)$ is a unit, up to shrinking $\rho$, the result is straightforward.

If $\ord(h(\x))=1$, there is an analytic diffeomorphism $\phi:\C\{\x\}\lgw\C\{\x\}$ sending $h(\x)$ onto $x_1$.  For every $\rho$, such a diffeomorphism sends $\C\{\x\}_\rho$ onto  $\C\{\x\}_{\rho'}$ for some $\rho'>0$. Therefore we may assume that $h(\x)=x_1$, and in this case the result easily follows, with $C=\rho^{-1}$. 

Therefore, we assume that $\ord(h(\x))\geq 2$. Up to a linear change of coordinates, we may assume that $h(\x)$ is $x_n$-regular of order $d\geq 2$, that is, $h(0,x_n)=x_n^du(x_n)$ with $u(0)\neq 0$. Note that we may identify the $\C$-vector space of polynomials $r(\x)\in\C\{\x'\}_{\rho}[x_n]$ of degree at most $d$ in $x_n$ with $(\C\{\x'\}_{\rho})^d$. We now define the following $\C$-linear maps:
$$
L_1, L_2: \C\{\x\}_{\rho}\times(\C\{\x'\}_{\rho})^d\lgw \C\{\x\}_{\rho}
$$
by
$$
L_1(q,r)=qh+r\text{ and } L_2(q,r)=qx_n^d+r.
$$
Set $L_3=L_2-L_1$. We remark that $L_2$ is a linear isomorphism (every series $a(\x)$ can be written in a unique way as $a(\x)=x_n^da_1(\x)+a_2(\x)$ where $a_2(\x)\in\C\{\x'\}_{\rho}[x_n]$ has degree $<d$ in $x_n$). We claim that $L_2$ and $L_2^{-1}$ are continuous. Indeed, since $qx_n^d$ and $r$ formal expansion have disjoint support, we have
$$
|L_2(q,r)|_{\rho}=|qx_n^d+r|_{\rho} = |q|_{\rho}\rho^d+|r|_{\rho}\leq C_1(\rho)\max\{|q|_{\rho},|r|_{\rho}\},$$
$$|qx_n^d+r|_{\rho}= |q|_{\rho}\rho^d+|r|_{\rho}\geq \min\{\rho, 1\}\max\{|q|_{\rho},|r|_{\rho}\}
$$
for some positive constant $C_1(\rho)$ depending on $\rho$. 

Now, we claim that $L_3L_2^{-1}$ is a continuous linear map of norm $<1$ for $\rho$ chosen small enough. Indeed, we have that
$L_3(q,r)=q(x_n^d-h)$ and, therefore
$$|L_3(q,r)|_{\rho}\leq |x_n^d-h|_{\rho}|q|_{\rho},$$
hence, for every $a(\x)\in\C\{\x\}_{\rho}$ we have
$$|L_3L_2^{-1}(a(\x))|_{\rho}\leq  \frac{|x_n^d-h|_{\rho}}{\min\{\rho, 1\}} |a(\x)|_{\rho}.$$
But  for $\rho$ small enough, we have $|x_n^d-h|_{\rho}\leq c\rho^2$ since $\ord(x_n^d-h(\x))\geq 2$, for some constant $c>0$. Therefore, the linear map
$$1\!\!1-L_3L_2^{-1}:\C\{\x\}_{\rho}\lgw\C\{\x\}_{\rho}$$
is an isomorphism of Banach algebras. Thus 
$$L_1=(1\!\!1-L_3L_2^{-1})L_2$$
is an isomorphism of Banach algebra if $\rho$ is chosen small enough.
 Therefore, there is a constant $C$, such that for every $a(\x)\in\C\{\x\}_{\rho}$, there is a unique couple $(q,r)\in \C\{\x\}_{\rho}\times(\C\{\x'\}_{\rho})^d$, such that
 $$a(\x)=q(\x)h(\x)+r(\x)
\text{ and }\max\{|q|_{\rho},|r|_{\rho}\}\leq C |a(\x)|_{\rho}.$$
Now, the fact that $a(\x)\in h(\x)\C\lb \x\rb$ is equivalent to $r=0$, and then $a/h=q$, whose norm is bounded by $C|a(\x)|_\rho$.
\end{proof}

\begin{corollary}\label{cor_inequalities}
Let $h(\x)\in\C[\x]$ be a homogeneous polynomial. Then, there is a constant $K>0$ such that, for every homogeneous polynomial $b(\x)\in\C[\x]$:
$$|h|_1|b|_1\leq K|hb|_1.$$
\end{corollary}

\begin{proof}
Let $\rho$ satisfy the previous proposition for $h(\x)$.  Then, with $a(\x)=b(\x)h(\x)$, we get
$$ \left|b(\x)\right|_{\rho}\leq C|h(\x)b(\x)|_{\rho}$$
Since $h$ and $b$ are homogeneous, this gives
$$\rho^{\deg(b)} \left|b(\x)\right|_1\leq C\rho^{\deg(b)+\deg(h)}|h(\x)b(\x)|_1,$$ and 
$$|h(\x)|_1 \left|b(\x)\right|_1\leq C|h(\x)|_1\rho^{\deg(h)}|h(\x)b(\x)|_1.$$
\end{proof}

We are ready to complete the goal of this section:

\begin{proof}[Proof of Proposition \ref{representative}]
First, let us remark that $\sum_{k\in\N}a_k\in\C\{\x\}$, where the $a_k$ are homogeneous polynomials of degree $k$, if and only if $\sum_ka_k\in \C\{\x\}_\rho$, for some $\rho>0$, that is, $\sum_{k\in\N} |a_k|_\rho<\infty$.

Let $h_1$ and $h_2$ be two homogeneous elements such that 
$$A:=\sum_{k\in\N}\frac{a_{1,k}}{h_1^{\a_1k+\b_1}}=\sum_{k\in\N}\frac{a_{2,k}}{h_2^{\a_2 k+\b_2}}$$
and assume that $\sum_ka_{1,k}\in\C\{\x\}$. 
We have 
$$\sum_k \frac{a_{1,k}}{h_1^{\a_1k+\b_1}}=\sum_k\frac{a_{1,k}h_2^{\a_1k+\b_1}}{(h_1h_2)^{\a_1k+\b_1}}.$$
Let $\rho>0$. We have
$$|a_{1,k}h_2^{\a_1k+\b_1}|_\rho\leq |a_{1,k}|_\rho |h_2|^{\a_1k+\b_1}_\rho=|a_{1,k}|_1|h_2|^{\a_1k+\b_1}_1 \rho^{\deg(a_{1,k})+\deg(h_2)(\a_1k+\b_1)}$$
But $\deg(a_{1,k})\leq \a k+\b$ for some $\a$, $\b\in\N$. Thus we have
$$|a_{1,k}h_2^{\a_1k+\b_1}|_\rho\leq |a_{1,k}|_1 (\rho^{\a+\deg(h_2)\a_1}|h_2|_1^{\a_1})^k\rho^\b C$$
for some constant $C>0$. Assume that $\sum_ka_{1,k}\in\C\{\x\}_{\rho'}$ for $\rho'>0$. Then, for $\rho>0$ such that $\rho^{\a+\deg(h_2)\a_1}<|h_2|^{-\a_1}\rho'$, we have
$$
\sum_{k\in\N}|a_{1,k}h_2^{\a_1k+\b_1}|_{\rho}\leq \rho^\beta C\sum_{k\in\N} |a_{1,k}|_1{\rho'}^k=\rho^\beta C\sum_{k}|a_{1,k}|_{\rho'}<\infty.
$$
In the same way, since
$$\frac{a_{1,k}h_2^{\a_1k+\b_1}}{(h_1h_2)^{\a_1k+\b_1}}=\frac{a_{1,k}h_2^{\a_1k+\b_1}(h_1h_2)^{\max\{0,\a_2-\a_1\}k+\max\{0,\b_2-\b_1\}}}{(h_1h_2)^{\max\{\a_1,\a_2\}k+\max\{\b_1,\b_2\}}}$$
we may assume that 
$$A=\sum_{k\in\N}\frac{a_{1,k}}{(h_1h_2)^{\a k+\b}}=\sum_{k\in\N}\frac{a_{2,k}}{h_2^{\a k+\b}}$$
and $\sum_ka_{1,k}\in\C\{\x\}$. Hence for every $k\in\N$, we have
$$a_{1,k}=a_{2,k}h_1^{\a k+\b}.$$
Thus, by Corollary \ref{cor_inequalities}, there is a constant $K>0$ such that
$$\forall k\in\N,\  |a_{2,k}h_1^{\a k+\b-1}|_1\leq \frac{K}{|h_1|_1}|a_{1,k}|_1,$$
and, by induction, 
$$\forall k\in\N,\   |a_{2,k}|_1\leq \frac{K^{\a k+\b}}{|h_1|^{\a k+\b}_1}|a_{1,k}|_1.$$
Therefore, if $\sum_ka_{1,k}\in\C\{\x\}_{\rho'}$ for $\rho'>0$, we have that $\sum_k a_{2,k}\in \C\{\x\}_\rho$ for every $\rho>0$ such that
$\rho<\dfrac{|h_1|_1^\a\rho'}{K^\a}.$
\end{proof}

\begin{remark}
In \cite{Ma}, Mahler proved the following inequality:
$$\forall h, b\in\C[\x], |hb|_1\leq 2^{\deg(h)+\deg(b)}|p|_1|q|_1.$$
Such an inequality has been first proven by Gel'fond for polynomials in one indeterminate \cite{Ge}. We could have used this inequality in order to prove Proposition \ref{representative}.

The inequality of Mahler has the advantage of being effective and uniform in both $h$ and $b$. On the other hand, the inequality given in Corollary \ref{cor_inequalities} shows that the  factor {$K$} can be chosen to be independent of $b$  when $h$ is fixed.
\end{remark}



\subsection{Semi-global formal extension}\label{ssec:SemiGlobalFormalExtention}

The proof of Theorem \ref{thm:SemiGlobalFormal} is done in two steps. First, assuming that $P \in \C\{\x\}[y]$ is a convergent polynomial, we can provide an elementary proof based on analytic continuation and on the classical Abhyankar-Jung Theorem. More precisely:

\begin{proposition}[Semi-global formal extension: the analytic case]\label{prop:SemiGlobalFormalAn}
Let $P(\x,y)\in\C \{ \x \} [y]$ be a monic reduced polynomial, and let $P= \prod_{i=1}^s Q_i$ be the factorization of $P$  given in Corollary \ref{cr:NewtonPuiseux}. Suppose that $n=2$ and let $\s: (N,F) \lgw (\C^2,0)$ be a sequence of point blowing ups such that, at every point $\pb \in F_r^{(1)}$, the pulled-back discriminant $\s_{\pb}^*(\Delta_P) $ is analytically monomial. Then, for every point $\pb  \in F_{r}^{(1)}$, each polynomial $Q_i$ extends formally at $\pb$. Furthermore, the extension is compatible with the factorization of $P$, that is $\prod_{i=1}^r \s_{\pb}^*(Q_i) = \s_{\pb}^*(P)$.
\end{proposition}
\begin{proof}
Given a point $\pb \in F^{(1)}_r$ which is not in the strict transform of $(h=0)$, nor the intersection of exceptional divisors, then each $\s_{\pb}^*(Q_i)$ is a monic formal factors of the convergent monic polynomial $\s_{\pb}^*(P)$. We conclude therefore that each $Q_i$ is convergent. In particular, note that $Q_i \in \PP_h\{\x\}[y]$ by Lemma \ref{max}. We now study the points $\pb\in F^{(1)}_r$ which are either in the strict transform of $(h=0)$, or in an intersection of exceptional divisors.

By combining the normal forms given in Remark \ref{rk:GeometricalFormalProjectiveRing}(2), up to a linear change of coordinates, there exists a coordinate system $(v,w)$ centered at $\pb$ and $c \in \mathbb{Z}_{\geq 0}$ such that:
\[
x_1 = v \cdot w^c, \quad x_2 = v\cdot w^{c+1}
\]

Now, by hypothesis, we know that $\Delta_{P_{\pb}}:=\s_{\pb}^*(\Delta_{P})$ is a analytically monomial (adapted to the exceptional divisor). More precisely, up to an analytic change of coordinates adapted to the exceptional divisor, that is:
\[
v= \widetilde{v}, \quad w = \varphi(\widetilde{v},\widetilde{w})
\] 
we can assume that $\Delta_{P_{\pb}} = \widetilde{v}^k\widetilde{w}^l U(v,w)$, where $U(v,w)$ is a unit (note that this is always the case if $\pb$ is in the intersection of two exceptional divisors, when we necessarily have $w = \widetilde{w}$). Then, by Abhyankar-Jung Theorem, we may write
$$
P_{\pb}=\prod_{i=1}^d(y-\xi_i(\widetilde{v}^{\frac{1}{d!}},\widetilde{w}^{\frac{1}{d!}}))
$$
where the $\xi_i(\widetilde{v}^{\frac{1}{d!}},\widetilde{w}^{\frac{1}{d!}}) \in \C\{ \widetilde{v}^{\frac{1}{d!}},\widetilde{w}^{\frac{1}{d!}}\}$. Now, there exists a connected punctured disc $D \subset F_r^{(1)}$ centred at $\pb$ (in particular, $\pb \notin D$) such that every point $\pc \in D$ does not belong to the strict tranform of $(h = 0)$, nor to the intersection of exceptional divisors. It follows that $\s_{\pc}^*(Q_i)$ is a well-defined analytic function. Furthermore, up to shrinking $D$, for every $i = 1,   \ldots, r$, there exists a subset $J_i(\pc) \subset \{1,   \ldots , d\}$ such that:
\[
\s_{\pc}^*(Q_i)=\prod_{j\in J_i(\pc)}(y-\xi_j(\widetilde{v}^{\frac{1}{d!}},(\widetilde{w}-w_0)^{\frac{1}{d!}})),
\]
where $\pb =(0,w_0)$. Since $D$ is a connected set, we conclude that $J_i(\pc)$ is independent of the point $\pc \in D$, and we denote it by $J_i$. There exists, therefore, an element $q\in\C\{ \widetilde{v}^{1/d!}, \widetilde{w}^{1/d!}\}[y]$ which is equal to $\s_{\pc}^*(Q_i)$ at every point $\pc \in D$. On the one hand, it follows, from analytic continuation, that $\s_{\pb}^*(Q_i)$ is equal to $q$ as a power series in $\C\{ \widetilde{v}^{1/d!}, \widetilde{w}^{1/d!}\}[y]$. Now, on the other hand, if $A \in \PP_h \lb \x \rb$ is a coefficient of $Q_i$ then:
\[
A_{\pb}:=\s_{\pb}^*(A) = \sum_{k\geq 0}\frac{a_k(1,\varphi(\widetilde{v},\widetilde{w}))}{h(1,\varphi(\widetilde{v},\widetilde{w}))^{\a k+\b}}\widetilde{v}^{k}\varphi(\widetilde{v},\widetilde{w})^{c k},
\]
and $A_{\pb}$ is invariant by the action of the $(d!)$-th roots of unity $(\zeta_1,\zeta_2)\lgm (\zeta_1 \widetilde{v} ,\zeta_2 \widetilde{w})$, which implies that $q$ is also invariant by these actions. We conclude that $q =: \s_{\pb}^*(Q_i)$ is a power series in $\C\{\widetilde{v},\widetilde{w}\}[y]$, finishing the proof.
\end{proof}

Unfortunately, there is no notion of analytic continuation in the case of formal power series, so the proof does not extend in a trivial way to the formal case. Following Gabrielov's original idea \cite{Ga2}, we address this extra technical issue via the Artin Approximation Theorem:

\begin{theorem}[Artin Approximation Theorem \cite{Artin}]
Let $G(\x,\y)$ be a nonzero function in $\C\{ \x \} [\y]$, where $\x = (x_1,  \ldots,x_n)$ and $\y = (y_1,  \ldots,y_m)$, and consider the equation $G(\x,\y)=0$. Suppose that there exists formal power series $\hat{\y}(\x) =(\hat{y}_1(\x),   \ldots,\hat{y}_m(\x)) \in \left(\C \lb \x \rb\right)^m$ such that:
\[
\widehat{G}(\x,\hat{\y}(\x)) \equiv 0
\]
Then there exist convergent power series $
\y^{(\iota)}(\x) = (y_1^{(\iota)}(\x),   \ldots,y_m^{(\iota)}(\x)) \in \left(\C \{ \x \}\right)^m
$ for every $\iota \in \mathbb{N}$, such that:
\[
G(\x,\y^{(\iota)}(\x)) \equiv 0 \text{ and for every }k,~ \hat{y}_k(\x) - y_k^{(\iota)}(\x)  \in (\x)^{\iota}.
\]
\end{theorem}

Before turning to the proof of Theorem \ref{thm:SemiGlobalFormal}, we need to study conditions under which a well-chosen approximation of a polynomial yields an approximation of their roots and factors. We start by setting, once and for all, the extension of the valuation $\nu$ which we are interested in working with:

\begin{definition}
Let  $V$ be a  ring equipped with a valuation $\nu: V\setminus\{0\}\lgw \R$.
Let $P(y)\in V[y]$, $P(y)=\sum_{k=0}^d a_iy^i$, $a_i\in V$ for every $i$. We define
$$\nu(P(y))=\min_{{a_j\neq 0}}\{ \nu(a_0),  \ldots,\nu(a_d)\}.$$
\end{definition}

We now turn to the proof of two preliminary results:

\begin{lemma}[Root-approximation]\label{key_app}
Let $\K$ be an algebraically closed field equipped with a valuation $\nu:\K\setminus\{0\}\lgw \R$. We denote by $V$ its valuation ring.
Let $P(y)\in V[y]$ be a reduced monic polynomial of degree $d$ in $y$, and $\xi_1$,   \ldots, $\xi_d$ be its roots in $\K$. Let $Q(y)\in V[y]$ be a monic polynomial such that
$$
2\nu(P-Q)> d\cdot \nu(\Delta_P).
$$
Then, for every $i=1,  \ldots,d$, $Q(y)$ has a unique root $\xi'_i$ in $\K$ such that 
\[
\nu(\xi_i-\xi_i')\geq \frac{\nu(P-Q)}{d} .
\]
\end{lemma}

Note that the unicity implies that, in these conditions, $Q$ is necessarily reduced.

\begin{proof}
We start by remarking that, if the degree $\deg_y(Q)$ is different from $d$, then $\nu(P-Q)=0$ since these are monic polynomials, and the hypothesis of the Lemma are not satisfied. We therefore have that $\deg_y(Q)=d$, and we denote by $\xi'_1$,   \ldots, $\xi'_d$ the roots of $Q(y)$ counted with multiplicity (hence we may have $\xi'_i=\xi'_j$ for some $i\neq j$). Since $P$ and $Q$ are monic polynomials, we have $\nu(\xi_i)$, $\nu(\xi'_j)\geq 0$ for all $i$ and $j$. We write
$$P(y)=y^d+a_1y^{d-1}+\cdots+a_d,\ Q(y)=y^d+b_1y^{d-1}+\cdots+b_d.$$
Let us fix $i\in\{1,  \ldots, d\}$. We have
$$
\prod_{j=1}^d(\xi_i-\xi'_j)=Q(\xi_i)=Q(\xi_i)-P(\xi_i)=\sum_{k=1}^d(b_k-a_k)\xi_i^{d-k},
$$
and, since $\nu(\xi_i) \geq 0$, we get:
\[
\nu\left(\prod_{j=1}^d(\xi_i-\xi'_j) \right) \geq \min_{k=1}^d\{\nu(a_k-b_k)\} =\nu(P-Q).
\]
Therefore, there is (at least) one integer $j(i)$ such that
\begin{equation}\label{ine_val}
\nu(\xi_i-\xi'_{j(i)})\geq \frac{1}{d}\nu(P(y)-Q(y))>\frac{1}{2}\nu(\Delta_P)\geq \max_{l\leq k}\{\nu(\xi_k-\xi_l)\}.
\end{equation}
Now let $m\in\{1,  \ldots, d\}$, $m\neq i$. Then, by \eqref{ine_val}, we get
$$
\nu(\xi_m-\xi'_{j(i)})=\nu(\xi_m-\xi_i+\xi_i-\xi'_{j(i)})=\nu(\xi_m-\xi_i)\leq \max_{l\leq k}\{\nu(\xi_k-\xi_l)\}.
$$
Therefore $j(m)\neq j(i)$ whenever $i\neq m$. This gives the unicity of $\xi_i'$ and concludes the proof.
\end{proof}


\begin{proposition}[Factor-approximation]\label{perturbation}
Let $P(y)\in\wdh V_\nu[y]$ be a reduced monic polynomial of degree $d$ in $y$. We write
$$
P(y)=P_1(y)\cdots P_s(y)
$$
where the $P_i(y)$ are irreducible monic polynomials of $\wdh V_\nu[y]$. Let $Q(y)\in\wdh V_\nu[y]$ be a monic polynomial such that
$$
2\nu(P(y)-Q(y))> d\nu(\Delta_P).
$$
Then $Q(y)$ can be written as a product of irreducible monic polynomials $Q_i[y]\in\wdh V_\nu[y]$:
$$Q(y)=Q_1(y)\cdots Q_s(y),$$
such that
$$
\nu(P_i-Q_i)\geq \frac{\nu(P-Q)}{d} \quad  i=1,  \ldots, s.
$$
\end{proposition}
\begin{proof}
Again, if the degree $\deg_y(Q)$ is different from $d$, then $\nu(P-Q)=0$ since these are monic polynomials, and the hypothesis of the Proposition are not satisfied. We therefore have $\deg_y(Q)=d$. We denote by $\xi_1$,   \ldots, $\xi_d$ the distinct roots of $P(y)$, belonging to $\wdh V_\nu[\g]$ for some homogeneous element $\g$, and by $\xi'_1$,   \ldots, $\xi'_d$ the roots of $Q(y)$ counted with multiplicity (hence we may have $\xi'_i=\xi'_j$ for some $i\neq j$).

By Lemma \ref{key_app}, for every $i\in\{1,  \ldots, d\}$, there is a unique $j(i)$ such that
$$
\nu(\xi_i-\xi'_{j(i)})\geq \frac{\nu(P-Q)}{d},
$$
and, up to renumbering, we can suppose that $j(i)=i$. Now, fix $i\in \{1,  \ldots, d\}$. There exists $p \in \mathbb{N}$ (where $p$ depends on $i$) and $\s_1=\Id$, $\s_2$,   \ldots, $\s_p$ $\Frac(\wdh V_\nu)$-automorphisms of $\ovl \K_\nu$, such that
$\xi_i=\s_1(\xi_i)$,   \ldots, $\s_p(\xi_i)$ are the distinct conjugates of $\xi_i$ over $\wdh V_\nu$.  Then
$\prod_{k=1}^p(y-\s_k(\xi_i))$ is an irreducible factor of $P(y)$ in $\wdh V_\nu[y]$.
 
Now let $\s$ be any $\Frac(\wdh V_\nu)$-automorphism of $\ovl \K_\nu$. Because $\s(\xi_i)$ is a conjugate of $\xi_i$, there is an integer $l$ such that $\s(\xi_i)=\s_l(\xi_i)$. Moreover we have
$$
\nu(\s(\xi'_{i})-\s(\xi_i))=\nu(\xi'_{i}-\xi_i)\geq \frac{\nu(P-Q)}{d}.
$$
Therefore
\[
\begin{aligned}
\nu(\s(\xi'_{i})-\s_l(\xi_i))=\nu(\s(\xi'_{i})-\s(\xi_i))=&\nu(\xi'_{i}-\xi_i)=\\
= \nu(\s_l(\xi'_{i})-\s_l(\xi_i))&\geq  \frac{\nu(P-Q)}{d},
\end{aligned}
\]
and we conclude that $\s(\xi'_i) = \s_l(\xi_i')$.
This proves that $\s_1(\xi'_{i})$,   \ldots, $\s_p(\xi'_{i})$ are exactly the (distinct) conjugates of $\xi_{i}'$ over $\Frac(\wdh V_\nu)$. In particular the polynomial
$\prod_{k=1}^p(y-\s_k(\xi'_{i}))$
is an irreducible monic factor of $Q(y)$ in $\wdh V_\nu[y]$ such that
$$
\nu\left(\prod_{k=1}^p(y-\s_k(\xi_i))-\prod_{k=1}^p(y-\s_k(\xi'_{i}))\right)\geq  \frac{\nu(P-Q)}{d},
$$
as we wanted to prove.
\end{proof}

We are now ready to prove the main result of this section:

\begin{proof}[Proof of Theorem \ref{thm:SemiGlobalFormal}]
The proof is done by approximating the formal polynomial $P(\x,y)$ by a suitable sequence of analytic polynomials $(P^{(\iota)})_{\iota \in \mathbb{N}}$ satisfying all hypothesis of Proposition \ref{prop:SemiGlobalFormalAn}, and arguing via Proposition \ref{perturbation}. Indeed, we start writing:
\[
P = P_0(\x) +   \cdots + P_{d-1}(\x)y^{d-1} + y^d, \quad \Delta_{P} =\D_1^{k_1}\cdots\D_e^{k_e},
\]
where the $\D_i$ are distinct irreducible formal polynomials and $k_i>0$ for all $i=1  \ldots,e$. We consider the universal discriminant polynomial $\Delta_d$ of degree $d:= \mbox{deg}_y(P)$, that is, $\D_d(P_0,  \ldots,P_{d-1})$ is the discriminant of the polynomial $P_0 +   \cdots + P_{d-1}y^{d-1} + y^d$. We apply
the Artin Approximation Theorem to the equation:
\[
\D_d(y_0,  \ldots,y_{d-1}) = z_1^{k_1} \cdots z_e^{k_e}
\]
with respect to the formal solution $y_i = P_i(\x)$ and $z_i = \D_i(\x)$. Therefore, for every $\iota \in \mathbb{N}$, there exist $P^{(\iota)}_0(\x),  \ldots,P^{(\iota)}_{d-1}(\x) \in \C\{\x\}$ such that
\[
P^{(\iota)}_j(\x) - P_j(\x) \in (\x)^{\iota},\quad j=0,  \ldots,d-1
\]
and we consider the polynomial with analytic coefficients
\[
P^{(\iota)}(\x,y):= P^{(\iota)}_0(\x)+    \cdots +P^{(\iota)}_{d-1}(\x) y^{d-1} + y^d
\]
which has a discriminant, by construction, of the form
\[
\Delta_{P^{(\iota)}} =(\D_1^{(\iota)})^{k_1}\cdots(\D_e^{(\iota)})^{k_e}, \quad \text{ where } \quad \D_j^{(\iota)} - \D_j \in (\x)^{\iota}, j=1,  \ldots,e.
\]
Note that, by Theorem \ref{thm:NewtonPuiseux}, these polynomials admit a factorization $P^{(\iota)}= \prod_{i=1}^{s^{(\iota)}} Q_i^{(\iota)}$ in $\PP_{h^{(\iota)}} \lb \x\rb[y] \subset \widehat{V}_{\nu}[y]$. By Proposition \ref{perturbation}, for $\iota$ big enough, the number of factors $s^{(\iota)}$ is constant equal to $s$ and $P^{(\iota)}$ is a reduced in $\widehat{V}_{\nu}[y]$.

In order to apply Proposition \ref{prop:SemiGlobalFormalAn}, it remains to verify that there exists $\iota_0 \in \mathbb{N}$ such that, for every $\iota>\iota_0$ and for every $\pb \in F^{(1)}_r$, the discriminant $\s_{\pb}^*(\D_{P^{(\iota)}})$ is analytically monomial. Indeed, fix a point $\pb \in F_r^{(1)}$. By Remark \ref{rk:GeometricalFormalProjectiveRing}(2), there exists a coordinate system
$(v,w)$ centered at $\pb$ and adapted to $F$ given by:
\[
x_1 = v w^c, \quad x_2 = vw^{c+1}
\]
with $c\in \mathbb{Z}_{\geq 0}$. In particular, we note that:
\begin{equation}\label{eq:discriminantArtin}
\s_{\pb}^*(\D_j^{(\iota)}) - \widehat{\s}_{\pb}^*(\D_j) \in (vw^c)^{\iota}, \quad  j=1,  \ldots,e.
\end{equation}
We now divide in three cases in order to prove that $\s_{\pb}^*(\D_j^{(\iota)})$ is analytically monomial (for $\iota$ big enough) depending on the nature of $\pb$:

\medskip
\noindent
\emph{Case I:} Suppose that $\pb$ is in the intersection of two exceptional divisors. In this case, $c>0$ and $\widehat{\s}_{\pb}^*(\D_j) = v^{\alpha_j} w^{\b_j} \widehat{u}_j$, where $\alpha_j, \beta_j >0$ and $\widehat{u}_j(0)\neq 0$, for every $j=1,  \ldots,e$. It follows from \eqref{eq:discriminantArtin} that if $\iota> \max\{\a_j,\b_j\}$, then $\s_{\pb}^*(\D_j^{(\iota)})$ is analytically monomial. We easily conclude that $\s_{\pb}(\Delta_{P^{(\iota)}})$ is analytically monomial for $\iota > \max_{j=1}^e\{\a_j,\b_j\}$. 

\medskip
\noindent
\emph{Case II:} Suppose that $\pb$ is not in the intersection of two exceptional divisors, nor in the strict transform of $\D_P$. In this case, $c=0$ and $\D_j = v^{\a_j}\widehat{u}_j$, where $\a_j>0$ and $\widehat{u}_j(0)\neq 0$, for every $j=1,  \ldots,e$. It follows from \eqref{eq:discriminantArtin} that if $\iota> \a_j$, then $\s_{\pb}^*(\D_j^{(\iota)})$ is analytically monomial. We easily conclude that $\s_{\pb}^*(\Delta_{P^{(\iota)}})$ is analytically monomial for $\iota > \max_{j=1}^e\{\a_j\}$.

\medskip
\noindent
\emph{Case III:} Suppose that $\pb$ is in the strict transform of $\D_P$. In this case, $c=0$. We note that, since $\widehat{\s}_{\pb}^*(\Delta_P)$ is formally monomial (adapted to the exceptional divisor), then $\widehat{\s}_{\pb}^*(\Delta_P) = v^{\a} \cdot g^{\b} \widehat{u}$, where $\a>0$, $\widehat{u}(0)\neq 0$ and $g$ is a formal series with initial term $a v + bw$, with $b\neq 0$. Since each term $\D_j$ is irreducible and distinct, we conclude that $g$ can only divide one term $\D_j$, say $\D_1$. Therefore, we can write $\D_1 = v^{\a_1} g \cdot \widehat{u}_1$ and $\D_j = v^{\a_j}\widehat{u}_j$, where $\a_j >0$ and $\widehat{u}_j(0)\neq 0$, for every $j=1,  \ldots,e$, and $\widehat{u}_1(0) =1$. Now, it follows from \eqref{eq:discriminantArtin} that if $\iota> \a_j$, then $\s^*_{\pb}(\D_j^{(\iota)})$ is analytically  monomial for $j=2,  \ldots,e$. Next, it follows from \eqref{eq:discriminantArtin} that if $\iota > \a_1$, then $\s^*_{\pb}(\D_j^{(\iota)}) = v^{\a_1} g^{(\iota)}$, where $\partial_{w}g^{(\iota)}(0) = b \neq 0$. It follows from the implicit function Theorem, therefore, that $\s^*_{\pb}(\D_j^{(\iota)})$ is analytically  monomial and adapted to the exceptional divisor. We easily conclude that $\s_{\pb}^*(\Delta_{P^{(\iota)}})$ is analytically  monomial for $\iota > \a$.

\medskip
Therefore, since $F_r^{(1)}$ is compact and $\Delta_j^{(\iota)}$ are convergent power series for every $\iota$ and $j=1,  \ldots,e$, we conclude that there exists $\iota_0 \in \mathbb{N}$ such that, for every $\iota>\iota_0$ and for every $\pb \in F^{(1)}_r$, the discriminant $\s_{\pb}^*(\D_{P^{(\iota)}})$ is analytically  monomial. By Proposition \ref{prop:SemiGlobalFormalAn}, the factors $Q_i^{(\iota)}$ admits a formal extension to every point $\pb \in F_r^{(1)}$ (which are compatible with the factorization of $P^{(\iota)}_{\pb}$). Next, by Proposition \ref{perturbation}, the sequence $(Q_j^{(\iota)})_\iota$ converges to $Q_j$ with respect to the usual $\nu$-adic topology in $\widehat{V}_{\nu}$, for every $j=1,  \ldots,s$. 

Now, let $j \in \{1,  \ldots,s\}$ be fixed, $A$ be a coefficient of $Q_j$ and $A^{(\iota)}$ be the corresponding coefficient of $Q^{(\iota)}_j$. Then we have
$$
\s^*_{\pb}(A^{(\iota)})=\sum_{k\geq 0}\frac{a_{k,\iota}(1,w)}{h_{\iota}^{\a_{\iota} k+\b_{\iota}}(1,w)} (w^{c}v)^k
$$
and, since $A^{(\iota)}$ extends analytically at $\pb$, there exist polynomials $b_{k,\iota}(w)$ such that $a_{k,\iota}(1,w)=b_{k,\iota}(w)h_{\iota}(1,w)^{\a_{\iota} k+\b_{\iota}}$ for every $k\geq 0$. Since $A-A^{(\iota)}\in (\x)^{\iota}$, we have $\s_{\pb}(A)-\s_{\pb}(A^{(\iota)})\in (vw^c)^{\iota}$. Therefore, if we write 
\[
A_{\pb}=\sum_{k\geq 0}\frac{a_k(1,w)}{h(1,w)^{\a k + \b}}(w^cv)^k
\]
we have $a_k(1,w)=b_{k,\iota}(w)h(1,w)^{\a k+\b}$ for $k\leq  \iota$. Since this is true for every $\iota$, we have that $A_{\pb}$ extends formally to $\pb$, and so does $Q_i$. Finally, this extension is compatible with the factorization of $P$ because, as it is pointed out in Remark \ref{rk:GeometricalFormalProjectiveRing}(2), $\wdh\sigma^*_{\pb}: \PP_{h}\lb \x\rb [y] \lgw 
C(w)\lb v\rb$ is a well-defined morphism.
\end{proof}

\subsection{Local-to-Semi-global convergence of factors}\label{ssec:LocalToSemiGlobal}
 The goal of this section is to prove Theorem \ref{thm:IzumiWalsh}. Roughly speaking, one needs  to prove that if, for some root $\xi:=\sum_{i=0}^{d-1} A_i\g^i$ of $P$, where the $A_i$ are in $\PP_h\lb \x\rb$ and $\g$ is a homogeneous element, $\s_\pb^*(\xi)$ is convergent, then $\s_\pb^*\left(\sum_{i=0}^{d-1} A_i {\g'}^i\right)$ is again convergent for every conjugate $\g'$ of $\g$ over $\C(\x)$. The main difficulty here, is that, if $\Gamma(\x,z)$ denotes the minimal polynomial of $\g$ over $\C(\x)$, the transform $\s_\pb^*(\Gamma(\x,z))$ may not be locally irreducible anymore. For instance let $\Gamma(\x,z):=z^4-(x_1^2+x_2^2)$ and $\s$ is given by $\s_0^*(x_1)=u$ and $\s_0^*(x_2)=uv$. Then in this case we have locally at the origin:
  $$\s_0^*(\Gamma(\x,z))=(z^2-u\sqrt{1+v^2})(z^2+u\sqrt{1+v^2}).$$
Therefore one needs to argue globally along the exceptional divisor, since the transform of $\Gamma(\x,z)$ remains globally irreducible.
 
Our proof follows the strategy inspired from Tougeron's \cite[$\S$3.1]{To2}, which relies on the following Lemma (stated in \cite[Lemme 3.3]{To2}):


\begin{lemma}\label{lemma:IzumiMajorationPol}
Let $\mathcal{C}\subset \C^n$ be an irreducible algebraic curve, and $D_1$, $D_2$ be two compact {subsets} of $\mathcal{C}${, such that the interior of $D_1$ is nonempty}. Then $$\exists M >0, \forall P\in \C[z_1,\ldots,z_n], \ \ \|P\|_{D_2}\leqslant M^{deg(P)} \|P\|_{D_1},$$ where $\|P\|_D$ denotes $\max\limits_{z\in D} |P(z)|$.
\end{lemma}

The proof given by Tougeron in \cite[Lemme 3.3]{To2} is only valid if the compactification $\overline{\mathcal{C}}\subset \C\mathbb{P}^n$ of the curve $\mathcal{C}$ is isomorphic to $\mathbb{P}^1$. In order to treat the general case, we incorporate the idea of the proof of \cite[Theorem 3.1]{IzDuke}, involving Green's functions.

\begin{proof}
Let $\overline{\mathcal{C}}$ be the smooth compact curve obtained as the normalization of the closure of $\mathcal{C}$ in $\C\mathbb{P}^n$. Note that the irreducibility of $\mathcal{C}$ is equivalent to that of $\overline{\mathcal{C}}$. Then the coordinate functions $z_i$ and $P$ lift canonically to meromorphic functions $\overline z_i, \overline{P}$ on $\overline{\mathcal{C}}$. Furthermore, each of these functions has its poles outside of {$K_1\cup K_2$}, where {$K_i$ denotes the preimage of $D_i$ by the normalization}.

Denote by $\pq_1,  \ldots,\pq_k$ the distinct poles of the functions $\overline z_i$, let $\pq_0$ be a point in the interior of {$K_1$}, and  for $i=1,  \ldots, n$, let 
\[
m_i:=\max\limits_{j =1,  \ldots,k} \operatorname{mult}_{\pq_i}\overline z_j,\] where $\operatorname{mult}_{\pq}\overline{z}=m$
means that $\pq$ is a pole of multiplicity $m$ of $\overline{z}$. With these notations, if $\pq_i$ is a pole of $\overline{P}$, it is of multiplicity at most $m_i\cdot \deg(P)$. Furthermore, the only possible poles of $\overline{P}$ are $\{\pq_1,\cdots,\pq_k\}$.

Now, set $U:=\overline{\mathcal C}\setminus \{\pq_0,\pq_1,\ldots,\pq_k\}$. Then, following \cite[Chapter II, $\S$1]{LA}, for every $i\in \{0,\ldots,k\}$, there exists a Green's function for $\pq_i$, that is, a smooth function $G_i\colon\overline{\mathcal C}\setminus \pq_i\lgw \R$ such that, for every $i,j$, $\Delta G_i = \Delta G_j$  on $U$ and, on a neighbourhood of $\pq_i$, the function $G_i+\log |z-\pq_i|$ can be extended to $\pq_i$ as a smooth function. Note that, since $\lim\limits_{z\to \pq_i} \log |z-\pq_i|=-\infty$, this last condition implies that $G_i$ is positive on a neigbourhood of $\pq_i$. Finally, such functions are uniquely determined up to additive constants. We can therefore pick the functions $G_i$ so that, for every $i\in \{1,\ldots,k\}$, $G_i-G_0\geqslant 0$ on $\overline{U\setminus {K_1}}$, by compacity of $\overline{\mathcal C}$. Set 
\[
G:=\sum\limits_{i=1}^k m_i(G_i-G_0),
\]
which is harmonic on $U$, because every function $G_i-G_0$ is harmonic on $U$, and set $h=\exp(G)$. The function $|\overline{P}/h^{\deg(P)}|^2$ is subharmonic on $U$ because at every point of $U$, it can locally be seen as the square of the module of a holomorphic function. Indeed, since $G$ is harmonic on $U$, for every point $\pp$ of $U$, there is a neighbourhood $V$ of $\pp$ and a harmonic function $H$ on $V$ which is conjugate to $G$, that is, $H$ is such that $G+iH$ is holomorphic on $V$, and we have 
\[
\left|
\frac{\overline{P}}{h^{\deg(P)}} \right|=\left|\frac{\overline{P}}{\exp(G+iH)^{\deg(P)}}\right|,
\] 
where $\exp(G+iH)^{\deg(P)}$ is holomorphic. For $1\leqslant i \leqslant k$, the function $G_i$ can be written near the point $\pq_i$ as $G_i=\alpha_i-\log|z-\pq_i|$ for some smooth function $\alpha_i$. Furthemore, because of the bound on its multiplicity, $\overline{P}$ can be written near $\pq_i$ as $\beta_i/(z-\pq_i)^{m_i}$, for some analytic function $\beta_i$.Therefore, the function $|\overline{P}/h^{\deg(P)}|^2$ can be extended to a smooth function on $\overline{\overline {\mathcal C} \setminus {K_1}}$. Moreover, by continuity of the partial derivatives, the non negativity of the Laplacian of $|\overline{P}/h^{\deg(P)}|^2$ extends to every point $\pq_i$, for $1\leqslant i \leqslant k$, therefore $|\overline{P}/h^{\deg(P)}|^2$ extends to a subharmonic function on $\overline{\overline{{\mathcal C}}\setminus  {K_1}}$. Then, the maximum principle applied to this function yields 
\[
\|\overline{P}/h^{\deg(P)}\|_{{K_2}}\leqslant \|\overline{P}/h^{\deg(P)}\|_{\overline{\overline{{\mathcal C}}\setminus {K_1}}}\leqslant \|\overline{P}/h^{\deg(P)}\|_{\partial {K_1}}\leqslant \|\overline{P}\|_{\partial {K_1}} \leqslant \|\overline{P}\|_{{K_1}}, 
\]
where the last {but one} inequality comes from the fact that $h\geqslant 1$ on $\partial {K_1}$. Finally, denoting $M=\max\limits_{z\in {K_2}} h(z)$, we get $\|\overline{P}\|_{{K_2}}\leqslant M^{\deg(P)} \|\overline{P}\|_{{K_1}}$.
\end{proof}

We are now ready to turn to the proof of Theorem \ref{thm:IzumiWalsh}:

\begin{proof}[Proof of Theorem \ref{thm:IzumiWalsh}]
Let $P\in \C\lb \x\rb [y]$ be a monic, reduced polynomial, and let $\tg \colon(N,F)\lgw (\C^2,0)$ be a sequence of point blowing ups, such that the pulled-back discriminant $\widehat{\sigma}_\pb^*(\Delta_P)$ is formally monomial at every point $\pb$ of $F_r^{(1)}$. We consider the factorization of $P$ into irreducible polynomials of $\PP_h\lb \x\rb [y]$ given by Corollary \ref{cr:NewtonPuiseux}:
\[
P(\x,y)=\prod\limits_{i=1}^s Q_i,
\]
By hypothesis, there exists $\pb_0\in F_r^{(1)}$ such that $\widehat{\sigma}_{\pb_{{0}}}^{\ast}(P)=:P_{\pb_0}$ has a convergent factor. By Proposition \ref{prop:FormalContinuationConvergentFactor}, we can assume that $\pb_0$ is in no other component of the exceptional divisor nor on the strict transform of $\{h=0\}$. Therefore, by Theorem \ref{thm:SemiGlobalFormal}, we can write 
\[
P_{\pb_0}=\prod\limits_{i=1}^s \widehat{\sigma}^*_{\pb_0}(Q_i), \quad \text{ where } \quad \widehat{\sigma}^*_{\pb_0}(Q_i) \in \wdh{\mathcal{O}}_{\pb_0}[y].
\] 
Now, by hypothesis, there is an index $i$ such that $\widehat{\sigma}_{\pb_0}^*(Q_i)$ admits a convergent factor. The problem is now to prove that the polynomial $Q_{\pb_0}:=\widehat{\sigma}_{\pb_0}^*(Q_i)$ itself is convergent. For simplicity, in the sequel we denote $Q_i$ by $Q$, and {$\gamma_i$ by $\gamma$} (with the notation of Corollary \ref{cr:NewtonPuiseux}). We now use the second equality given by Corollary \ref{cr:NewtonPuiseux}. Let $\g$ be a homogeneous integral element of degree $\omega = p/e$ such that
\begin{equation}\label{eq:IzumiQbefore}
Q=\prod\limits_{i=1}^s\left(y-\xi(\x,\g_i)\right), \quad \text{ where } \quad \xi\in \PP_h\lb \x, \g\rb,
\end{equation}
where $\g_i$ are distinct roots of the minimal polynomial of $\g$:
\begin{equation}\label{eq:IzumiFbefore}
 \Gamma(\x,z)=z^d+\sum\limits_{i=1}^d f_{i}(\x)z^{d-i}.
\end{equation} 
which is a $\omega$-weighted irreducible homogeneous polynomial. In particular, $f_{i}(\x)$ is homogeneous of degree $\omega \cdot i$, $f_i=0$ if $e$ does not divide $i$, and $e$ divides $d$ (because $f_d\neq 0$; otherwise $z$ would divide $\Gamma$ which is irreducible). We conclude that $\Gamma \in\C[\x,z^e]$. Now, note that by the definition of $\PP_h\lb \x, \g\rb$, we can write:
\begin{equation}\label{eq:IzumiXibefore}
\xi(\x,\gamma_i)=\sum\limits_{j=1}^d A_j(\x)\cdot \g^j_i, \quad \text{ where } \quad A_j=\sum\limits_{k\geqslant k_j} \frac{a_{k,j}(\x)}{h^{\alpha_jk+\beta_j}(\x)},
\end{equation}
with $k_j+\omega j\geqslant 0$.

Denote by $\overline{\Gamma}$ the strict transform of $\Gamma$ by the weighted blowing up given by $x_1=v^e,x_2=v^ew,z=v^p z$ in the chart $v\neq 0$:
 $$
 \overline{\Gamma}(w,z)=z^d+\sum\limits_{i=1}^d f_{i}(1,w)z^{d-i}.
 $$
Since $\Gamma$ is irreducible in $\C[\x,z]$, the compact algebraic curve $\ovl{\mathcal{C}}$ corresponding to $\overline{\Gamma}$ in the weighted projective space $\C \mathbb{P}^2_{(e,e,p)}$ is irreducible. {In a more algebraic language, this means that one has the following:}\\
\textbf{Claim:} {The polynomial} $\overline \Gamma$ is irreducible as an element of $\C[w,z^e]$.

\begin{proof}[Proof of the claim]Indeed, assume that $\ovl \Gamma=GH$ where $G$, $H$ are non trivial polynomials of $\C[\x,z^e]$, monic in $z$. Let us write
$$
G=\sum_{i=1}^{d_1}g_i(w)z^{d_1-i}+z^{d_1}
$$
where $g_i(w)=0$ if $e$ does not divide $i$, and $e$ divides $d_1$. Then
$$v^{pd_1}G=\sum_{i=1}^{d_1}v^{pi}g_i(w)(v^pz)^{d_1-i}+(v^pz)^{d_1}.$$
Let $i$ be a multiple of $e$. Then a monomial of $v^{pi}g_i(w)$ has the form 
$$
v^{pi}w^k=v^{p\frac{i}ee}w^k=v^{\left(p\frac{i}{e}-k\right)e}(v^ew)^k=x_1^{p\frac{i}{e}-k}x_2^k.
$$
Thus $G=G'(\x,z)$ with $G'(\x,z)\in\C(\x)[z^e]$. The same being true for $H$ we have that $\Gamma$ factors as a product of two monic polynomials in $\C(\x)[z^e]$. Since $\Gamma$ is irreducible and $\C[\x]$ is a unique factorization domain, by Gauss's Lemma $\Gamma$ factors as a product of two monic polynomials of $\C[\x,z^e]$, which is not possible by assumption. This shows that $\ovl \Gamma$ is irreducible in $\C[w,z^e]$. 
\end{proof}

Therefore $\overline \Gamma$ may not be irreducible in $\C[w,z]$, but its irreducible factors can be obtained from one another by multiplying the second variable by a $e$-th root of unity. Denote $\mathcal{C}:=\{\ovl \Gamma=0\}\subset \C^2_{w,z}$.

 This implies that:
\[
 \overline{\Gamma}(w,z) = \prod_{i=1}^t \overline{\Gamma}_i(w,z),
\]
where $\mathcal{C}_i:=(\overline{\Gamma}_i(w,z)=0)$ are irreducible curves and each of the irreducible factors $\Gamma_i$ of $ \overline{\Gamma}(w,z)$ can be obtained from one another by multiplying the second variable by a $e$-th root of unity.

Now, let $(v,w)$ be coordinates at ${\pb_0}$ such that $\sigma_{\pb_0}$ is locally given by $(x_1,x_2)=(v,vw)$, as in Remark \ref{rk:GeometricalFormalProjectiveRing}(2a). It follows from equation \eqref{eq:IzumiQbefore} that:
\[
Q_{\pb_0}=\prod\limits_{i=1}^s\left(y-\xi_i(v,vw,\tilde{\gamma})\right)
\]
where, by equation \eqref{eq:IzumiFbefore}, $\tilde{\g}$ is a root of the polynomial:
\[
\widetilde{\Gamma}(v,w,z):=z^d+\sum\limits_{i=1}^d v^{\omega i}f_{i}(1,w)z^{d-i}.
\]
By comparing the expressions of $\widetilde{\Gamma}$ and $\overline{\Gamma}$, we conclude that $\tilde{\g}$ is a root of $\widetilde{\Gamma}(v,w,z)$ if and only if $\tilde{\g}=v^\omega \overline{\g}$, where $\overline{\g}$ is a root of $\overline{\Gamma}(w,z)$. 

We use again Proposition \ref{prop:FormalContinuationConvergentFactor} in order to assume that ${\pb_0}$ is not a zero of the discriminant of $\overline{\Gamma}(w,z)$ with respect to the projection {$\pi$ on} the $w$-axis. With this condition, the implicit function Theorem implies that there is a compact disc $D'\subset \{z=0\}$ centered at the origin of $\C^2_{v,w}$ such that, on each one of the connected components $D'_1,  \ldots,D'_d$ of $\pi^{-1}(D')\subset \mathcal{C}$, $z$ can be written as an analytic function in $w$. From now on, $\overline{\gamma}_{{i}}\in \C\{w\}$ denotes the solution of $\overline{\Gamma}(w,z)=0$ on $D'_{{i}}$.

{Note that for every $j$, $v^{\omega j}A_j(v,vw)\in \C\lb v^{1/e},w\rb$, therefore each $\widehat{\sigma}_{\pb_0}^*(\xi_i)$ is in $\C\lb v^{1/e},w\rb$.} Up to renumbering, since $Q_{\pb_0}$ has a convergent factor, we can assume that $\widehat{\sigma}^*_{\pb_0}(\xi_1)\in \C\{v^{1/e},w\}$. By equation \eqref{eq:IzumiXibefore}, we can write (c.f. the normal form given in Remark \ref{rk:GeometricalFormalProjectiveRing}(2a)):
$$
\widehat{\sigma}_{\pb_0}^*(\xi_i) =\sum\limits_{j=1}^d \widehat{\sigma}_{\pb_0}^*(A_j) \cdot v^{\omega j}\overline{\g}^j_i
=\sum\limits_{j=1}^d v^{\omega j}\overline{\g}^j_i\sum\limits_{k\geqslant k_j} v^k\frac{a_{k,j}(1,w)}{h(1,w)^{\alpha_jk+\beta_j}}.
$$
The coefficient of $v^{\ell/e}$, for $\ell\in\N$, in the previous sum is
$$
\sum_{j=1}^d\ovl \g^j\frac{a_{\ell/e-j\omega,j}(1,w)}{h(1,w)^{\a_j(\ell/e-j\omega)+\b_j}}.
$$
(where $a_{\ell/e-j\omega,j}(1,w)=0$ if $\ell/e-j\omega \notin \mathbb{N}$). Therefore we have 
$$
\widehat{\sigma}^*_{\pb_0}(\xi_i)=\sum\limits_{\ell\in \N} \frac{v^{\ell/e}}{h(1,w)^{\delta_\ell}} P_\ell(w,\ovl \gamma_i),
$$
where $\delta_\ell\in \N$ is bounded by an affine function in $\ell$, and the degree of $P_\ell\in \C[w,z]$ is bounded by an affine function in $\ell$. 

Now, note that if $\eta$ is a primitive $e$-th root of unity, then the set $\{\eta^l \ovl\g, 0\leqslant l \leqslant e-1\}$ contains a root of each irreducible factor of $\ovl \Gamma$. Furthermore, since $\gcd(p,e)=1$, given an integer $l$, $\exists \varepsilon\in \C$ such that $\varepsilon^p=\eta$ and $\varepsilon^e=1$. This implies that 
$$
\sum\limits_{j=1}^d \widehat{\sigma}^*_{\pb_0}(A_j) v^{\omega j}(\eta^l \overline{\g}_1)^j=\widehat{\sigma}_{\pb_0}^*(\xi_1)(\varepsilon v^{1/e},w) \in \C\{v^{1/e},w\}.
$$
which shows that $\widehat{\sigma}^*_{\pb_0}(\xi_i)$ is convergent whenever $\overline{\g}_i$ is conjugated to $\overline{\g}_1$ by a root of unity.

Now fix an arbitrary $\xi_i$. By the previous argument, up to renumbering, we may assume that $\overline{\g}_1$ and $\overline{\g}_i$ belong to the same irreducible component of $\mathcal{C}$, where $\widehat{\sigma}^*_{\pb_0}(\xi_1)\in \C\{v^{1/e},w\}$ is convergent, and let $D_i'$ and $D_1'$ be the discs corresponding to the roots $\overline{\g}_1$ and $\overline{\g}_i$.

The fact that $\widehat{\sigma}_{\pb_0}^*(\xi_1)\in \C\{v,w\}$ is equivalent to the existence of $A,B\geqslant 0$ and a compact disc $D$ in the $w$-axis, containing the origin in its interior, such that 
\[
\left\|\frac{P_\ell(w,\ovl\g_1)}{h(1,w)^{\d(\ell)}}\right\|_D\leqslant AB^\ell.
\]
Since $h(1,0)\neq 0$, we can assume that $D$ is such that 
\begin{equation}\label{ineq_izumi}
\exists C\geqslant 1, \forall w\in D, \frac 1C \leqslant |h(1,w)|\leqslant C.\end{equation}
In particular we have
$$
\left\|P_\ell(w,\ovl\g)\right\|_D\leqslant AB^\ell C^{\delta(\ell)}.
$$

Even if it means shrinking $D$, we may suppose furthermore that $D\subset D'$. Denote by $D_1\subset D_1',  \ldots,D_d\subset D_d'$ the connected components of $\pi^{-1}(D)$. Let $D_i$ be the connected component of $\pi^{-1}(D)$ such that $(w,\overline{\g}_i)\in D_i$. Then $\left\|P_\ell(w,\overline{\g}_i)\right\|_D=\left\|P_\ell(w,z)\right\|_{D_i}$. Since $D_1$ and $D_i$ are in the same irreducible component of $\mathcal{C}$, Lemma \ref{lemma:IzumiMajorationPol} states the existence of a constant $M\geqslant 1$ such that 
\[
\|P_n(w,z)\|_{D_i}\leqslant M^{\deg(P_n)} \|P_n(w,z)\|_{D_1}.
\] 
Finally, we conclude that 
$$
\|P_n(w,\overline{\g}_i)\|_D=\|P_n(w,z)\|_{D_i}\leqslant M^{\deg(P_\ell)}AB^\ell C^{\delta(\ell)},
$$ and, by \eqref{ineq_izumi},
$$
\left\|\frac{P_\ell(w,\g_i)}{h(1,w)^{\delta(\ell)}}\right\|_D\leqslant M^{\deg(P_\ell)}AB^\ell C^{2\delta(\ell)}.
$$ 
Therefore, because  $\deg(P_\ell)$ and $\delta(\ell)$ are bounded by affine functions, we conclude that $\widehat{\sigma}(\xi_i)\in \C\{v,w\}$, for any $i$. Since the choice of $\xi_i$ was arbitrary, we conclude that $Q_{\pb_0}\in \C\{v,w\}[y]$. Thanks to the assumption that $\pb_0$ is not on the strict transform of $\{h=0\}$, Lemma \ref{max} implies that $Q_i\in \PP_h\{\x\}[y]$.

Therefore, we can identify convergent factors of $P$ and of $P_\pb$, for every $\pb$ outside of a discrete subset of $F_r^{(1)}$. In other words, if $Q$ is the maximal convergent factor of $P$, then $Q_\pb$ is the maximal convergent factor of $P_\pb$, for every $\pb\in F_r^{(1)}$ outside of a discrete subset of $F_r^{(1)}$. We conclude easily via Proposition \ref{prop:FormalContinuationConvergentFactor}.
\end{proof}


\section{Applications and variations}\label{sec:Aplication}

In this section we prove the results announced in $\S\S$\ref{ssec:Application}.

\subsection{Proof of Theorem \ref{strong}}\label{ssec:ApliStrong}

 Let us assume that $\Gr(\phi)=\Fr(\phi)=\Ar(\phi)$. By replacing $A$ by $\frac{A}{\Ker(\phi)}$ we may assume that $\phi$ is injective. Thus $\Gr(\phi)=\Fr(\phi)=\Ar(\phi)=\dim(A)$. Let $f\in\wdh{\phi}(\wdh{A})\cap B$. We may assume that $f(0)=0$ by replacing $f$ by $f-f(0)$. We define a new morphism $\psi : \frac{\C\{\x,z\}}{I\C\{\x,z\}}\lgw B$, where $A=\frac{\C\{ \x\}}{I}$ and $z$ is a new indeterminate, by
 $$
 \psi(g(\x,z))=g(\phi(x_1),\ldots, \phi(x_n),f)) \text{ for any } g\in\frac{\C\{\x,z\}}{I\C\{ \x,z\}}.
 $$
Since $\psi_{|_A}=\phi$, we have $\Gr(\psi)\geqslant \Gr(\phi)$.  Since $f\in \wdh{\phi}(\wdh{A})$, there exist $h\in \wdh{A}$ such that $\wdh{\phi}(h)=f$. Thus $z-h\in \Ker(\wdh{\psi})$. In fact, by the Weierstrass division Theorem, every element of $\Ker(\wdh\psi)$ is equal to an element of $\wdh A$ modulo $(z-h)$. But, because $\Ker(\wdh{\psi})\cap \wdh A=\Ker(\wdh{\phi})=(0)$, we have that $\Ker(\wdh \psi)=(z-h)$. In particular the injection  $\wdh{A}\lgw \frac{\C\lb \x,z\rb}{I\C\lb  \x,z\rb+\Ker(\wdh{\psi})}$ is an isomorphism, thus $\Fr(\psi)=\Fr(\phi)=\Gr(\phi)$. Finally, since $\Fr(\psi)\geqslant \Gr(\psi)$, we get $\Fr(\psi)=\Gr(\psi)$, hence $\Fr(\psi)=\Ar(\psi)$, by Theorem \ref{thm:main}. Therefore, by Proposition \ref{rk:BasicPropertiesGabrielov1} we have that $\Ker(\psi).\wdh C=\Ker(\wdh\psi)$. Thus, if $f\in \C\{\x,z\}$ is a generator of $\Ker(\psi)$, there is a unit $U(\x,z)$ such that $f=U(\x,z)(z-h)$. But, by the unicity in the Weierstrass preparation Theorem, we have that $u$ and $z-h$ are convergent, hence $h\in A$. This proves that $ \phi$ is strongly injective. 

On the other hand, assume that $\phi$ strongly injective. There exists a finite injective morphism $\C\{\x\}\lgw A$ (by Noether normalization Lemma) and an injective morphism of maximal rank $B\lgw \C\{\y\}$ (by resolution of singularities). Hence, if we denote by $\psi$ the induced morphism $\C\{\x\}\lgw \C\{\y\}$, by Proposition \ref{rk:BasicPropertiesGabrielov2}, $\Gr(\psi)=\Gr(\phi)$ and $\Fr(\psi)=\Fr(\phi)$ and $\psi$ is strongly injective. Therefore we are exactly in the situation of \cite[Theorem 1.2]{EH} that asserts that $\Gr(\psi)=\Fr(\psi)=\Ar(\psi)$.

\subsection{Proof of Theorem \ref{thm:equivalentGabrielov}}\label{ssec:AplVariation}
In what follows, we prove that $(I) \implies (II) \implies (III) \implies (IV) \implies (I)$. The Theorem immediately follows because $(I)$ is Gabrielov's rank Theorem \ref{thm:main}.

\medskip
\noindent
$(I) \implies (II)$ By replacing $A$ by $\frac{A}{\Ker(\phi)}$ we may assume that $\phi$ is injective, thus $\Gr(\phi)=\dim(A)=n$ by Theorem \ref{strong}. By Lemma \ref{lem_1st_red}, there exists a finite injective morphism $\C\{\x\}\lgw A$ and an injective morphism of maximal rank $B\lgw \C\{\y\}$. The induced morphism $\C\{\x\}\lgw \C\{\y\}$  is strongly injective, since $\phi$ is strongly injective, and $f$ is integral over $\C\lb \x\rb$. Thus we may assume that $A=\C\{\x\}$ and $B=\C\{\y\}$.

Let $P(\x,z)\in\C\lb \x\rb[z]$ be the minimal polynomial of $f$ over $\C\lb \x\rb$. By replacing $f$ by $f-f(0)$ we may assume that $f\in( \y)\C\{\y\}$. Consider the morphism $\psi:\C\{\x,z\}\lgw \C\{\y\}$ defined by $$\psi(h( \x,z))=h(\phi(x_1),\cdots,\phi(x_n),f) \text{ for any }h\in \C\{\x,z\}.$$ Then $\Fr(\psi)=\Fr(\phi)=n$ since $P(z)\in \Ker(\wdh{\psi})$. Hence, because $\Fr(\psi)\geq \Gr(\psi)\geq \Gr(\phi)$, we have $\Gr(\psi)=n$. Thus, by $(I)$, $\Ar(\psi)=n$ so $\Ker(\psi)$ is a  height one prime ideal, thus a principal ideal by \cite[Theorem 20.1]{Mat}. Let $Q\in\C\{\x,z\}$ be a generator of $\Ker(\psi)$. Then $Q$ is a generator of $\Ker(\wdh{\psi})$ and $Q$ divides $P$ in $\C\lb \x,z\rb$. Moreover $P(0,z)\neq 0$ since $P( \x,z)$ is a monic polynomial in $z$. Thus $Q(0,z)\neq 0$ and, by the Weierstrass preparation Theorem, there exists a unit $U(\x,z)\in\C\{\x,z\}$ such that $UQ$ is a monic polynomial in $z$. But $UQ\in\Ker(\psi)$, i.e. $(UQ)(\phi(\x),f)=0$, thus $f$ is integral over $\C\{\x\}$.

\medskip
\noindent
$(II) \implies (III)$ First of all, we may assume that $f$ is irreducible. Then we apply the result to each irreducible divisor of $f$.
 Let $A=\C\{\x\}$ and $B=\C\{\x,t\}/(f)$. We have $fg=a_0(\x)+a_1(\x)t+\cdots+a_d(\x)t^d$ where the $a_i$ are in $\C\lb \x\rb$. There is a sequence of quadratic transforms $\t=\s_k\circ\cdots\circ\s_1:\C\{\x\}\lgw \C\{\x\}$ such that $\t(a_d(\x))=\x^\a U(\x)$ for some unit $U(\x)\in\C\lb \x\rb$. Thus
 $$\t(f)\t(g)=\t(a_0(\x))+\t(a_1(\x))t+\cdots+\x^\a U(\x)t^d.$$
Therefore
 \begin{equation}\label{int_dep}\begin{split}
\x^{(d-1)\a}U^{-1}&\t(f)\t(g)=\\
 y^d+U^{-1}&\t(a_{d-1})y^{d-1}+U^{-1}x^\a\t(a_{d-1})y^{d-2}+\cdots+U^{-1}x^{(d-1)\a}\t(a_0)
 \end{split}
 \end{equation}
 where $y:=\x^\a t$. Let $k\in\C\{\x,t\}$ be a prime divisor of $\t(f)$.  Then $y\in\C\{\x,t\}/(k)$ is integral over $\C\lb \x\rb$ by \eqref{int_dep}.
 But the composed morphism 
 $$
 \xymatrix{\C\{\x\}\ar[r] & \C\{\x,t\}/(f) \ar[r]^{\t} & \C\{\x,t\}(k)}
 $$ 
 has maximal rank, thus $y$ is integral over $\C\{ \x\}$ by $(II)$, and $t$ is algebraic over $\C\{\x\}$.

\medskip
\noindent
$(III) \implies (IV)$ Let $P(\x, t)\in\C\lb \x\rb[t]$ be a nonzero  polynomial in $t$ such that $P(\x,f)=0$ mod. $(x_1-x_2z)$. This means that there exists a formal power series $g\in\C\lb \x,z\rb$ such that
 $$P(\x,f(\x,z))+(x_1-x_2z)g(\x,z)=0.$$
 We remark that, because  $f$ is algebraic over $A$,  $f+\la z$ is algebraic over $A$ for any $\la\in\C$ and,  if $\deg_t(P)=d$, the polynomial 
 $$T(\x,t):=x_2^dP(\x,t-\la x_1/x_2)\in\C\lb \x\rb[t]$$
  is a vanishing polynomial of $f+\la z$.
  Let us choose $\la\in\C$ such that if 
  $$h:=f+\la z$$
   then $h(0,z)$ is a nonzero power series of order 1. We define $I:=(t-h(\x,z), x_1-x_2z)$ as an ideal of $\C\{\x,t,z\}$. The ideal $I$ is prime since $\C\{\x,t,z\}/I\simeq \C\{x_2,  \ldots, x_n,z\}$, and $\het(I)=2$ since it is generated by two coprime elements.
   
   By the Weierstrass division Theorem
   $$t-h(\x,z)=U(\x,z,t)(z+h'(\x,t))$$
   for some unit $U(\x,z,t)\in \C\{\x,z,t\}$ and $h'(\x,t)\in \C\{\x,t\}$. Thus 
   $$I=(z+h'(\x,t), x_1-x_2z).$$
   Since $z+h'(\x,t)$ and $x_1-x_2z$ are coprime polynomials, $I_1:=I\cap \C\{\x,t\}\neq (0)$. Moreover $I_1$ is a height one prime ideal so it is principal since $\C\{\x,t\}$ is a unique factorization domain. Let $Q(\x,t)$ denote a generator of $I_1$, $Q(\x,t)=x_1+x_2h'(\x,t)$. By the Weierstrass preparation Theorem we may assume that
   $$Q=x_1+k(x_2,  \ldots, x_n,t)$$
   for some $k\in \C\{x_2,  \ldots, x_n,t\}$. Moreover we can do the change of variables $x_1\lgm x_1+k(x_2,  \ldots, x_n,0)$ and assume that $t$ divides $k(x_2,  \ldots, x_n,t)$.
   
   On the other hand $\wdh I\cap\C\lb \x,t\rb$ is also a height one prime ideal as for $I_1$, and $\wdh I_1\subset \wdh I\cap\C\lb \x,t\rb$. But since $I_1$ is prime, $\wdh I_1$ is also prime and $\wdh I_1= \wdh I\cap\C\lb \x,t\rb$ because both have the same height. Hence $Q$ is a generator of $\wdh I\cap\C\lb \x,t\rb$.
   
Since $T(\x,h)=0$ modulo $(x_1-x_2z)$, $T(\x,t)\in \wdh I\cap\C\lb \x,t\rb$. Thus there is a formal power series $R(\x,t)$ such that
   $$R(\x,t)Q(\x,t)=T(\x,t)\in\C\lb \x\rb[t].$$
But $(III)$ allows us to  assume that $R(\x,t)\in \C\{x,t\}$ and $\wdt T(\x,t):=R(\x,t)Q(\x,t)\in \C\{\x\}[t]$. Therefore, $\wdt T(\x,f+\la x_1/x_2)=0$, and $f$ is algebraic over $\C\{\x\}$. This proves $(IV)$.

\medskip
\noindent
$(IV) \implies (I)$ We follow the beginning of the proof of Theorem \ref{thm:main}: we argue by contradiction and assume that $\phi:\C\{x_1,x_2,x_3\}\lgw \C\{u_1,u_2\}$ satisfies $\Gr(\phi)=\Fr(\phi)=2$ and $\Ar(\phi)=3$. We will replace, step by step, the morphism $\phi$ by another morphism $\phi'$ such that   $\Gr(\phi')=2$ and $ \Ker(\wdh{\phi'})$  is generated by a Weierstrass polynomial in $x_3$ and such that $\phi'$ has a  particularly simple form.

First, we use Lemma \ref{lemma:preparationphi}  to assume that $\phi(x_1)=u_1$ and  $\phi(x_2)=u_1^{\a}u_2^{\b}U( \uu)$ where $U( \uu)$ is a unit in $\C\{\uu\}$. Now let $\s':\C\{\uu\}\lgw \C\{\uu\}$ be defined by $\s'(u_1)=u_1^\b$ and $\s'(u_2)=u_1u_2^{\a+1}$. Then, we have
$$
\s'\circ\phi(x_1)=u_1^\b\text{ and } \s'\circ\phi(x_2)=(u_1u_2)^{\b(\a+1)}V( \uu)
$$
for some unit $V( \uu)$. Therefore, up to ramification, we may assume that
$$
\phi(x_1)=u_1\text{ and } \phi(x_2)=u_1u_2.
$$
Let $P$ be a generator of $\Ker(\wdh \phi)$. If we denote by $f(u_1,u_2)$ the image of $x_3$ by $\phi$, we have
$$
P(u_1,u_1u_2,f(u_1,u_2))=0
$$
or
$$
f^d(u_1,u_2)+a_1(u_1,u_1u_2)f^{d-1}(u_1,u_2)+\cdots+a_d(u_1,u_1u_2)=0
$$
or
$$
f^d(x_1,x_3)+a_1(x_1,x_2)f^{d-1}(x_1,x_3)+\cdots+a_d(x_1,x_2)=0 \text{ modulo } (x_2-x_1x_3).$$
Therefore, by $(IV)$, we may assume that the $a_i$ are in $\C\{x_1,x_2\}$. But this implies that $\Ker(\phi)\neq (0)$ and $\Ar(\phi)<3$ which is a contradiction. This proves $(I)$.

\subsection{Proof of Corollary \ref{strongly_inj2}}\label{ssec:AplStronglyInj2}
As in the proof of $(I) \implies (II)$ in Theorem \ref{thm:equivalentGabrielov}, we can assume that $\phi$ is injective, and $\phi\colon \C\{\x\}\to \C\{\y\}$. Since $f$ is algebraic, there is a nonzero polynomial $P(\x,z)\in\C\lb \x\rb[z]$  in  the kernel of $\wdh\psi$, where $\psi$ is given by
$$
\begin{array}{cccc}\psi : &\C\{ \x,z\}&\lgw &\C\{ \y\} \\
& h(\x,z) &\lgm & h(\phi(\x),f)\end{array}
$$
Once again as in the proof of $(I) \implies (II)$ in Theorem \ref{thm:equivalentGabrielov}, $\Ker(\psi)$ is a nonzero principal ideal. Let $Q\in\Ker(\psi)$. Since $Q\in\Ker(\wdh\psi)$, there is a formal power series $g$ such that $gQ=P\in\C\lb \x\rb[z]$. Then, by Theorem \ref{thm:equivalentGabrielov}$(III)$ there is $h\in \C\{x,z\}$ such that $hQ\in \C\{\x\}[z]$. Moreover $hQ\in \Ker(\psi)$, which proves that $f$ is integral over $\C\{\x\}$.

\subsection{Proof of Theorem \ref{thm:AplMacDonald}}\label{ssec:AplMacDonald}
If $f(x_1^{\frac{1}{d}},  \ldots ,x_n^{\frac{1}{d}})\in \C\{ \s\cap\frac{1}{d}\Z^n\}$ is a root of $P(\x,z)$, then $f(\x)$ is a root of $P(\x^d,z)$. Let $Q(z)$ be the monic irreducible factor of $P(\x^d,z)$ in $\C\lb \x \rb [z]$ having $f(\x)$ as a root. Since $\sigma$ is strongly convex, there exists an invertible linear map $L : \Z^n\lgw \Z^n$ with positive coefficients such that $L(\s)\subset \R_{\geq 0}^n$. Let $(l_{i,j})_{i,j}$ be the matrix of $L$. Then, since $L$ is the product of elementary matrices, the morphism $\t:\C\{\x\}\lgw \C\{\x\}$ defined by $\t(x_i):=x_1^{l_{1,i}}  \cdots x_n^{l_{n,i}}$ for $1\leq i\leq n$ is a composition of the map $\pi$ defined by $\pi(x_1)=x_1x_2$, $\pi(x_i)=x_i$ for $i\geq 2$, and of the maps permuting the variables $x_1$,  \ldots , $x_n$. Since $f(\x)$ is a root of $Q(z)$ then  $\t(f(\x))=f(\t(\x))$ is integral over $\t\left(\C\lb \x\rb\right)$. Thus, by Theorem \ref{thm:equivalentGabrielov}(II), $f(\t(\x))$ is integral over $\t(\C\{\x\})$ and the coefficients of $Q(z)$ are convergent power series.

Then the  factors of $P(\x^d,z)$ are the polynomials $Q(\xi_1x_1,  \ldots ,\xi_nx_n,z)$ for any $d$-th roots of the unity $\xi_1$,   \ldots , $\xi_n$ and they are in $\C\{ \x\}[z]$. Thus the coefficients of $P(\x^d,z)$ are convergent power series and the coefficients of $P( \x,z)$ also.



\section{Abhyankar-Jung Theorem}\label{section:AbhyankarJung}

\begin{thmAY}[{\cite{Jung,Ab2}}]
Let $\K$ be an algebraically closed field of characteristic zero.
Let $P(y)\in\K\lb \x\rb[y]$ be a monic polynomial in $y$ of degree $d$ such that $\D_P=\x^\a u(\x)$ where $u(0)\neq 0$. Then the roots of $P(y)$ are in $\K\lb \x^{1/d!}\rb$.
\end{thmAY}

In the case of holomorphic polynomials, the Theorem admits a very simple proof based on the local monodromy of solutions of the polynomial $P$ (this is in fact the original proof of Jung \cite{Jung}, even if he stated the theorem only for the ring of convergent power series in 2 indeterminates). The formal case is much more involved, since the same geometrical arguments are unavailable. The first proof of the general case is due to Abhyankar \cite{Ab2}. Kiyek and Vicente gave a modern proof of this result \cite{KV} and, recently, Parusi\'nski and the third author have provided a more direct proof reducing the general case to the complex case, via Lefschetz Principle \cite{P-R}. In this section, we provide a new and very short proof of Abhyankar-Jung Theorem, in all of its generality, following the techniques developed in $\S$ \ref{ssec:SemiGlobalFormalExtention} (and based on the Lefschetz Principle).

\begin{proof}[Proof of the Abhyankar-Jung Theorem] We will prove this result in three steps: first the case where $P(y)\in\C\{\x\}[y]$, then the case where $P(y)\in\C\lb \x\rb[y]$, and finally the general case.

\medskip
\noindent
\emph{Step I:} Assume that $P(y)\in\C\{\x\}[y]$. Let $\e>0$ be such that the coefficients of $P(y)$ are analytic on an open neighbourhood of the closure of $D_\e^n:=\{\x\in\C^n\mid |x_i|<\e, \ \forall i\}$. The projection map $\phi:V:=\{(\x,y)\in\C^{n+1}\mid P(\x,y)=0\}\lgw \C^n\times\{0\}$ is a branched covering, and its restriction over $U:=D_\e^n\backslash\{x_1\cdots x_n=0\}$ is a finite covering of degree $d$.

Let $\pa\in U$. We have $\pi_1(U,\pa)=\Z^n$. 
{Therefore 
$$(d!\Z)^n\subset \phi_*(\pi_1(V\cap\phi^{-1}(U), \pb)).$$
But the map $$
\rho:\x\in(D_{\e^{1/d!}}^*)^n\lgm \x^{d!}
$$
satisfies $\rho^*(\pi((D_{\e^{1/d!}}^*)^n))=(d!\Z)^n$.
Thus by pulling back the covering by $\rho$,
the cover has no monodromy, thus $P(x^{d!},y)$ factors into linear factors.
The polynomial $P(\x^{d!},y)$ being monic, its roots $\xi_i(\x)$ are bounded near the origin. Therefore, by Riemann Removable Singularity Theorem, they extend to analytic functions in a neighborhood of the origin.}

\medskip
\noindent
\emph{Step II:} Now suppose that $P(y)\in\C\lb \x\rb[y]$, and $\D_P=\x^\a u(\x)$. We denote by $\ovl \K$ an algebraic closure of the field $\C(\!(\x)\!)$. The valuation $\ord$ on $\C(\!(\x)\!)$ extends on $\ovl \K$, and this extension is again denoted by $\ord$. Denote by $\xi_1,  \ldots,\xi_d$ the roots of $P(y)$ in $\ovl \K$.

If $P(y)=y^d+a_1(\x)y^{d-1}+\cdots+a_d(\x)$, there is a universal polynomial $\D=\D(A_1,\ldots, A_d)\in\Q[A_1,  \ldots, A_d]$, such that
$\D_P=\D(a_1,  \ldots, a_d)$.

Now we apply Artin approximation Theorem: for every integer $c$, there are $a_{c,1}(\x)$,   \ldots, $a_{c,d}(\x)$, $u_c(\x)\in\C\{\x\}$ such that
$$\D(a_{c,1},  \ldots, a_{c,d})=\x^\a u_c(\x)$$
and $a_i-a_{c,i}$, $u-u_c\in (\x)^c$. In particular, for $c\geq 1$, we have $u_c(0)\neq 0$. We set $P_c(y)=y^d+\sum_{i=1}^da_{c,i}(\x) y^{d-i}$. By Step I, there exist $\xi_{c,1}$,   \ldots, $\xi_{c,d}\in\C\{\x^{1/d!}\}$ such that
$$
P_c(y)=\prod_{i=1}^d(y-\xi_{c,i}).
$$
By Lemma \ref{key_app}, after renumbering the $\xi_{c,i}$ we may assume that $\ord(\xi_{c,i}-\xi_c)$ goes to infinity. Therefore $\xi_c\in\C\lb \x^{1/d!}\rb$, and the result is proved.

\medskip
\noindent
\emph{Step III:} Finally, in the general case, we denote  by $\K_0$ the subfield of $\K$ generated by all the coefficients of the series defining $P(y)$. Such a field $\K_0$ can be embedded in $\C$, because $\Q\lgw \K_0$ is a field extension of finite or countable degree, while the degree of $\Q\lgw \C$ is uncountable and $\C$ is algebraically closed. We denote by $\iota$ this embedding. Then, by the previous case, the roots of $P(y)$ are in $\C\lb \x^{1/d!}\rb$. Let $\K_1$ be the subfield of $\C$ generated by $\iota(\K_0)$ and the coefficients of the roots of $P(y)$ in $\C\lb \x^{1/d!}\rb$. Then, there is an embedding of $\K_1$ in $\K$ whose restriction to $\iota(\K_0)$ is $\iota^{-1}$, because the degree of $\iota(\K_0)\lgw \K_1$ is 0 (by Lemma \ref{finite} given below),  and $\K$ is algebraically closed. 
\end{proof}

\begin{lemma}\label{finite}
Let $\K\subset \LL$ be two  fields. Let $f\in\LL\lb \x\rb$ be algebraic over $\K\lb \x\rb$. Let $\K_1$ be the field extension of $\K$ generated by the coefficients of $f$. Then $\K\lgw \K_1$ is an algebraic field extension.
\end{lemma}

\begin{proof}
This result is well-known (see for instance \cite{G}). The proof goes as follows: we fix a monomial order on $\LL\lb \x\rb$ such that every subset of $\N^n$ has a minimal element. For $f:=\sum_{\a\in\N^n}f_\a\x^\a\in\LL\lb \x\rb$, we set
$$\exp(f):=\min\{\a\in\N^n\mid f_\a\neq 0\} \text{ and } \ini(f):=f_{\exp(f)}\x^{\exp(f)}.$$
Assume that $P(\x,f(\x))=0$ where $P(\x,y)\in\K\lb \x\rb[y]$ is nonzero. From this relation we obtain that $\ini(f)$ is algebraic over $\K\lb \x\rb$, therefore $f_{\exp(f)}$ is integral over $\K$. Then we replace $f$ by $f-\ini(f)$ and we replace $\K$ by $\K(f_{\exp(f)})$. The result follows by induction.
\end{proof}

\end{document}